\documentclass[11pt,a4paper,twosides]{amsart}
\usepackage{amssymb,amsmath,amsthm,graphicx,color}
\usepackage{amsrefs}
\usepackage{leftidx}
\theoremstyle{plain}
\newtheorem{theorem}{Theorem}[section]

\newtheorem{lemma}[theorem]{Lemma}
\newtheorem{corollary}[theorem]{Corollary}
\newtheorem{proposition}[theorem]{Proposition}
\newtheorem{assumption}[theorem]{Assumption}
\theoremstyle{remark}
\newtheorem{remark}[theorem]{Remark}

\newtheorem{example}[theorem]{Example}

\allowdisplaybreaks[4]

\usepackage{amsrefs}

\numberwithin{equation}{section}

\newcommand{\C}{\mathbb{C}}
\newcommand{\R}{\mathbb{R}}
\newcommand{\Z}{\mathbb{Z}}

\renewcommand{\Im}{\operatorname{Im}}
\renewcommand{\Re}{\operatorname{Re}}
\newcommand{\I}{\infty}

\newcommand{\norm}[1]{\left\lVert #1\right\rVert}

\newcommand{\tnorm}[1]{\lVert #1\rVert}

\newcommand{\IN}{\quad\text{in }}

\def\({\left(}
\def\){\right)}
\def\<{\left\langle}
\def\>{\right\rangle}
\def\le{\leqslant}
\def\ge{\geqslant}

\def \l{\lambda}

\def \pa{\partial}

\def \a{\alpha}

\newcommand{\eps}{\varepsilon}

\newcommand{\la}{\lambda}

\newcommand{\pt}{\partial}

\DeclareMathOperator{\sign}{sign}
\DeclareMathOperator{\supp}{supp}

\DeclareMathOperator{\rank}{rank}
\DeclareMathOperator{\tr}{tr}
\DeclareMathOperator{\Span}{span}

\newcommand{\todayd}{\the\year/\the\month/\the\day}

\theoremstyle{definition}

\newcommand{\ol}{\overline}

\begin{document}
\title[cubic nonlinear Klein-Gordon system]
{On asymptotic behavior of solutions to cubic nonlinear Klein-Gordon systems in one space dimension}

\author{Satoshi Masaki}
\address{Department of systems innovation, 
Graduate school of Engineering Science, 
Osaka University, Toyonaka Osaka, 560-8531, Japan}
\email{masaki@sigmath.es.osaka-u.ac.jp}

\author{Jun-ichi Segata}
\address{Faculty of Mathematics, Kyushu University, 
Fukuoka, 819-0395, Japan}
\email{segata@math.kyushu-u.ac.jp}

\author{Kota Uriya}
\address{Department of Applied Mathematics, Faculty of Science, 
Okayama University of Science, Okayama, 700-0005, Japan}
\email{uriya@xmath.ous.ac.jp}


\begin{abstract}
In this paper, we consider the large time asymptotic behavior of solutions to systems of two cubic nonlinear Klein-Gordon equations in one space dimension.
We classify the systems by studying the quotient set of a suitable subset of systems by the equivalence relation naturally induced by the linear transformation of the unknowns.
It is revealed that the equivalence relation is well described by an identification with a matrix.
In particular, we characterize some known systems in terms of the matrix and specify all systems equivalent to them.
An explicit reduction procedure from a given system in the suitable subset to a model system, i.e., to a representative, is also established.
The classification also draws our attention to
 some model systems which admit solutions with a new kind of asymptotic behavior. Especially, we find new systems which admit a solution of which decay rate is worse than that of a solution to the linear Klein-Gordon equation by logarithmic order.
\end{abstract}

\maketitle

\section{Introduction}
The paper is devoted to the study of the large time behavior of small solutions to the following system of the cubic nonlinear Klein-Gordon equations
in one space dimension:
\begin{equation}\label{E:sys}
\left\{
\begin{aligned}
&(\square + 1)u_1 
= \lambda_1u_1^3 + \l_2u_1^2u_2 + \l_3u_1u_2^2 + \l_4u_2^3, \\
&(\square + 1)u_2 
= \lambda_5u_1^3 + \l_6u_1^2u_2 + \l_7u_1u_2^2 + \l_8u_2^3, 
\end{aligned}
\right.
\end{equation}
where $u_j: \R\times \R \to \R$ ($j=1,2$) are real-valued unknowns, $\square = \partial_t^2 - \partial_x^2$ is
the d'Alembertian, $\l_1,\dots, \l_8$  are real constants.
We consider the system with the initial condition
\begin{equation}\label{E:IC}
	(u_1(0),\partial_t u_1(0),u_2(0),\partial_t u_2(0)) = (\eps u_{1,0}, \eps u_{1,1},\eps u_{2,0},\eps u_{2,1}),
\end{equation}
where $u_{1,0},u_{1,1},u_{2,0},u_{2,1}$ are given data, and $\eps>0$ is a small parameter. 
The aim of this paper is to investigate the large time asymptotic behavior of solutions to the system \eqref{E:sys} with \eqref{E:IC}.
It is known that the cubic nonlinearities are critical in one dimension with respect to asymptotic behavior 
of the solution (\cites{GeoL,GY,Gl,HN08JDE,HN09CCM,Kl,Matsu,Sh}).
We expect that a combination of the coefficients of the nonlinearities in \eqref{E:sys}
decides the type of the large time behavior of the solutions to \eqref{E:sys}.
We are interested in classifying the system \eqref{E:sys} from the view point of the large time behavior 
of the solution.

\smallbreak

There exist several particular combinations of the coefficients with which the large time asymptotic behavior of solutions to \eqref{E:sys} is known.
When $(\l_1, \l_8) \neq (0,0)$ and $\l_2 = \cdots = \l_7 = 0$, we have the decoupled system  
\begin{equation}\label{E:sysdecouple}
\left\{
\begin{aligned}
&(\square + 1)u_1 
= \lambda_1u_1^3 , \\
&(\square + 1)u_2 
= \l_8u_2^3.
\end{aligned}
\right.
\end{equation}
In this case, the system is merely an alignment of two independent single equations.
Hence, we are able to apply the result due to Delort~\cite{Del} for the single real-valued cubic nonlinear Klein-Gordon equation
\begin{equation}\label{E:cNLKG}
	(\square + 1) u = \pm u^3
\end{equation}
and we see that there exist functions 
$\Phi_1$, $\Phi_2$, $\Psi_1=\Psi_1(\Phi_1)$, and $\Psi_2=\Psi_2(\Phi_2)$ such that
\[
	u_j(t) = \frac{2}{t^{1/2}}\Re\left[
	\Phi_j\left(\frac{x}{t}\right)
	\exp\left(i\sqrt{t^2-|x|^2}+i\Psi_j\left(\frac{x}{t}\right)\log t\right)
	\right] + O( t^{-\frac{3}{2}+\delta})
\]
for $j= 1,2$ as $t \to \infty$ provided $\eps \lesssim_\delta 1$, where $\delta>0$ is an arbitrary small constant
(see also \cite{HN08ZAMP,LS,Sti}).
It is worth mentioning that the above asymptotic behavior contains 
a logarithmic phase correction due to the long-range effect of the cubic nonlinearity.
The appearance of logarithmic phase correction is a typical phenomenon in the critical dispersive equations (see, e.g., \cite{Oz}).

In the case where $\l_2=\l_4=\l_5=\l_7=0$ and $\l_1=\l_3=\l_6=\l_8=\l\neq 0$, 
we have the system
\begin{equation}\label{E:syscomplex}
\left\{
\begin{aligned}
&(\square + 1)u_1 
= \lambda (u_1^2 + u_2^2)u_1 , \\
&(\square + 1)u_2 
= \l (u_1^2 + u_2^2)u_2.
\end{aligned}
\right.
\end{equation}
This is another known case where the asymptotic behavior of the solution involves logarithmic phase correction. 
One sees that $w = u_1 + i u_2$ satisfies a single complex-valued cubic
nonlinear Klein-Gordon equation
\begin{equation}\label{E:gcNLKG}
	(\square + 1) w = \l |w|^2 w.
\end{equation}
The second author~\cite{S} shows that there exist functions 
$\Phi_+, \Phi_-$, and $ \Psi_\pm = \Psi_\pm (\Phi_+,\Phi_-)$ such that
\begin{align*}
	u(t) &= \frac{2}{t^{1/2}}\left[
	\Phi_+\left(\frac{x}{t}\right)
	\exp\left(i\sqrt{t^2-|x|^2}+i\Psi_+\left(\frac{x}{t}\right)\log t\right)
	\right] \\
	&\quad+ \frac{2}{t^{1/2}}\left[
	\Phi_-\left(\frac{x}{t}\right)
	\exp\left(-i\sqrt{t^2-|x|^2}+i\Psi_-\left(\frac{x}{t}\right)\log t\right)
	\right]
	+ O( t^{-\frac{3}{2}+\delta})
\end{align*}
as $t \to \infty$. 
A difference between this behavior and that for \eqref{E:sysdecouple} is that the phase correction term $\Psi_\pm$ depends on both $\Phi_+$ and $\Phi_-$,
which reflects the presence of an interaction between two components of the system.
See also \cite{MS2,MS3,MSU} for the two- and three-dimensional gauge invariant critical equation of the form \eqref{E:gcNLKG}.

There exist other kinds of systems.
One such example is the case $\l_j = \delta_{j5}$, where $\delta_{jk}$ is the Kronecker delta:
\begin{equation}\label{E:sysSunagawa}
\left\{
\begin{aligned}
&(\square + 1)u_1 
= 0, \\
&(\square + 1)u_2 
= u_1^3.
\end{aligned}
\right.
\end{equation}
As for the system,
Sunagawa~\cite{Su3} showed that there exist functions 
$\Phi_1, \Phi_2, \tilde{\Psi}=\tilde{\Psi} (\Phi_1)$ such that
\begin{align*} 
\notag u_1(t,x) =& \frac{2}{t^{1/2}}
\Re\Bigg[\Phi_1\left(\frac{x}{t}\right)
\exp\left(i\sqrt{t^2-|x|^2}
\right)
\Bigg]+ O\left( t^{-\frac{3}{2}+\delta }\right),\\
\notag u_2(t,x) =&
\frac{2}{t^{1/2}}
\Re\Bigg[\left\{\tilde{\Psi}\left(\frac{x}{t}\right)\log \sqrt{t^2-|x|^2}+\Phi_2\left(\frac{x}{t}\right)\right\}
\exp\left(i\sqrt{t^2-|x|^2}
\right)
\Bigg]+ O\left( t^{-\frac{3}{2}+\delta }\right)
\end{align*}
as $t \to \infty$. 
Remark that 
the time decay rate of $u_2$ is slow due to the appearance of a \emph{logarithmic amplitude correction}.

In \eqref{E:sys}, the two equation has the same linear part. 
It is known that 
if the linear parts have different masses (i.e., if they are $\square + m_j^2$ $(j=1,2)$, respectively, with $m_1 \neq m_2$), then 
certain resonance phenomenon may appear. The phenomenon is called a \emph{mass resonance phenomenon}.  
We refer the readers to \cite{Su03,Tsu,KS,KOS} for studies in this direction. 
When the nonlinearities involve derivatives of unknowns, 
a dissipative effect appears under some structural condition on the nonlinearity. 
See \cite{KimS,Su2} for the precise condition. 
In \cite{MSug}, the first author and Sugiyama showed that
the dissipative decay rate is optimal in a class of systems.
In particular, for any choice of coefficients of \eqref{E:sys},
there is no nontrivial classical solution which satisfies
\[
	\sum_{j=1}^2 \(\norm{u_j(t)}_{L^\I_x(\R)} + \norm{\pt_x u_j(t)}_{L^\I_x(\R)} + \norm{\pt_t u_j(t)}_{L^\I_x(\R)}\) = o(t^{-\frac12} (\log t)^{-\frac12})
\]
as $t\to\I$.

\subsection{Classification of systems}

\subsubsection{The equivalence relation on systems induced by the linear transformation}
The system \eqref{E:sys} is closed under the linear transformation of unknowns. 
Namely,
if we define a new unknown $(v_1,v_2)$ by the linear transformation
\begin{equation}\label{E:linearchange}
	\begin{pmatrix} v_1 \\ v_2 \end{pmatrix} = M\begin{pmatrix} u_1 \\ u_2 \end{pmatrix}, \quad M \in GL_2(\R)
\end{equation}
then
$(v_1,v_2)$ also solves a system of the form \eqref{E:sys},
where $M$ is an invertible $2\times2$ matrix. 

Thanks to this fact, the above known results are applicable to several more systems.
For instance, the system
\[
	\left\{
\begin{aligned}
&(\square + 1)u_1 
= (u_1^2+ 3u_2^2)u_1 , \\
&(\square + 1)u_2 
= (3u_1^2 +u_2^2)u_2
\end{aligned}
\right.
\]
is reduced to \eqref{E:sysdecouple}
by letting $v_1=u_1+u_2$ and $v_2=u_1-u_2$.
Since we know the large time behavior of $(v_1,v_2)$, that of $(u_1,u_2)$ is easily deduced 
by the inverse transform $u_1=\frac12 (v_1+v_2)$ and $u_2=\frac12 (u_1-u_2)$.
In this manner, once we know the behavior of solutions for a system, it means we have it for all systems which  are obtained by a change of variable of the form \eqref{E:linearchange}.

The linear transformation of unknowns naturally induces an equivalence relation on the set of systems of the form \eqref{E:sys}.
By looking at the equivalence classes of systems \eqref{E:sys}, we obtain a classification of the systems of the form \eqref{E:sys}.
In the paper, we classify the systems in this way.

\subsubsection{Identification of a system with a matrix}

A novelty of the present paper is to identify the system \eqref{E:sys} with a matrix.
For this purpose, we introduce notation.
Let $M_N(\R)$ be the set of $N\times N$ matrices with real entries.
Let $GL_N(\R)=\{ M \in M_N(\R) \ |\ \det M\neq0\}$ and $SL_N(\R)=\{ M \in M_N(\R) \ |\ \det M=1\}$
be the general linear group and the special linear group, respectively.
For a matrix or a vector $M$, $\ltrans{M}$ denotes its transpose.
A row vector $\begin{pmatrix}a_1 & a_2 & \cdots & a_N \end{pmatrix} \in \R^N$ is frequently denoted as $(a_1,a_2,\dots,a_N)$, for simplicity.

Let us introduce the matrix
\begin{equation}\label{E:defA}
A := 
\begin{pmatrix}
\la_2 & \la_6 - 3\la_1 & -3\la_5 \\
\la_3 & \la_7 - \la_2 & -\la_6\\
3\la_4 & 3\la_8-\la_3 & -\la_7
\end{pmatrix}
\end{equation}
defined by the coefficients of the nonlinearity. 
Note that the matrix $A$ always satisfies $\tr A = 0$. 
Let us 
define a subspace $Z$ of $M_3(\R)$ by
\begin{equation}\label{E:defZ}
	Z := \{ A \in M_3(\R) \ | \ \tr A = 0 \}.
\end{equation}
The map $(\l_1, \dots, \l_8) \mapsto A$ defined in \eqref{E:defA} is a bijection from $\R^8$ to $Z$
and we are able to identify a combination of the coefficients of the nonlinearity, or simply the system \eqref{E:sys} itself, with an element of $Z$ by the bijection.
In the sequel, we freely use the identification.

An equivalence relation on $Z$ is naturally induced by
the above equivalence relation on systems through the identification.
Namely, we say the relation $Z \ni A \sim A' \in Z$ holds
if there exists 
matrix $M\in GL_2(\R)$ such that $A$ is the system for $(u_1,u_2)$ and $A'$ is the system for $(v_1,v_2)$
defined by \eqref{E:linearchange} with $M$. 
The relation $\sim$ defines a quotient set $Z/{\sim}$.
For $A\in Z$, $[A]:=\{ A' \in Z \ |\ A' \sim A \} \subset Z$ stands for the equivalence class of $A$ in $Z$. 
The quotient set is the set of all equivalence classes, i.e, $Z/{\sim}:=\{ [A] \subset Z \ |\ A \in Z \} $.

One of the benefit of the matrix representation of a system is that 
we have a clear formula of the equivalence relation.
\begin{theorem}\label{T:classA}
For an invertible matrix 
\[
	M= \begin{pmatrix} a & b \\ c & d \end{pmatrix} \in {GL}_2 (\R),
\]
we denote
\begin{equation}\label{E:DMform}
	D= D(M) := \frac1{\det M}
	\begin{pmatrix} d^2 & -2dc & c^2 \\ -bd & ad+ bc & -ac \\ b^2 & -2ab & a^2 \end{pmatrix}
	\in {SL}_3 (\R).
\end{equation}
If $A'\in Z$ is the matrix obtained from $A \in Z$ by the transform \eqref{E:linearchange}
with $M$, then $A'$ is written as
\begin{equation}\label{E:Apform}
	A' = 	(\det {M})^{-1}   D
	A
	 D^{-1}.
\end{equation}
In particular, the equivalence class $[A] \in Z/{\sim}$ is characterized as
\[
	[A] = \left\{
	(\det {M})^{-1}   D
	A
	 D^{-1}
	\ \middle| \ 
	M \in GL_2 (\R)
	\right\}.
\]
Further, $\rank A$ is an invariant quantity under the equivalence relation, that is,
$\rank A' = \rank A$ holds for any $A' \in [A]$.
\end{theorem}


\subsubsection{Corresponding ODE system and the rank of the matrix $A$}

Theorem \ref{T:classA} and the consequent classification results below 
are true not only for
the system \eqref{E:sys} but also for any system which has the same structure with respect to the change of variable of the form \eqref{E:linearchange}.
More precisely, our results hold true for a system, or any mathematical object, if there exists a one-to-one correspondence between an element of $Z$
and if the actions of the following three specific matrices are the same as in \eqref{E:Apform}:
\begin{itemize}
\begin{subequations}
\item \emph{scalar multiplication;}
\begin{equation}\label{E:M1}
	M = \begin{pmatrix}p & 0 \\ 0 & q \end{pmatrix},
	\quad{i.e.,} \quad
	v_1 = p u_1, \quad v_2 = q u_2,
\end{equation}
where $p,q \in \R\setminus\{0\}$ are constants.
\item \emph{swapping variables;}
\begin{equation}\label{E:M2}
	M = \begin{pmatrix} 0 & 1 \\ 1 & 0 \end{pmatrix},
	\quad{i.e.,} \quad
	v_1 = u_2, \quad v_2 = u_1.
\end{equation}
\item \emph{row addition;}
\begin{equation}\label{E:M3}
	M = \begin{pmatrix} 1 & 1 \\ 0 & 1 \end{pmatrix},
	\quad{i.e.,} \quad
	v_1 = u_1 + u_2, \quad v_2 = u_2.
\end{equation}
\end{subequations}
\end{itemize}
This is because a general invertible linear transformation \eqref{E:linearchange}
is given as a finite composition of the above three changes.
For instance, a system of equations
\begin{equation}\label{E:gensys}
\left\{
\begin{aligned}
&\mathcal{L} u_1 
= \lambda_1u_1^3 + \l_2u_1^2u_2 + \l_3u_1u_2^2 + \l_4u_2^3, \\
&\mathcal{L} u_2 
= \lambda_5u_1^3 + \l_6u_1^2u_2 + \l_7u_1u_2^2 + \l_8u_2^3
\end{aligned}
\right.
\end{equation}
is such an example,
where $\mathcal{L}$ is an arbitrary linear operator.

Another example is the following ODE system:
\begin{equation}\label{E:limitODE}
\left\{
\begin{aligned}
	&2i\frac{\partial \alpha_1}{\partial s}=
	3\l_{1} |\a_1|^2\a_1+\l_{2} (2|\a_1|^2\a_2+\a_1^2\ol{\a_2})+
\l_{3}(2\a_1|\a_2|^2 +\ol{\a_1}\a_2^2)+3\l_{4}|\a_2|^2\a_2,\\
	&2i\frac{\partial \alpha_2}{\partial s}=
	3\l_{5} |\a_1|^2\a_1+\l_{6} (2|\a_1|^2\a_2+\a_1^2\ol{\a_2})+
\l_{7}(2\a_1|\a_2|^2 +\ol{\a_1}\a_2^2)+3\l_{8}|\a_2|^2\a_2.
\end{aligned}
\right.
\end{equation}
The system can be regarded as an element of $Z$ by the formula \eqref{E:defA},
and one easily verifies that the changes of the form \eqref{E:M1}--\eqref{E:M3} cause the same transform of the coefficients as for \eqref{E:sys}.
As  seen below, the system \eqref{E:limitODE}, which we refer to as \emph{a limit ODE system} in the paper, appears in the study of the large time asymptotic behavior
of solutions to the system \eqref{E:sys} with the same combination of the coefficients $\{\l_j\}_{j=1}^8$.
In view of the correspondence, it is quite reasonable that \eqref{E:limitODE} has the same structure as \eqref{E:sys} has.

In fact, our classification results is inspired by the analysis of \eqref{E:limitODE}.
It will turn out that the system has a
conserved quantity of the form
\[
	a |\alpha_1|^2 + 2 b\Re (\ol{\alpha}_1 \alpha_2 )+ c |\alpha_2|^2, \quad a,b,c \in \R
\]
if and only if $(a,b,c)$ lies in the kernel of $A$
 (see Subsection \ref{subsec:conserved_quantity}).
One immediate consequence is that
the rank of the matrix $A$ is an important quantity to identify the character of the system \eqref{E:limitODE}.
Further,
in consideration of the correspondence between \eqref{E:limitODE} and \eqref{E:sys},
it is natural to expect that
the larger $\rank A$ is, the more complicated the behavior of the solutions to \eqref{E:sys} is.

%


\subsubsection{Classification theorems}
In this paper, we focus our attention to the case $\rank A = 1$ or $2$. More precisely, we work with the following two subsets of $Z$:
The first one is
\[
	Z_1 := \{A\in Z \ |\ \rank A = 1\}.
\]
The second one is
\[
	Z_2 := \{  A \in Z \ |\ A \text{ satisfies }\eqref{E:main3cond}
 	\} 
	\subsetneq \{A\in Z \ |\ \rank A = 2\},
\]
where
\begin{equation}\label{E:main3cond}
	(\l_1,\l_2,\l_3) \neq (0,0,0), \quad \l_4=\l_5=0 ,\quad \l_6 = \l_1,\quad \l_7=\l_2,\quad \l_8=\l_3 .
\end{equation}
Remark that the systems \eqref{E:sysdecouple} and \eqref{E:sysSunagawa}
correspond to the matrices
\[
	\begin{pmatrix}
	0 & -3\l_1 & 0 \\
	0 & 0 & 0 \\
	0 & 3\l_8 & 0 
	\end{pmatrix} \in Z_1, \qquad
	\begin{pmatrix}
	0 & 0 & -3 \\
	0 & 0 & 0 \\
	0 & 0 & 0 
	\end{pmatrix} \in Z_1,
\]
respectively. Similarly, the system \eqref{E:syscomplex} is written as
\[
	\l
	\begin{pmatrix}
	0 & -2 & 0 \\
	1 & 0 & -1 \\
	0 & 2 & 0 
	\end{pmatrix} \in Z_2 \qquad(\l\neq0).
\]


We say a subset $\tilde{Z}$ of $Z$ is \emph{invariant} if $\tilde{Z} \ni A \sim A' \in Z$ implies $A' \in 
\tilde{Z}$. 
Our classification results are as follows:
\begin{theorem}[Classification of $Z_1$]\label{T:main1.5}
$Z_1$ is an invariant subset of $Z$.
One has
\begin{equation}\label{E:Z1roster}
\begin{aligned}
	Z_1/{\sim} ={}&
	\left\{ 
	\left[\begin{pmatrix}
	0 & -3 & 0 \\
	0 & 0 & 0 \\
	0 & 3 & 0 
	\end{pmatrix}\right],
	\left[\begin{pmatrix}
	0 & 3 & 0 \\
	0 & 0 & 0 \\
	0 & -3 & 0 
	\end{pmatrix}\right],
	\left[\begin{pmatrix}
	0 & -3 & 0 \\
	0 & 0 & 0 \\
	0 & -3 & 0 
	\end{pmatrix}\right],
	\right. \\
	&\quad \left. 
	\left[\begin{pmatrix}
	0 & -3 & 0 \\
	0 & 0 & 0 \\
	0 & 0 & 0 
	\end{pmatrix}\right],
	\left[\begin{pmatrix}
	0 & 3 & 0 \\
	0 & 0 & 0 \\
	0 & 0 & 0 
	\end{pmatrix}\right],
	\left[\begin{pmatrix}
	0 & 0 & 0 \\
	0 & 0 & -3 \\
	0 & 0 & 0 
	\end{pmatrix}\right],
		 \right. \\
	& \quad\left. 
	\left[\begin{pmatrix}
	0 & 0 & 0 \\
	0 & 0 & 3 \\
	0 & 0 & 0 
	\end{pmatrix}\right],
	\left[\begin{pmatrix}
	0 & 0 & -3 \\
	0 & 0 & 0 \\
	0 & 0 & 0 
	\end{pmatrix}\right],
	\left[\begin{pmatrix}
	0 & 0 & 0 \\
	-3 & 0 & -3 \\
	0 & 0 & 0 
	\end{pmatrix}\right]
	 \right\}.
\end{aligned}
\end{equation}
Further, $\# (Z_1/{\sim})=9$, i.e., the classes in the above roster notation are mutually disjoint 
\end{theorem}

\begin{theorem}[Classification of $Z_2$]\label{T:main3}
$Z_2$ is an invariant subset of $Z$.
One has
\begin{equation}\label{E:Z2roster}
\begin{aligned}
	Z_2/{\sim} ={}&
	\left\{ 
	\left[\begin{pmatrix}
	0 & -2 & 0 \\
	1 & 0 & -1 \\
	0 & 2 & 0 
	\end{pmatrix}\right],
	\left[\begin{pmatrix}
	0 & 2 & 0 \\
	-1 & 0 & 1 \\
	0 & -2 & 0 
	\end{pmatrix}\right],
	\left[\begin{pmatrix}
	0 & -2 & 0 \\
	0 & 0 & -1 \\
	0 & 0 & 0 
	\end{pmatrix}\right],
	\right. \\
	&\quad \left. 
	\left[\begin{pmatrix}
	0 & 2 & 0 \\
	0 & 0 & 1 \\
	0 & 0 & 0 
	\end{pmatrix}\right],
	\left[\begin{pmatrix}
	0 & -2 & 0 \\
	-1 & 0 & -1 \\
	0 & -2 & 0 
	\end{pmatrix}\right]
	 \right\}.
\end{aligned}
\end{equation}
Further, $\# (Z_2/{\sim})=5$, i.e., the classes in the above roster notation are mutually disjoint 
\end{theorem}

\subsubsection{Four new model systems}
Theorems \ref{T:main1.5} and \ref{T:main3}
draw our attention to the following four systems which do not seem to be studied before:
The first two are
\begin{equation}\label{E:sysnewa}
\left\{
\begin{aligned}
&(\square + 1)u_1 
= \l u_1^3, \\
&(\square + 1)u_2 
= 3\l u_1^2u_2,
\end{aligned}
\right.
\qquad
A= 
\begin{pmatrix}
0 & 0 & 0 \\
0 & 0 & -3\l \\
0 & 0 & 0 
\end{pmatrix}
\end{equation}
and
\begin{equation}\label{E:sysnewb}
\left\{
\begin{aligned}
&(\square + 1)u_1 
= \l u_1^3, \\
&(\square + 1)u_2 
= \l u_1^2u_2,
\end{aligned}
\right.
\qquad
A= 
\begin{pmatrix}
0 & -2\l & 0 \\
0 & 0 & -\l \\
0 & 0 & 0 
\end{pmatrix},
\end{equation}
where $\l \in \{\pm 1\}$.
The other two are
\begin{equation}\label{E:sysnew2}
\left\{
\begin{aligned}
&(\square + 1)u_1 
= (u_1^2-3u_2^2)u_1 , \\
&(\square + 1)u_2 
= (3u_1^2-u_2^2)u_2,
\end{aligned}
\right.
\qquad
A= 
\begin{pmatrix}
0 & 0 & 0 \\
-3 & 0 & -3 \\
0 & 0 & 0 
\end{pmatrix}
\end{equation}
and
\begin{equation}\label{E:sysnew3}
\left\{
\begin{aligned}
&(\square + 1)u_1 
= (u_1^2-u_2^2)u_1, \\
&(\square + 1)u_2 
= (u_1^2-u_2^2)u_2,
\end{aligned}
\right.
\qquad
A= 
\begin{pmatrix}
0 & -2 & 0 \\
-1 & 0 & -1 \\
0 & -2 & 0 
\end{pmatrix}.
\end{equation}

With theses four systems, we have all \emph{model systems}, representatives, of $Z_1/{\sim}$
and $Z_2/{\sim}$.
Theorem \ref{T:main1.5} is then formally summarized as
\[
	Z_1/{\sim} = \{ [\eqref{E:sysdecouple}],\, [\eqref{E:sysnewa}], \, [\eqref{E:sysSunagawa}], \, [\eqref{E:sysnew2}]  \}.
\]
There is an abuse of notation: $[\eqref{E:sysdecouple}]$ stands for the first five equivalent classes and $[\eqref{E:sysnewa}]$ for the next two in the roster notation in \eqref{E:Z1roster}.
Similarly, Theorem \ref{T:main3} is formally summarized as
\[
	Z_2/{\sim} = \{ [\eqref{E:syscomplex}],\, [\eqref{E:sysnewb}], \, [\eqref{E:sysnew3}]  \}
\]
with an abuse of notation: $[\eqref{E:syscomplex}]$ stands for the first two equivalence classes and $[\eqref{E:sysnewb}]$ for the next two in the roster notation in \eqref{E:Z2roster}.

\subsection{Reduction procedure}

From the view point of the application, it is important to have a procedure to transform
a given matrix $A$ in $Z_1$ or $Z_2$ into the corresponding model system.
We introduce notation.
For a real number $x$, $\sign x$ stands for its sign:
$\sign x = 1$ if $x>0$, $\sign x = 0$ if $x=0$, and $\sign x = -1$ if $x<0$.
For sets $A$ and $B$, we use $A \sqcup B$ to denote their union if they are disjoint.

\subsubsection{The case $\rank A=1$}
Let us begin with the explicit procedure for systems in $Z_1$. 
To this end, we first divide $Z_1$ into three disjoint subsets.
To define the subsets, we introduce several more notation.
Let $\mathcal{B}$ be the unit sphere of $\R^3$. 
For $A\in Z_1$, let $W(A)$ be the solution space of the equation
\begin{equation}\label{E:Aeq}
	A \begin{pmatrix}
a \\ b \\ c
\end{pmatrix} =0.
\end{equation}
Since $\rank A=1$, 
$W(A)$ is a two dimensional linear subspace of $\R^3$, i.e., a plain which contains the origin.
Therefore $W(A)$ is characterized by a unit vector $\nu(A) \in \mathcal{B}$ as
\begin{equation}\label{E:defnu}
	W(A) = \nu(A)^\perp:= \{ x=\ltrans{(a,b,c)} \in \R^3 \ | \ \nu(A) \cdot x=0 \}.
\end{equation}
Notice that $\nu(A)$ is the vector such that the matrix $A$ is reduced into 
$\ltrans{\begin{pmatrix}
\ltrans{\nu(A)} & \ltrans{\bf 0} & \ltrans{\bf 0}
\end{pmatrix}}$
by row operations, where ${\bf 0}\in \R^3$ denotes the zero column vector.
Remark that $\nu(A)$ is not determined uniquely from $A$ because $\nu^\perp = (-\nu)^\perp$ for any $\nu\in \mathcal{B}$.
The choice becomes unique if we introduce some rules to define the sign for each $A$
or if we simply identify them, i.e., if we regard $\nu$ as an element of the projective plane $\R P_2$.
Anyway, we do not care about this uniqueness issue in what follows because the choice does not matter.

Let $\mathcal{C},\mathcal{D}\subset \mathcal{B}$ be simple closed smooth curves given by
\begin{equation}\label{E:defC}
	\begin{aligned}
	\mathcal{C}:={}& \{ \ltrans{(a,b,c)} \in \mathcal{B} \ | \ b^2 = ac ,\, a\ge0,\, c\ge 0\}\\
	={}&\left\{ \ltrans{\( \tfrac{\sqrt2(1+\sin \theta)}{\sqrt{7-\cos 2\theta}} , 
	\tfrac{\sqrt2 \cos \theta}{\sqrt{7-\cos 2\theta}},
	\tfrac{\sqrt2 (1-\sin \theta)}{\sqrt{7-\cos2\theta}}\)} \in \mathcal{B} \ \middle| \ 0 \le \theta \le 2\pi \right\}
	\end{aligned}
\end{equation}
and 
\begin{equation}\label{E:defD}
	\begin{aligned}
	\mathcal{D} :={}& \{ \ltrans{(a,b,c)} \in \mathcal{B} \ | \ b^2 = 4ac ,\, a\ge0,\, c\ge 0\}\\
	={}&
	\left\{ \ltrans{\( \tfrac{1-\sin \theta}{\sqrt{5+\cos 2\theta}} , 
	-\tfrac{2 \cos \theta}{\sqrt{5+\cos 2\theta}},
	\tfrac{ 1+\sin \theta}{\sqrt{5+\cos2\theta}}\)}\in \mathcal{B} \ \middle| \ 0 \le \theta \le 2\pi \right\},
	\end{aligned}
\end{equation}
respectively.
Further, let $S_1, S_2\subset \mathcal{B}$ be open connected surfaces given by
\begin{equation}\label{E:defS12}
	S_1 := \{ \ltrans{(a,b,c)} \in \mathcal{B} \ | \ b^2 > 4ac \}
\quad\text{and}
\quad
	S_2 := \{ \ltrans{(a,b,c)} \in \mathcal{B} \ | \ b^2 < 4ac ,\, a>0 \},
\end{equation}
respectively. Remark that
\begin{equation}\label{E:cBdecomp}
	\mathcal{B}  = S_2 \sqcup \mathcal{D} \sqcup S_1 \sqcup (-\mathcal{D}) \sqcup (-S_2).
\end{equation}
Further, $\partial S_1=\mathcal{D} \sqcup (-\mathcal{D})$ and $\partial S_2=\mathcal{D}$.
We define three subsets
\begin{align*}
	Z_{1,+} :={}& \{ A \in Z_1 \ |\ \nu (A) \in S_1 \},\\
	Z_{1,0} :={}& \{ A \in Z_1 \ |\ \nu (A) \in \mathcal{D} \sqcup (-\mathcal{D}) \},\\
	Z_{1,-} :={}& \{ A \in Z_1 \ |\ \nu (A) \in S_2 \sqcup (-S_2) \}.
\end{align*}
By \eqref{E:cBdecomp}, we have
\[
	Z_1 =  Z_{1,+} \sqcup  Z_{1,0} \sqcup  Z_{1,-}.
\]

Remark that it is easy to calculate $\nu(A)\in \mathcal{B}$ from a given $A\in Z_1$,
and consequently it is easy to tell which of three subsets $A$ belongs to.
For $A \in Z_1$ with $\nu(A) = \ltrans{(a,b,c)}$, the quantity $\sign (b^2-4ac)$
is an invariant quantity under the equivalence relation (see Proposition \ref{P:nuA}).
This tells us that $Z_{1,+}$, $Z_{1,0}$, and $Z_{1,-}$ are invariant subspaces and it holds that
\[
	Z_{1}/{\sim} = (Z_{1,+}/{\sim}) \sqcup (Z_{1,0}/{\sim}) \sqcup (Z_{,-1}/{\sim}).
\]
Hence, it is sufficient to discuss the reduction procedure in each subspace.

\begin{theorem}[Characterization of $Z_{1,+}/{\sim}$]\label{T:main1}
$Z_{1,+}$ is an invariant subset of $Z$.
One has
\begin{align*}
	Z_{1,+}/{\sim} ={}&
	\left\{ 
	\left[\begin{pmatrix}
	0 & -3 & 0 \\
	0 & 0 & 0 \\
	0 & 3 & 0 
	\end{pmatrix}\right],
	\left[\begin{pmatrix}
	0 & 3 & 0 \\
	0 & 0 & 0 \\
	0 & -3 & 0 
	\end{pmatrix}\right],
	\left[\begin{pmatrix}
	0 & -3 & 0 \\
	0 & 0 & 0 \\
	0 & -3 & 0 
	\end{pmatrix}\right],
	\right. \\
	&\quad \left. 
	\left[\begin{pmatrix}
	0 & -3 & 0 \\
	0 & 0 & 0 \\
	0 & 0 & 0 
	\end{pmatrix}\right],
	\left[\begin{pmatrix}
	0 & 3 & 0 \\
	0 & 0 & 0 \\
	0 & 0 & 0 
	\end{pmatrix}\right]
	 \right\}
\end{align*}
and $\# (Z_{1,+}/{\sim})=5$.
With the abuse of notation, this reads as $Z_{1,+}/{\sim} = \{[\eqref{E:sysdecouple}]\}$.
\end{theorem}

The following is an explicit procedure of the reduction of a system in $Z_{1,+}$ into the corresponding model system.

\begin{theorem}[Reduction of systems in $Z_{1.+}$]\label{T:main1a}
If $\rank A=1$ and $\nu(A) \in S_1$ then 
there exist 
\[
	{\bf p}_j = \ltrans{\( \tfrac{\sqrt2(1+\sin \theta_j)}{\sqrt{7-\cos 2\theta_j}} , 
	\tfrac{\sqrt2 \cos \theta_j}{\sqrt{7-\cos 2\theta_j}},
	\tfrac{\sqrt2 (1-\sin \theta_j)}{\sqrt{7-\cos2\theta_j}}\)}\in \mathcal{C} \qquad (\theta_j \in \R/2\pi\Z )
\]
($j=1,2$), ${\bf p}_1\neq{\bf p}_2$,  such that
$\nu(A)^\perp = \Span \{ {\bf p}_1,{\bf p}_2\}$.
Further, there exist $(k_1 ,k_2) \in \R^2 \setminus \{0\}$ such that $A$ is written as
\[
	A = (k_1 {\bf p}_1+k_2 {\bf p}_2)
	\ltrans{\begin{pmatrix} \cos \theta_1(1-\sin \theta_2)-\cos \theta_2(1-\sin \theta_1)\\
	-2 (\sin \theta_1 -  \sin \theta_2)\\
	-\cos \theta_1 (1+\sin \theta_2)+\cos \theta_2 (1+\sin \theta_1)
	\end{pmatrix}}.
\] 
If we define new unknowns $v_1$ and $v_2$ by
\[
	v_{j} = \begin{cases}
	\sqrt{1+\sin \theta_{j}} u_{1} + \sign (\cos \theta_{j}) \sqrt{1-\sin \theta_{j}}  u_{2} & \theta_{j} \neq 3\pi/2,\\
	\sqrt 2 u_{2} & \theta_{j} = 3\pi/2
	\end{cases}
\]
($j=1,2$) then $(v_1,v_2)$ solves
\begin{equation}\label{E:systh1}
\left\{
\begin{aligned}
&(\square + 1)v_1 
= \tfrac{2k_1}{3} (1- \cos (\theta_1 - \theta_2)) v_1^3 , \\
&(\square + 1)v_2 
= -\tfrac{2k_2}{3} (1- \cos (\theta_1 - \theta_2)) v_2^3.
\end{aligned}
\right.
\end{equation}
\end{theorem}
\begin{remark}
Remark that \eqref{E:systh1} is of the form \eqref{E:sysdecouple}.
As mentioned above, one has a system with $\tilde{\l}_1=\sign k_1$, $\tilde{\l}_8 = -\sign k_2$, and $\tilde{\l}_2=\dots = \tilde{\l}_7=0$
 by
a further application of \eqref{E:linearchange} with \eqref{E:M1}.
\end{remark}

\begin{theorem}[Characterization of $Z_{1,0}/{\sim}$]\label{T:main2}
$Z_{1,0}$ is an invariant subset of $Z$.
One has 
\begin{align*}
	Z_{1,0}/{\sim} ={}&
	\left\{ 
	\left[\begin{pmatrix}
	0 & 0 & 0 \\
	0 & 0 & -3 \\
	0 & 0 & 0 
	\end{pmatrix}\right],
	\left[\begin{pmatrix}
	0 & 0 & 0 \\
	0 & 0 & 3\\
	0 & 0 & 0 
	\end{pmatrix}\right],
	\left[\begin{pmatrix}
	0 & 0 & -3 \\
	0 & 0 & 0 \\
	0 & 0 & 0 
	\end{pmatrix}\right]
	 \right\}
\end{align*}
and $\# (Z_{1,0}/{\sim})=3$.
With the abuse of notation, this reads as $Z_{1.0}/{\sim}=\{[\eqref{E:sysnewa}], [\eqref{E:sysSunagawa}]\}$.
\end{theorem}
We have an explicit procedure also in this case.

\begin{theorem}[Reduction of systems in $Z_{1,0}$]\label{T:main2a}
If $\rank A=1$ and  $\nu(A) \in \mathcal{D} \cup (-\mathcal{D})$ then
there exist $\theta \in \R / 2\pi \Z$ and $\sigma \in \{-1,1\}$ such that
\[
	\sigma \nu(A) = 
	 \ltrans{\( \tfrac{1-\sin \theta}{\sqrt{5+\cos 2\theta}} , 
	-\tfrac{2 \cos \theta}{\sqrt{5+\cos 2\theta}},
	\tfrac{ 1+\sin \theta}{\sqrt{5+\cos2\theta}}\)}\in \mathcal{D}.
\]
Further, there exist $k ,\ell \in \R$ such that $A$ is written as
\[
	A = \(k \begin{pmatrix} 1+\sin \theta\\
	\cos \theta\\
	1- \sin \theta
	\end{pmatrix}+\ell \begin{pmatrix} \cos \theta\\
	-\sin \theta\\
	-\cos \theta
	\end{pmatrix}\)
	\ltrans{\begin{pmatrix} 1-\sin \theta\\
	-2\cos \theta\\
	1+ \sin \theta
	\end{pmatrix}}.
\] 
If we define new unknowns $v_1$ and $v_2$ by
\[
	v_{1} = \begin{cases}
	\sqrt{1+\sin \theta} u_{1} + \sign (\cos \theta) \sqrt{1-\sin \theta}  u_{2} & \theta \neq 3\pi/2,\\
	\sqrt2 u_{2} & \theta = 3\pi/2
	\end{cases}
\]
and
\[
	v_{2} = \begin{cases}
	\begin{aligned}&(k \sqrt{1+\sin \theta}+ \ell \sign (\cos \theta) \sqrt{1-\sin \theta} )u_{1} \\ 
	&\qquad +(k \sign (\cos \theta) \sqrt{1-\sin \theta} - \ell \sqrt{1+\sin \theta} )u_{2} 
	\end{aligned} & \theta \neq 3\pi/2 \text{ and } \ell\neq0,\\
	\sqrt2 (\ell u_1 + k u_2) & \theta = 3\pi/2 \text{ and } \ell\neq0, \\
	-(3/k \sqrt{1+\sin \theta})u_{2} & \theta \neq 3\pi/2 \text{ and } \ell=0, \\
	(3/\sqrt{2} k)u_{1} & \theta = 3\pi/2 \text{ and } \ell=0, 
	\end{cases}
\]
respectively, then $(v_1,v_2)$ solves
\begin{equation}\label{E:systh21}
\left\{
\begin{aligned}
&(\square + 1)v_1 
= \tfrac{\ell}3 v_1^3, \\
&(\square + 1)v_2 
= \ell v_1^2 v_2, 
\end{aligned}
\right.
\end{equation}
if $\ell\neq0$ and
\begin{equation}\tag{\ref{E:sysSunagawa}}\label{E:systh23}
\left\{
\begin{aligned}
&(\square + 1)v_1 
= 0, \\
&(\square + 1)v_2 
=  v_1^3, 
\end{aligned}
\right.
\end{equation}
if $\ell=0$.
\end{theorem}
By a further application of \eqref{E:linearchange} with \eqref{E:M1},
the first case \eqref{E:systh21} corresponds to \eqref{E:sysnewa}.
We study the behavior in Theorem \ref{T:main4} below.
The behavior of the solution to \eqref{E:sysSunagawa} is studied in \cite{Su3}.

For the remaining case, we have the following:
\begin{theorem}[Characterization of $Z_{1,-}/{\sim}$]\label{T:main5}
$Z_{1,-}$ is an invariant subset of $Z$.
One has 
\begin{align*}
	Z_{1,-}/{\sim} ={}&
	\left\{ 
	\left[\begin{pmatrix}
	0 & 0 & 0 \\
	-3 & 0 & -3 \\
	0 & 0& 0 
	\end{pmatrix}\right]
	 \right\}.
\end{align*}
This reads as $Z_{1.-}/{\sim}=\{[\eqref{E:sysnew2}]\}$.
\end{theorem}
There is an explicit procedure of the reduction. The construction is divided into two steps.
\begin{theorem}[Reduction of systems in $Z_{1,-}$]\label{T:main5a}\ 
\begin{enumerate}
\item Suppose that $\rank A=1$ and  $\nu(A)= \ltrans{( a_0 , b_0 , c_0 )}\in S_2 \cup (-S_2)$.
If we define 
\[
	\begin{pmatrix} w_1 \\ w_2 \end{pmatrix}
	= \begin{pmatrix} 2c_0 & -b_0 \\ 0 &  \sqrt{4a_0c_0-b_0^2} \end{pmatrix} \begin{pmatrix} u_1 \\ u_2 \end{pmatrix}
\]
then $(w_1,w_2)$ solves
\begin{equation}\label{E:systh5}
\left\{
\begin{aligned}
&(\square + 1)w_1 
= r \cos \theta (w_1^2 -3 w_2^2)w_1 - r \sin \theta (3w_1^2 - w_2^2)w_2, \\
&(\square + 1)w_2 
= r \sin \theta (w_1^2-3w_2^2)w_1 + r \cos \theta (3w_1^2 - w_2^2)w_2 
\end{aligned}
\right.
\end{equation}
with some $r>0$ and $\theta \in \R/2\pi \Z$.
\item Suppose that $(w_1,w_2)$ solves \eqref{E:systh5}.
If we define
\[
	\begin{pmatrix} v_1 \\ v_2 \end{pmatrix}
	= \begin{pmatrix} r^{\frac12} \cos \frac{\theta}2 & -r^{\frac12} \sin \frac{\theta}2 \\  r^{\frac12} \sin \frac{\theta}2 &  r^{\frac12} \cos \frac{\theta}2 \end{pmatrix} \begin{pmatrix} w_1 \\ w_2 \end{pmatrix}
\]
then $(v_1,v_2)$ solves \eqref{E:sysnew2}.
\end{enumerate}
\end{theorem}

\begin{example}
We consider a specific system 
\begin{equation}\label{E:NLKG1}
\left\{
\begin{aligned}
(\square + 1)u_1 
&=  a u_1^3 +3a u_1^2 u_2+ 3 u_1u_2^2 +  u_2^3, \\
(\square + 1)u_2 
&= a^2 u_1^3 + 3a u_1^2u_2 + 3a u_1u_2^2 + u_2^3, 
\end{aligned}
\right.
\end{equation}
where $a \in \R$, as an example
and demonstrate how we reduce the system by applying our main theorems.
First of all, the corresponding matrix $A$ defined by \eqref{E:defA} is
\[
	\begin{pmatrix}
	3a & 0 & -3a^2 \\
	3 & 0 & -3a \\
	3 & 0 & -3a
	\end{pmatrix}.
\]
Hence, $A\in Z_1$ and
\[
	\nu(A)= \tfrac{1}{\sqrt{a^2+1}}\ltrans{(1 , 0 , -a)}
	\in \begin{cases}
	S_1 & a>0 \\
	\mathcal{D} &  a= 0 \\
	S_2 & a<0.
	\end{cases}
\]
By means of Theorems \ref{T:main1a}, \ref{T:main2a}, and \ref{T:main5a}, we have an explicit change of variable of the form
\eqref{E:linearchange}
which reduces the system into a corresponding model system.
First consider the case $a>0$. 
Let us now specify the change of unknowns given by Theorem \ref{T:main1a}.
We want to find $\theta_1$ and $\theta_2$ such that
\[
	\nu(A) = k \begin{pmatrix} \cos \theta_1(1-\sin \theta_2)-\cos \theta_2(1-\sin \theta_1)\\
	-2 (\sin \theta_1 -  \sin \theta_2)\\
	-\cos \theta_1 (1+\sin \theta_2)+\cos \theta_2 (1+\sin \theta_1)
	\end{pmatrix}
\]
holds for some $k\in \R$. Comparing the coordinates, one has
\[
	\theta_1 = \arcsin (\tfrac{a-1}{a+1}) \in (-\tfrac{\pi}2, \tfrac{\pi}2),\quad \theta_2 = \pi - \theta_1.
\]
Hence we see that
\begin{align*}
	v_1 := \sqrt{\tfrac{2a}{a+1}} u_1 + \sqrt{\tfrac{2}{a+1}} u_2,\quad
	v_2 := \sqrt{\tfrac{2a}{a+1}} u_1 - \sqrt{\tfrac{2}{a+1}} u_2
\end{align*}
are the new unknowns given by the theorem. 
A direct computation shows that the introduction of these new unknowns reduces
the system \eqref{E:NLKG1} into
\begin{equation*}
\left\{
\begin{aligned}
(\square + 1)v_1&= \tfrac{a+1}2 (1+\sqrt{a})v_1^3,\\
(\square + 1)v_2&= \tfrac{a+1}2 (1-\sqrt{a})v_2^3.
\end{aligned}
\right.
\end{equation*}
Next consider the case $a=0$.
Recall that $\nu(A) \in \mathcal{D}$.
We will apply Theorem \ref{T:main2a}.
One has $\theta=3\pi/2$ and $\sigma=1$. $A$ is written as
\[
	A = 	\begin{pmatrix}
	0 & 0 & 0 \\
	3 & 0 & 0 \\
	3 & 0 & 0
	\end{pmatrix}
	= \(\tfrac{3}{4} \cdot \begin{pmatrix} 0 \\ 0 \\ 2 \end{pmatrix} + \tfrac32 \cdot \begin{pmatrix} 0 \\ 1 \\ 0 \end{pmatrix}\)
	\begin{pmatrix} 2 & 0 & 0 \end{pmatrix} 
\]
and so $k=3/4$ and $\ell=3/2$. We define
\[
	v_1 = \sqrt{2} u_2 , \quad v_2 = \sqrt{2} (\tfrac32 u_1 + \tfrac{3}4 u_2).
\]
Then, $(v_1,v_2)$ solves \eqref{E:systh21} with $\ell=3/2$. We further define
\[
	\tilde{v}_1 := \sqrt{|\ell|} v_1 = u_2,\quad
	\tilde{v}_2 := \sqrt{|\ell|} v_2 = \tfrac34 (2u_1+u_2).
\]
Then, $(\tilde{v}_1,\tilde{v}_2)$ solves \eqref{E:sysnewa} with $\l=1$.
Finally, let us consider the case $a<0$.
Recall that $\nu (A) \in S_2$.
We apply Theorem \ref{T:main5a}. Since $\nu(A) = \tfrac1{\sqrt{a^2+1}} \ltrans{(1 , 0 , -a)}$, we
define
\[
	w_1 = \tfrac{2|a|}{\sqrt{a^2+1}} u_1, \quad w_2 = \tfrac{2 |a|^{\frac12}}{\sqrt{a^2+1}} u_2.
\]
Then, $(w_1,w_2)$ solves
\begin{equation*}
\left\{
\begin{aligned}
(\square + 1)w_1 
&= \tfrac{a^2+1}{4|a|} ( -(w_1^2- 3w_2^2)w_1  -|a|^{1/2} (3w_1^2-w_2^2) w_2 ), \\
(\square + 1)w_2 
&= \tfrac{a^2+1}{4|a|} (  |a|^{1/2} (w_1^2 - 3 w_2^2)w_1 - (3 w_1^2-w_2^2) w_2).
\end{aligned}
\right.
\end{equation*}
Further, we take $r= \frac{(a^2+1)\sqrt{1+|a|}}{4|a|}$ and $\theta = \arctan (-|a|^{1/2}) + \pi \in (\tfrac12\pi, \pi)$
and let
\begin{align*}
	v_1 &{}= r^{\frac12}( w_1\cos \tfrac\theta2  - w_2 \sin \tfrac\theta2 ) = \tfrac{1}{\sqrt2 }(|a|^{\frac12}(\sqrt{1+|a|}-1)^{\frac12}u_1 - (\sqrt{1+|a|}+1)^{\frac12} u_2) ,\\
	v_2 &{}= r^{\frac12}( w_1 \sin \tfrac\theta2  + w_2 \cos \tfrac\theta2 )
	= \tfrac{1}{\sqrt2 }(|a|^{\frac12}(\sqrt{1+|a|}+1)^{\frac12}u_1 + (\sqrt{1+|a|}-1)^{\frac12} u_2) .
\end{align*}
Then, $(v_1,v_2)$ solves \eqref{E:sysnew2}.
\end{example}

\subsubsection{The case $\rank A=2$}
We move on to the reduction of systems in $Z_2$.

\begin{theorem}[Reduction of systems in $Z_2$]\label{T:main3a}
Suppose that \eqref{E:main3cond} holds. 
If $\l_2^2 < 4\l_1 \l_3$ then 
\[
	v_1 = \sqrt{|\l_1|} u_1 + \tfrac{\sign(\l_1)\l_2}{2\sqrt{|\l_1|}} u_2
\quad
\text{and} 
\quad
	v_2=  \tfrac{\sqrt{4\l_1\l_3-\l_2^2}}{2\sqrt{|\l_1|}} u_2
\]
solve
\begin{subequations}
\begin{equation}\label{E:systh31}
\left\{
\begin{aligned}
&(\square + 1)v_1 
= \sign (\l_1) (v_1^2 + v_2^2)v_1, \\
&(\square + 1)v_2 
= \sign (\l_1) (v_1^2+v_2^2)v_2.
\end{aligned}
\right.
\end{equation}
In particular, $w=v_1 + i v_2$ solves
$(\square +1)w = \sign (\l_1) |w|^2 w$.
If $\l_2^2 = 4\l_1 \l_3$ then 
\[
	v_1 = 
	\begin{cases}
	\sqrt{|\l_1|} u_1 + \sign (\l_1 \l_2) \sqrt{|\l_3|} u_2 & \l_1 \neq0, \\
	\sqrt{|\l_3|} u_2 & \l_1 = 0
	\end{cases}
\qquad
\text{and}
\qquad v_2=u_1
\]
solve
\begin{equation}\label{E:systh32}
\left\{
\begin{aligned}
&(\square + 1)v_1 
= \tilde{\l} v_1^3, \\
&(\square + 1)v_2 
= \tilde{\l} v_1^2v_2,
\end{aligned}
\right.
\end{equation}
where $\tilde{\l} = \sign (\l_1)$ if $\l_1\neq 0$ and $\tilde{\l} = \sign (\l_3)$ if $\l_1=0$.
If $\l_2^2 > 4\l_1 \l_3$ then 
\[
	v_1 = 
	\begin{cases}
	\sqrt{|\l_1|} u_1 + \tfrac{\sign(\l_1)\l_2}{2\sqrt{|\l_1|}} u_2 & \l_1 \neq0, \\
	u_1 + \frac{2\l_3+\l_2^2}{4\l_2} u_2 & \l_1 = 0
	\end{cases}
\]
and
\[
	v_2 = 
	\begin{cases}
	\tfrac{\sqrt{\l_2^2-4\l_1\l_3}}{2\sqrt{|\l_1|}} u_2 & \l_1 \neq0, \\
	u_1 + \frac{2\l_3-\l_2^2}{4\l_2} u_2 & \l_1 = 0
	\end{cases}
\]
solve
\begin{equation}\label{E:systh33}
\left\{
\begin{aligned}
&(\square + 1)v_1 
= \tilde{\l} (v_1^2-v_2^2)v_1, \\
&(\square + 1)v_2 
= \tilde{\l} (v_1^2-v_2^2)v_2,
\end{aligned}
\right.
\end{equation}
\end{subequations}
where $\tilde{\l} = 1$ if $\l_1\ge 0$ and $\tilde{\l} = -1$ if $\l_1<0$.
\end{theorem}
As seen in the proof below, this is a simple consequence of the theory of quadratic form.
The system \eqref{E:systh31} and \eqref{E:systh32} correspond to \eqref{E:syscomplex} and \eqref{E:sysnewb}, respectively.
By a further application of \eqref{E:linearchange} with \eqref{E:M2} if necessary, the system \eqref{E:systh33} corresponds to \eqref{E:sysnew3}.

\subsection{Results on the large time asymptotic behavior for new model systems}
We turn to the study of the asymptotic behavior of solutions to the following systems:
We give an asymptotic behavior of solutions to the system
\begin{equation}\label{E:sysnew}
\left\{
\begin{aligned}
&(\square + 1)u_1 
= \l_1 u_1^3, \\
&(\square + 1)u_2 
= \l_6 u_1^2u_2
\end{aligned}
\right.
\end{equation}
in the two specific choices $\l_6 = 3 \l_1\neq 0$ and  $\l_6 = \l_1\neq0$.
This is a unified notation of \eqref{E:sysnewa} and \eqref{E:sysnewb}.
We use the unified notation because these two systems are handled quite similarly.

On the other hand,
we do not have a similar result on the rest two systems
\begin{equation}\tag{\ref{E:sysnew2}}
\left\{
\begin{aligned}
&(\square + 1)u_1 
= (u_1^2-3u_2^2)u_1 , \\
&(\square + 1)u_2 
= (3u_1^2-u_2^2)u_2
\end{aligned}
\right.
\end{equation}
and
\begin{equation}\tag{\ref{E:sysnew3}}
\left\{
\begin{aligned}
&(\square + 1)u_1 
= (u_1^2-u_2^2)u_1, \\
&(\square + 1)u_2 
= (u_1^2-u_2^2)u_2.
\end{aligned}
\right.
\end{equation}
Nevertheless, we have an explicit solution for the corresponding limit ODE system.

\subsubsection{Asymptotic behavior of solutions to \eqref{E:sysnew}}

We consider initial value problem of \eqref{E:sysnew}.
Remark that the second equation of \eqref{E:sysnew} is linear with respect to $u_2$.
Hence, we do not need any smallness on the data of $u_2$.
To emphasize this respect, we work with the data of the form
\begin{equation}\label{E:ICprime}
	(u_1(0),\partial_t u_1(0),u_2(0),\partial_t u_2(0)) = (\eps u_{1,0}, \eps u_{1,1}, u_{2,0}, u_{2,1}).
\end{equation}
Remark that this is a merely an issue of notation.

Let us state the assumption on the initial data. 
\begin{assumption}\label{A:init}
Suppose that $u_{j,0} \in H^{15}(\R)$ and $u_{j,1} \in H^{14}(\R)$ for $j=1,2$.
Further, suppose that $u_{j,0}$ and $u_{j,1}$ ($j=1,2$) are real-valued and compactly supported.
Let $B>0$ be the number such that
\begin{equation}\label{E:support}
\supp u_{j,0}\cup\ \supp u_{j,1}
\subset\{x\in\R\ |\ |x|\le B\}
\end{equation}
holds for $j=1,2$.
\end{assumption}

\begin{theorem}\label{T:main4} 
Suppose $\l_6=3\l_1 \neq 0$ or $\l_6=\l_1 \neq 0$ in \eqref{E:sysnew}.
Suppose Assumption \ref{A:init}. 
Then, for any $\delta>0$ there exists $\eps_0>0$ such that if $\eps\le \eps_0$
then there exists a unique global solution 
$u_{j} \in C([0,\infty);H^{15}(\R))\cap C^1([0,\infty);H^{14}(\R))$ to \eqref{E:sysnew} with \eqref{E:ICprime}. 
Furthermore, there exist
$\Phi_1, \Phi_2 \in (C^1 \cap W^{1,\I})(\R) \cap C^{5}((-1,1))$ with $\supp \Phi_1, \supp \Phi_2 \subset [-1,1]$
such that
\begin{align*}
\notag u_1(t,x) =& \frac{2}{t^{1/2}}
\Re\Bigg[\Phi_1\left(\frac{x}{t}\right)
\exp\left(i\sqrt{t^2-|x|^2}
+i\Psi\left(\frac{x}{t}\right)\log t\right)
\Bigg]
+ O\left( t^{-\frac{3}{2}+\delta}\right),\\
\notag u_2(t,x) =&
\frac{2}{t^{1/2}}
\Re\Bigg[\left\{\tilde{\Psi}\left(\frac{x}{t}\right)\log \sqrt{t^2-|x|^2}+\Phi_2\left(\frac{x}{t}\right)\right\}
\exp\left(i\sqrt{t^2-|x|^2}
+i\Psi\left(\frac{x}{t}\right)\log t\right)
\Bigg]\\
&+ O\left( t^{-\frac{3}{2}+\delta}\right)
\end{align*}
holds in $L^\I_x(\R)$
as $t \to \infty$, where 
\begin{align*} 
\Psi(y) &{}=-\frac32\lambda_1\sqrt{1-|y|^2}|\Phi_1(y)|^2,\\
\tilde{\Psi}(y)&{}=
\left\{
\begin{aligned}
&-\frac32i\lambda_1\sqrt{1-|y|^2}
\left(\left|\Phi_1\left(y\right)\right|^2\Phi_2\left(y\right)
+\Phi_1\left(y\right)^2\overline{\Phi}_2\left(y\right)\right), & &\l_6= 3\l_1, \\
&\frac{i}{2}\lambda_1\sqrt{1-|y|^2}
\left(\left|\Phi_1\left(y\right)\right|^2\Phi_2\left(y\right)
-\Phi_1\left(y\right)^2\overline{\Phi}_2\left(y\right)\right), & & \l_6=\l_1.
\end{aligned}
\right.
\end{align*}
It holds that $|\Phi_1|^2, |\Phi_2|^2 \in C^{5} \cap W^{5,\I}(\R)$, $\Psi \in C^3 \cap W^{3,\I}(\R)$, and $\tilde{\Psi} \in C^5 \cap W^{5,\I}(\R)$.
Further, 
\[
	|\Phi_1(y)| + |\Phi_2(y)| = o(|1-y^2|^{\frac32} )
\]
as $|y|\to1$.
\end{theorem}

\begin{remark}
By Theorem \ref{T:main2}, one sees that the system \eqref{E:sysSunagawa} can be regarded as a limiting case of \eqref{E:sysnewa}.
Indeed, the behavior of solution is similar in these two cases. Especially,
it is common that the second component has a \emph{logarithmic amplitude correction} denoted by $\tilde{\Psi}$.
The difference is as follows:
The behavior of solutions to \eqref{E:sysnew2} involves a logarithmic phase correction term, i.e. we have $\Psi\neq0$.
Further, the logarithmic amplitude correction depends not only on $\Phi_1$ but also on $\Psi_2$.
This reflects the difference of the mechanism of appearance of logarithmic amplitude correction
(see Remark \ref{R:lac}).
\end{remark}

\subsubsection{Solutions to the limit ODE systems corresponding to \eqref{E:sysnew2} and \eqref{E:sysnew3}}

We next consider the new systems \eqref{E:sysnew2} and \eqref{E:sysnew3}.
We do not have a result on the asymptotic behavior of a solution to the systems.
Nevertheless, we have an explicit solution for the corresponding limit ODE system
\begin{equation}\label{E:newODE1}
\left\{
\begin{aligned}
	&2i\frac{\partial \alpha_1}{\partial s}=
	3 |\a_1|^2\a_1-3(2\a_1|\a_2|^2 +\ol{\a_1}\a_2^2),\\
	&2i\frac{\partial \alpha_2}{\partial s}=
	3 (2|\a_1|^2\a_2+\a_1^2\ol{\a_2})-3|\a_2|^2\a_2,
\end{aligned}
\right.
\end{equation}
and
\begin{equation}\label{E:newODE2}
\left\{
\begin{aligned}
	&2i\frac{\partial \alpha_1}{\partial s}=
	3 |\a_1|^2\a_1-(2\a_1|\a_2|^2 +\ol{\a_1}\a_2^2),\\
	&2i\frac{\partial \alpha_2}{\partial s}=
	(2|\a_1|^2\a_2+\a_1^2\ol{\a_2})-3|\a_2|^2\a_2,
\end{aligned}
\right.
\end{equation}

\begin{theorem}\label{T:main6}
\begin{enumerate}
\item
The solution to the ODE system \eqref{E:newODE1} is 
\begin{align*}
	\alpha_1(s) ={}& c_{+} \exp \(- 6i c_+ \overline{c_-} s\) + c_{-} \exp(- 6i \overline{c_+}c_- s),
	\\
	\alpha_2(s) ={}& -i c_{+} \exp(-6 i  c_+ \overline{c_-}  s) + i c_{-} \exp(- 6i \overline{c_+}c_- s ),	
\end{align*}
where
\[
	c_{\pm}  = \frac{\alpha_1(0) \pm i \alpha_2(0)}2.
\]
Moreover, $|\alpha_1(s)|^2-|\alpha_2(s)|^2$ and $2 \Re (\overline{\alpha_1(s)}\alpha_2(s))$ are
independent of $s$, and are equal to the real part and the imaginary part of $4 c_+ \overline{c_-}$, respectively.
\item The solution to the ODE system \eqref{E:newODE2} is 
\begin{align*}
	\alpha_1(s) ={}& c'_{+} \exp(-2 i( c'_+ \overline{c'_-} + 2 \overline{c'_+}c'_-) s) + c'_{-} \exp(- i(2 c'_+ \overline{c'_-} +  \overline{c'_+}c'_-) s),
	\\
	\alpha_2(s) ={}& c'_{+} \exp(-2 i( c'_+ \overline{c'_-} + 2 \overline{c'_+}c'_-) s) - c'_{-} \exp(- i(2 c'_+ \overline{c'_-} +  \overline{c'_+}c'_-) s),
\end{align*}
where
\[
	c'_{\pm }  = \frac{\alpha_1(0) \pm \alpha_2(0)}2.
\]
Moreover, $|\alpha_1(s)|^2-|\alpha_2(s)|^2$ and $2 \Im (\overline{\alpha_1(s)}\alpha_2(s))$ are
independent of $s$, and are equal to  the real part and the imaginary part of
$4\overline{c'_+}c'_-$, respectively.
\end{enumerate}
\end{theorem}

Although the systems \eqref{E:newODE1} and \eqref{E:newODE2} are \emph{nonlinear} ODE systems with a constant coefficients, their solution is a linear combination of exponential functions, which is a
typical solution to \emph{linear} ODE systems.
This is due to the special structure of the nonlinearity.
These systems have two quadratic conserved quantities, and further, the nonlinearity is 
written as a linear combination of product of the quadratic conserved quantity and an unknown.
As a result, the nonlinearity acts like a linear term.
Note that this gimmick is essentially same as in the ODE
\[
	2i \frac{d \alpha}{d s} = 3|\alpha|^2 \alpha
\]
which appears in the study of single cubic nonlinear Klein-Gordon equation \eqref{E:cNLKG}.
The difference compared with the solution to this equation is that the exponential factors have nonzero real part and 
hence modulus of the both components of the solution grow exponentially.

\begin{remark}\label{R:conjecture}
The analysis of the ODE systems would be a key step to find asymptotic behavior to the other two new systems \eqref{E:sysnew2} and \eqref{E:sysnew3}.
In view of Theorem \ref{T:main6}, we obtain a conjecture on the behavior:
A class of small solution $(u_1,u_2)$ of \eqref{E:sysnew2} may behave like
\begin{align*} 
 u_1(t,x) =& \tfrac{1}{t^{1/2}}
\Big[ \Re\left[\Phi_+\left(\tfrac{x}{t}\right)
\exp\left(i\sqrt{t^2-|x|^2}
-i\Psi\left(\tfrac{x}{t}\right)\log \sqrt{t^2-|x|^2}\right)
\right] \\
&\qquad+ \Re\left[\Phi_-\left(\tfrac{x}{t}\right)
\exp\left(i\sqrt{t^2-|x|^2}
-i\ol{\Psi\left(\tfrac{x}{t}\right)}\log \sqrt{t^2-|x|^2}\right)
\right]\Big]
+ o\left( t^{-\frac{1}{2}}\right),\\
 u_2(t,x) =& \tfrac{1}{t^{1/2}}
\Big[ \Im\left[\Phi_+\left(\tfrac{x}{t}\right)
\exp\left(i\sqrt{t^2-|x|^2}
-i\Psi\left(\tfrac{x}{t}\right)\log \sqrt{t^2-|x|^2}\right)
\right] \\
&\qquad - \Im\left[\Phi_-\left(\tfrac{x}{t}\right)
\exp\left(i\sqrt{t^2-|x|^2}
-i\ol{\Psi\left(\tfrac{x}{t}\right)}\log \sqrt{t^2-|x|^2}\right)
\right]\Big]
+ o\left( t^{-\frac{1}{2}}\right)
\end{align*}
in $L^\I_x(\R)$ as $t \to \infty$, where $\Phi_\pm$ is suitable complex-valued functions and
\begin{align*} 
\Psi(y) :=\tfrac32 \sqrt{1-|y|^2}\Phi_+(y) \ol{\Phi_-(y)}.
\end{align*}
A class of small solution $(u_1,u_2)$ of \eqref{E:sysnew3} may behave like
\begin{align*} 
 u_1(t,x) =& \tfrac{1}{t^{1/2}}
\Big[ \Re\left[\Phi_+\left(\tfrac{x}{t}\right)
\exp\left(i\sqrt{t^2-|x|^2}
+i\Psi\left(\tfrac{x}{t}\right)\log \sqrt{t^2-|x|^2}\right)
\right] \\
&\qquad + \Re\left[\Phi_-\left(\tfrac{x}{t}\right)
\exp\left(i\sqrt{t^2-|x|^2}
+i\ol{\Psi\left(\tfrac{x}{t}\right)}\log \sqrt{t^2-|x|^2}\right)
\right]\Big]
+ o\left( t^{-\frac{1}{2}}\right),\\
 u_2(t,x) =& \tfrac{1}{t^{1/2}}
\bigg[ \Re\left[\Phi_+\left(\tfrac{x}{t}\right)
\exp\left(i\sqrt{t^2-|x|^2}
+i\Psi\left(\tfrac{x}{t}\right)\log \sqrt{t^2-|x|^2}\right)
\right] \\
&- \Re\left[\Phi_-\left(\tfrac{x}{t}\right)
\exp\left(i\sqrt{t^2-|x|^2}
+i\ol{\Psi\left(\tfrac{x}{t}\right)}\log \sqrt{t^2-|x|^2}\right)
\right]\Big]
+ o\left( t^{-\frac{1}{2}}\right)
\end{align*}
holds in $L^\I_x(\R)$
as $t \to \infty$, where $\Phi_\pm$ is suitable complex-valued functions and
\begin{align*} 
\Psi(y) :=-\tfrac12 \sqrt{1-|y|^2}(  \Phi_+(y) \ol{\Phi_-(y)} + 2\ol{\Phi_+(y)} \Phi_-(y)).
\end{align*}
Remark that in the both cases, $\Psi$ is a complex-valued function and hence the decay rates
of $u_1$ and $u_2$ are different from that of the linear solutions by a \emph{polynomial order}.
\end{remark}

\subsection{A summary}
As mentioned above, the classification results can be summarized as $Z_1/{\sim} = \{ [\eqref{E:sysdecouple}]$, $[\eqref{E:sysSunagawa}]$, $[\eqref{E:sysnewa}]$, $[\eqref{E:sysnew2}] \}$ and $Z_2/{\sim} = \{ [\eqref{E:syscomplex}]$, $[\eqref{E:sysnewb}]$, $[\eqref{E:sysnew3}]\}$ with an abuse of notation.
The model systems for $Z_1$ and $Z_2$ are divided into three groups as shown in Table \ref{tab:group}.
\begin{table}[bh]
\centering
\caption{Classification of model systems}
\begin{tabular}{|c||c|c|c|}
\hline
Subset of $Z$ & Group 1 & Group 2 & Group 3 \\
\hline 
$Z_1$
 & \eqref{E:sysdecouple} & \eqref{E:sysSunagawa} and \eqref{E:sysnewa} & \eqref{E:sysnew2} \\
\hline
$Z_2$
 & \eqref{E:syscomplex} & \eqref{E:sysnewb} & \eqref{E:sysnew3} \\
\hline
\end{tabular}
\label{tab:group}
\end{table}
Firstly, small good solutions to the systems in Group 1 involve logarithmic phase correction. It was well-known.
As for the systems in Group 2, the behavior of solutions have a logarithmic amplitude correction.
This kind of behavior was studied for \eqref{E:sysSunagawa}.
We show in Theorem \ref{T:main4} that the other two systems \eqref{E:sysnewa} and \eqref{E:sysnewb} also admit a solution with the same kind of behavior.
Further, Theorem \ref{T:main2a} reveals that the system \eqref{E:sysSunagawa} is rather a limit case of \eqref{E:sysnewa}.
Although the appearance of a logarithmic amplitude correction is common in these systems, 
the mechanism of the appearance is slightly different (see Remark \ref{R:lac}).
Another difference is that the behavior of solutions to \eqref{E:sysnewa} and \eqref{E:sysnewb} also involves a logarithmic phase correction.
Theorem \ref{T:main6} suggests that the systems in Group 3 may admit a solution with a \emph{polynomial} amplitude correction (see Remark \ref{R:conjecture}).
However, it is not proven yet.

\bigskip

The rest of the paper is organized as follows.
We first prove Theorem \ref{T:classA} and collect several preliminary results
in Section \ref{S:Theorem1}.
By exploiting them,
we prove our reduction results on $Z_1$, i.e., Theorems  \ref{T:main1a}, \ref{T:main2}, \ref{T:main2a}, \ref{T:main5}, and \ref{T:main5a} in Section \ref{S:main1}.
Section \ref{S:main3} is devoted to the proof of Theorems \ref{T:main3} and \ref{T:main3a}.
Then, we turn to the study of asymptotic behavior.
In Section \ref{S:pre}, we collect preliminary results for the study of large time behavior of solutions to the system \eqref{E:sys} with a general nonlinearity.
In particular, the derivation of the limit ODE system is discussed in Subsection \ref{subsec:limitODEsys}.
We study conserved quantities for the limit ODE system in Subsections \ref{subsec:conserved_quantity} and \ref{subsec:another_conservation}.
Then, we prove Theorem \ref{T:main4} in Section \ref{S:main4}.
Finally, Theorem \ref{T:main6} is shown in Section \ref{S:main6}.

\section{Proof of Theorem \ref{T:classA}}\label{S:Theorem1}

\subsection{Explicit formula of an equivalent matrix}

Recall the equivalence relation $\sim$ between two elements of $Z$ is defined via
change of unknowns \eqref{E:linearchange}.

\begin{lemma}\label{L:Dhomo}
The map $GL_2(\R)\ni M \mapsto D(M) \in SL_3(\R)$ defined by \eqref{E:DMform}
 is a group homomorphism,
i.e., $D(M_1M_2)= D(M_1)D(M_2)$ is true for any $M_1,M_2 \in GL_2(\R)$.
In particular,
it holds that $D(M^{-1})=D(M)^{-1}$. Further,
\begin{equation}\label{E:DM-1form}
	D(M)^{-1} =\frac1{\det M}\begin{pmatrix} a^2 & 2ac & c^2 \\ ab & ad+ bc & cd \\ b^2 & 2bd & d^2 \end{pmatrix} \in SL_3(\R).
\end{equation}
\end{lemma}
\begin{proof}
It is easy to verify 
$D(M_1M_2)= D(M_1)D(M_2)$ for any $M_1,M_2 \in GL_2(\R)$ by definition.
The formula \eqref{E:DM-1form} follows, for instance, from the identity $D(M)^{-1}=D(M^{-1})$ and the explicit formula for an inverse of $2\times2$ matrix.
\end{proof}

\begin{remark}
The matrix $D(M)^{-1}$ appears in the classical invariant theory of \emph{quadratic polynomials}
(see \cite{OlBook}, for instance).
\end{remark}

\begin{proof}[Proof of Theorem \ref{T:classA}]
We first consider the case where the change between $A$ and $A'$ is given by the specific 
changes \eqref{E:M1}, \eqref{E:M2}, or \eqref{E:M3}.
The validity of \eqref{E:Apform} is directly checked by applying the corresponding change of variables to the system \eqref{E:sys}.
If 
$M$ is given by \eqref{E:M1} then
\begin{subequations}
\begin{equation}\label{E:At1}
	A' = \frac1{pq} \begin{pmatrix} q/p & 0 & 0 \\ 0 & 1 & 0 \\ 0 & 0 & p/q \end{pmatrix}
	A  \begin{pmatrix} q/p & 0 & 0 \\ 0 & 1 & 0 \\ 0 & 0 & p/q \end{pmatrix}^{-1}.
\end{equation}
If $M$ is given by \eqref{E:M2} then
\begin{equation}\label{E:At2}
	A' = \frac1{-1} \begin{pmatrix} 0 & 0 & -1 \\ 0 & -1 & 0 \\ -1 & 0 & 0 \end{pmatrix}
	A \begin{pmatrix} 0 & 0 & -1 \\ 0 & -1 & 0 \\ -1 & 0 & 0 \end{pmatrix}^{-1}.
\end{equation}
If $M$ is given by \eqref{E:M3} then
\begin{equation}\label{E:At3}
	A' = \frac{1}1 \begin{pmatrix} 1 & 0 & 0 \\ -1 & 1 & 0 \\ 1 &- 2 & 1 \end{pmatrix}
	A \begin{pmatrix} 1 & 0 & 0 \\ -1 & 1 & 0 \\ 1 & -2 & 1 \end{pmatrix}^{-1}.
\end{equation}
\end{subequations}

The general case follows from the above three cases and Lemma \ref{L:Dhomo}.
Indeed, if $M\in GL_2(\R)$ is written as $M=M_1M_2$ with $M_1,M_2 \in GL_2(\R)$ and
if the formula \eqref{E:Apform} is true for $M_1$ and $M_2$
then
\begin{align*}
	A' ={}& \frac1{\det M_1} D(M_1) \( \frac1{\det M_2} D(M_2) A D(M_2)^{-1} \) D(M_1)^{-1}\\
	={}&\frac1{\det M_1 \det M_2} D(M_1) D(M_2) A (D(M_1) D(M_2))^{-1}\\
	={}&\frac1{\det (M_1M_2)} D(M_1M_2) A D(M_1M_2)^{-1},
\end{align*}
which means \eqref{E:Apform} is true for $M=M_1M_2$.
\end{proof}

\subsection{Preliminary results on rank one matrices in $Z$}

\begin{proposition}\label{P:dnuA}
Let $A \in Z$.
$\rank A =1 $ if and only if $A$ is written as
\[
	A = {\bf d} \ltrans{\nu(A)}
\]
with some nonzero vector ${\bf d} \in \nu(A)^\perp$, where $\nu(A)$ is given by \eqref{E:defnu}.
\end{proposition}
\begin{proof}
The if-part is obvious. Let us show the only-if-part. 
Suppose that $\rank A=1$.
Since $A$ is transformed to a row-echelon form $\ltrans{\begin{pmatrix}\ltrans{\nu(A)}& \ltrans{\bf 0}& \ltrans{\bf 0}\end{pmatrix}}$,
$A$ is given by $A = {\bf d} \ltrans{\nu(A)}$ with some nonzero vector ${\bf d} \in \R^3$.
Further, since  
$
	\tr A= \ltrans{{\bf d}} \nu(A) = 0,
$
one has ${\bf d} \in \nu(A)^\perp$.
\end{proof}

\begin{corollary}\label{C:nuA}
Let $A , A' \in Z_1$. 
Let $ \nu(A) $ and $ \nu(A') $ be the vectors given by \eqref{E:defnu}.
If $A'$ is given from $A$ through
the change \eqref{E:linearchange} with $M \in GL_2(\R)$ then
\[
	\nu (A') = k \(\ltrans{D(M)^{-1}} \nu(A)\)
\]
holds with a normalizing factor $k\in \R \setminus \{0\}.$
\end{corollary}
\begin{proof}
Combining Theorem \ref{T:classA} and Proposition \ref{P:dnuA}, we have
\[
	A' = \frac1{\det M} D(M) ( {\bf d} \ltrans{\nu(A)}) D(M)^{-1}
	=\(\frac1{\det M}D(M)  {\bf d}\) \ltrans{ \(\ltrans{D(M)^{-1}} \nu(A)\)}
\]
with some ${\bf d} \in \nu(A)^\perp$.
On the other hand, $A'$ has another representation $A' = {\bf d}' \ltrans{\nu(A')}$
with some ${\bf d}' \in \nu(A')^\perp$.
Comparing the two forms, we obtain the desired identity.
\end{proof}

\begin{proposition}\label{P:nuA}
Let $A , A' \in Z_1$. 
Let $ \nu(A) = \ltrans{( a , b , c )}$ and $ \nu(A') = \ltrans{( \tilde{a} , \tilde{b} , \tilde{c} )}$ be the vectors given by \eqref{E:defnu}.
If $A \sim A'$ then $\sign (b^2 - 4ac) = \sign (\tilde{b}^2 - 4\tilde{a}\tilde{c})$.
In particular, $Z_{1,+}$, $Z_{1,0}$, and $Z_{1,-}$ are invariant.
\end{proposition}

\begin{proof}
By Proposition \ref{P:dnuA}, $A={\bf d} \ltrans{\nu(A)}$ holds with some nonzero ${\bf d} \in \R^3$. 
If $A'$ is written as $A' = \kappa D A D^{-1}$
with $\kappa \in \R \setminus\{0\}$ and a matrix $D\in SL_3(\R)$ then one has
\[
	A' = (\kappa D{\bf d}) \ltrans{(\ltrans{D}^{-1} \nu(A))}.
\]
By Proposition \ref{P:dnuA}, this yields
\[
	\nu(A')=\ltrans{(\tilde{a} , \tilde{b} , \tilde{c})}=|(\ltrans{D^{-1}} \nu(A))|^{-1} (\ltrans{D^{-1}} \nu(A)).
\]

\begin{subequations}
The invariant property of $\sign (b^2 - 4ac) = \sign (\tilde{b}^2 - 4\tilde{a}\tilde{c})$ is 
a well-known property of $D^{-1}$ in the classical invariant theory of quadratic polynomials (see \cite{OlBook}, for instance).
Here we give a proof for reader's convenience.
Notice that it suffices to consider the case where the change between $A$ and $A'$ is given by
\eqref{E:M1}, \eqref{E:M2}, or \eqref{E:M3}.
Recall the formula \eqref{E:DM-1form}.
When
\begin{equation}\label{E:D1}
	D^{-1} = \begin{pmatrix} p/q & 0 & 0 \\ 0 & 1 & 0 \\ 0 & 0 & q/p \end{pmatrix},
\end{equation}
it holds that
$
	\ltrans{D^{-1}} \nu(A) = \ltrans{( \frac{p}{q} a ,  b , \frac{q}{p} c
	)}.
$
Since
$
	\tilde{b}^2 - 4\tilde{a}\tilde{c} = k(b^2-4ac)
$
with $k=|\ltrans{D^{-1}} \nu(A)|^{-2}>0$,
we have $\sign (b^2 - 4ac) = \sign (\tilde{b}^2 - 4\tilde{a}\tilde{c})$. 
When
\begin{equation}\label{E:D2}
	D^{-1} = \begin{pmatrix} 0 & 0 & -1 \\ 0 & -1 & 0 \\ -1 & 0 & 0 \end{pmatrix},
\end{equation}
it holds that
$
	\ltrans{D^{-1}} \nu(A) = \ltrans{( -c ,- b ,- a)}.
$
Hence, we have $\sign (b^2 - 4ac) = \sign (\tilde{b}^2 - 4\tilde{a}\tilde{c})$. 
When
\begin{equation}\label{E:D3}
	D^{-1} = \begin{pmatrix} 1 & 0 & 0 \\ 1 & 1 & 0 \\ 1 & 2 & 1 \end{pmatrix},
\end{equation}
it holds that
$
	\ltrans{D^{-1}} \nu(A) = \ltrans{( a +b + c, b + 2c , c
	)}.
$
Since 
\[
	\tilde{b}^2 - 4\tilde{a}\tilde{c} = k((b+2c)^2 - 4 (a+b+c) c) = k(b^2 -4ac)
\]
with $k=|\ltrans{D^{-1}} \nu(A)|^{-2}>0$,
we have $\sign (b^2 - 4ac) = \sign (\tilde{b}^2 - 4\tilde{a}\tilde{c})$. 
\end{subequations}
This completes the proof.
\end{proof}

\section{Proofs of classification theorems for $Z_1$}\label{S:main1}

In this section, we consider the case $\rank A =1$ and prove Theorems \ref{T:main1}, \ref{T:main1a}, \ref{T:main2}, \ref{T:main2a},
\ref{T:main5}, and \ref{T:main5a}.

\subsection{A good variable}\label{subs:good}
To begin with, we explain our key idea for the proof.
As mentioned in the introduction, the study of the limit ODE system helps us finding a
good variable to work with.
We will see in Subsection \ref{subsec:conserved_quantity} that if $\ltrans{(a,b,c)}$ is a solution to
\begin{equation}\label{E:Aeq2}
	A \begin{pmatrix}
a \\ b \\ c
\end{pmatrix} =0
\end{equation}
 then one has
\[
\frac{d}{dt}
\left(
a|\alpha_1|^2+2b\Re(\overline{\alpha}_1\alpha_2)+c|\alpha_2|^2
\right)
=0
\]
 for any solution to the ODE system \eqref{E:limitODE}.
If, in addition, the solution $\ltrans{(a,b,c)}$ belongs to 
\[
	V:=\{ \ltrans{(a,b,c)}\in \R^3 \ |\ b^2 - ac=0\}.
\]
then it reads as
\[
	\frac{d}{dt} ||a|^{1/2} \alpha_1 + \sign (ab) |c|^{1/2} \alpha_2|^2 = 0
\]
if $a \neq 0$ or
\[
	\frac{d}{dt} |\sign (bc)|a|^{1/2} \alpha_1 + |c|^{1/2} \alpha_2|^2 = 0.
\]
if $c \neq 0$. 
This suggests that
\[
	|a|^{1/2} \alpha_1 + \sign (ab) |c|^{1/2} \alpha_2 \quad \text{ or } \quad 
	\sign (bc)|a|^{1/2} \alpha_1 + |c|^{1/2} \alpha_2
\]
is a good variable since they have a constant modulus.
It will turn out that these variables are actually good to work with, not only for the ODE system
\eqref{E:limitODE} but also for the original system \eqref{E:sys}.

\subsection{Trichotomy}

When $\rank A=1$, there exist two linearly independent solutions to \eqref{E:Aeq2}.
Here we investigate the number of solutions which belong to $V$.
Note that $V$ is not a linear space since it is not closed with respect to sum of vectors.
On the other hand, it is closed with respect to scalar multiplication: 
if $\ltrans{(a_0,b_0,c_0)} \in V$ then $k\ltrans{(a_0,b_0,c_0)} \in V$ for all $k\in \R$.

We have the following:

\begin{theorem}[Number of the solutions in $V$]\label{T:rank1criteria}
Let $A \in Z$ satisfy $\rank A=1$. The following trichotomy is true:
\begin{enumerate}
\item if $\nu(A) \in S_1$ then
$\nu(A)^\perp\cap V = \{ k {\bf p}_1 \ |\ k \in \R \} \cup \{ k {\bf p}_2 \ |\ k \in \R \}$ for some ${\bf p}_j \in \mathcal{C}$ ($j=1,2$)
with ${\bf p}_1\neq{\bf p}_2$;
\item if $\nu(A) \in \mathcal{D} \cup (-\mathcal{D})$ then
$\nu(A)^\perp\cap V = \{ k {\bf p} \ |\ k \in \R \}$ for some ${\bf p} \in \mathcal{C}$;
\item If $\nu(A) \in S_2 \cup (-S_2)$ then $\nu(A)^\perp\cap V = \emptyset$,
\end{enumerate}
where $\mathcal{C}$, $\mathcal{D}$, $S_1$, and $S_2$ are defined by \eqref{E:defC}, \eqref{E:defD}, and \eqref{E:defS12},
respectively.
\end{theorem}

Let $\mathcal{Q}$ be the quadrangular prism $\partial(\{  |a|+|c|\le1\} \cap \{ |b|\le1 \})$.
Recall that $V$ is closed with respect to scalar multiplication.
Hence, it suffices to characterize the restriction of $V$ onto $\mathcal{B}$ or $\mathcal{Q}$.
Let $P$ be a continuous bijection from $\mathcal{B}$ to $\mathcal{Q}$ such that $\mathcal{B} \ni {\bf v} \mapsto k {\bf v} \in \mathcal{Q}$ with suitable 
$k=k({\bf v}) >0$.

\begin{lemma}
There exists a simple closed smooth curve $\mathcal{C} \subset \mathcal{B}\cap \{ a\ge 0 , c\ge0 \}$ such that
$V\cap \mathcal{B} = \mathcal{C} \cup (-\mathcal{C})$.
Moreover, $\mathcal{C}$ is given by \eqref{E:defC}.
\end{lemma}
\begin{proof}
First consider the set $\tilde{\mathcal{C}}:=V \cap \{ a+c = 1 , a\ge 0 , |b|\le 1 ,c\ge0\} \subset \mathcal{Q}$.
An element of the set $\{ a+c = 1 , a\ge 0 , |b|\le 1 ,c\ge0\}$ is written as
$
	(a,b,1-a) 
$
with $a \in [0,1]$ and $ |b| \le  1$.
The element belongs $V$ if and only if
\[
	b^2 = a(1-a) \Leftrightarrow (a-\tfrac12)^2 + b^2 = (\tfrac12)^2 .
\]
Hence, $\tilde{\mathcal{C}}$ is an ellipse, a simple closed curve, on $\mathcal{Q}$.
Further, $\tilde{\mathcal{C}}$ is given by
\[
	 \left\{ (\tfrac12(1 + \sin\theta), \tfrac12 \cos \theta, \tfrac12(1 - \sin \theta)) \ \middle| \ \theta \in [0,2\pi]\right\}
\]
for instance.
Since $P^{-1}: \mathcal{Q}\to \mathcal{B}$ is a continuous bijection,
we see that $\mathcal{C}:= P^{-1} \tilde{\mathcal{C}}=V\cap \mathcal{B} \cap\{a\ge0,c\ge0\}$ is also a simple closed curve in $\mathcal{B}$.
The parametrization \eqref{E:defC} of $\mathcal{C}$ is easily constructed from that of $\tilde{\mathcal{C}}$ above.
By the explicit formula \eqref{E:defC}, one sees that $\mathcal{C}$ is a smooth curve.

Remark that $\mathcal{C}\cap \{c=0\} = \{(1,0,0)\}$.
It is easy to see that $V\cap \mathcal{B} \cap\{a<0,c\ge0\}=\{(-1,0,0)\} = -\mathcal{C}\cap\{c\ge0\}$.
Hence, 
\begin{align*}
	V\cap \mathcal{B} \cap\{c\ge0\} ={}& ((V\cap \mathcal{B} \cap\{ a\ge0,\, c\ge0\} ) \cap \{ c\ge0 \} ) \cup
	(V\cap \mathcal{B} \cap\{ a<0,\, c\ge0\})\\
	={}& ( \mathcal{C}\cap \{ c\ge0 \} ) \cup ( (-\mathcal{C})\cap \{ c\ge0 \} ) \\
	={}& (\mathcal{C}\cup(-\mathcal{C}))\cap\{c\ge0\}.
\end{align*}
In particular, $V\cap \mathcal{B} \cap\{c>0\} = (\mathcal{C}\cup(-\mathcal{C}))\cap\{c>0\}$.
Since both $V\cap \mathcal{B}$ and $\mathcal{C}\cup(-\mathcal{C})$ are symmetric with respect to the origin, one has
\[
	V\cap \mathcal{B} \cap\{c<0\} = -(V\cap \mathcal{B} \cap\{c>0\})
	= -( (\mathcal{C}\cup(-\mathcal{C}))\cap\{c>0\})
	=  (\mathcal{C}\cup(-\mathcal{C}))\cap\{c<0\}.
\]
Thus, we conclude that $V\cap \mathcal{B} = \mathcal{C}\cup(-\mathcal{C})$.
\end{proof}

Let us now turn to $W \cap \mathcal{B}$.
Remark that $W \cap \mathcal{B}$ is a geodesic of $\mathcal{B}$.
\begin{lemma}
If a geodesic of $\mathcal{B}$ is tangent to $\mathcal{C}$ then it has only one intersection point
with $\mathcal{C}$.
\end{lemma}
\begin{proof}
Let $g$ be a geodesic of $\mathcal{B}$.
The geodesic tangent to $\mathcal{C}$ at $(1,0,0)$ is $\mathcal{B} \cap \{ c=0\}$.
The geodesic tangent to $\mathcal{C}$ at $(0,0,1)$ is $\mathcal{B} \cap \{ a=0\}$.
If $g$ is one of them, one directly verifies that $g \cap \mathcal{C}$ is a point.
To handle the other case, we consider the problem on the quadrangular prism $\mathcal{Q}$.

If $g\cap\{a>0,c>0\}\neq \emptyset$ then
$\ell:=P(g \cap \{a>0,c>0\})$ is a line segment on $\mathcal{Q}$.
Remark that $g \cap \{a>0,c>0\}$ is tangent to $\mathcal{C}\cap \{a>0,c>0\}$ if and only if $\ell$
is tangent to $\tilde{\mathcal{C}}\cap \{a>0,c>0\}$.
Since $\tilde{\mathcal{C}}$ is an ellipse,
if $\ell$ is tangent to $\tilde{\mathcal{C}}\cap \{a>0,c>0\}$ then
$\ell$ do not intersect with it at any other point in $\tilde{\mathcal{C}}\cap \{a>0,c>0\}$.
Further, since $\tilde{\mathcal{C}}$ does not intersect with any other faces of $\mathcal{Q}$,
$Pg$ do not have any more intersection points with $\tilde{\mathcal{C}}$.
\end{proof}

\begin{proof}[Proof of Theorem \ref{T:rank1criteria}]
Let ${\bf v}(s): \R/\Z \to \R^3$ be a smooth parametrization of $\mathcal{C}$.
Define $\tilde{\mathcal{D}}$ by
\[
	\tilde{\mathcal{D}} := \{ |{\bf v}'(s)\times {\bf v}(s)|^{-1} {\bf v}'(s)\times {\bf v}(s) \ |\ s \in \R/\Z \}.
\]
Remark that $\tilde{\mathcal{D}} \cup (-\tilde{\mathcal{D}})$ is independent of the direction of the parametrization of $\mathcal{C}$
and hence uniquely determined.
By means of the parametrization \eqref{E:defC} of $\mathcal{C}$,
$$\mathcal{D}:=(\tilde{\mathcal{D}} \cup (-\tilde{\mathcal{D}})) \cap \{a\ge0, c\ge0\}$$ is given by \eqref{E:defD}.
Moreover, $\mathcal{D}=\tilde{\mathcal{D}}$ or $-\tilde{\mathcal{D}}$ holds and so
 $\mathcal{D}\cup (-\mathcal{D}) =\tilde{\mathcal{D}} \cup (-\tilde{\mathcal{D}})$.
Moreover, by the explicit formula, one easily sees that
$\mathcal{D}$ is a simple closed smooth curve on $\mathcal{B}\cap\{a\ge0,c\ge0\}$ and satisfies $\mathcal{D} \cap (-\mathcal{D}) = \emptyset$
and that
\[
	\mathcal{D} \cup (-\mathcal{D}) = \{\ltrans{(a,b,c)} \in \mathcal{B} \ |\ b^2-4ac=0\}. 
\]
The unit sphere $\mathcal{B}$ is divided into three parts by $\mathcal{D} \cup (-\mathcal{D})$.
Namely, $\mathcal{B} \setminus (\mathcal{D} \cup (-\mathcal{D})) = S_2 \cup S_1 \cup (-S_2)$, where 
$S_1$ and $S_2$ are defined in \eqref{E:defS12}.

Let us first consider the second case.
Let ${\bf v}(s):[0,1]\to \R^3$ be a parametrization of the curve $\mathcal{C}$.
By definition of $\mathcal{D}$,
if $\nu (A) \in \mathcal{D} \cup (-\mathcal{D})$ then there exists $s_0 \in [0,1]$ such that
$\nu(A)$ is perpendicular to both ${\bf v}(s_0)$ and ${\bf v}'(s_0)$.
This implies that $\nu(A)^\perp$ is the plane spanned by ${\bf v}(s_0)$ and ${\bf v}'(s_0)$.
Hence, the geodesic $\nu(A)^\perp\cap \mathcal{B}$ is tangent to $\mathcal{C}$ at ${\bf v}(s_0) \in \mathcal{C}$.
By the previous lemma, $W \cap \mathcal{B}$ does not have any more intersection with $\mathcal{C}$.
Thus, $W \cap V \cap \mathcal{B} = \{\pm {\bf v}(s_0)\}$. 

Let us consider the other cases.
Fix $\nu (A_0) \not\in \mathcal{D} \cup (-\mathcal{D})$.
In this case, the corresponding geodesic $\nu(A_0)^\perp \cap \mathcal{B}$ do not tangent to $\mathcal{C}$.
Since $\mathcal{C}$ is a simple closed smooth curve, there exists a small neighborhood of $U \subset \mathcal{B}$ of $\nu(A_0)$ such that if $\nu(A) \in U$
then the number of the intersection points of $\nu(A)^\perp \cap \mathcal{B}$ and $\mathcal{C}$ is constant.
Thus, the number of intersections is constant in each open surfaces $S_j$.
By a direct observation of the specific choices $(2^{-1/2},0,2^{-1/2}) \in S_2$ and $(0,1,0) \in S_1$,
we obtain the result.
\end{proof}

\subsection{Proof of Theorems \ref{T:main1} and \ref{T:main1a}}

\begin{proof}[Proof of Theorem \ref{T:main1a}]
Suppose that $\rank A =1$ and $\nu(A) \in S_1$.
By the trichotomy (Theorem \ref{T:rank1criteria}), $\nu (A)^\perp$ has two intersection points with $\mathcal{C}$.
We name them as
\[
	{\bf p}_j = r(\theta_j) \begin{pmatrix}1+\sin \theta_j\\ \cos \theta_j\\ 1-\sin \theta_j \end{pmatrix}\qquad (j=1,2)
\]
with the normalizing factor $r(\theta)=\sqrt2 / \sqrt{7-\cos 2\theta}$ and $\theta_1,\theta_2 \in [0,2\pi)$ such that $\theta_1\neq \theta_2$.
Remark that $\nu(A)\in \mathcal{B}$ is written as
\[
	\nu (A) = c r(\theta_1)^{-1}r(\theta_2)^{-1}({\bf p}_1 \times {\bf p}_2) = c
	\begin{pmatrix} \cos \theta_1(1-\sin \theta_2)-\cos \theta_2(1-\sin \theta_1)\\
	-2 (\sin \theta_1 -  \sin \theta_2)\\
	-\cos \theta_1 (1+\sin \theta_2)+\cos \theta_2 (1+\sin \theta_1)
	\end{pmatrix}
\]
with a suitable normalizing factor $c=c(\theta_1,\theta_2)\neq0$.
Since $\nu(A)^\perp = \Span \{ {\bf p}_1 , {\bf p}_2\}$, we deduce from Proposition \ref{P:dnuA}
that there exists $(k_1,k_2) \in \R^2 \setminus \{0\}$ such that
the matrix $A$ is given by
\[
	A = (k_1 {\bf p}_1 + k_2 {\bf p}_2) (c^{-1}\ltrans{\nu(A)}).
\]
Thus, by \eqref{E:defA},
\begin{align*}
	\lambda_1 ={}& \tfrac13(k_1 \cos \theta_1 + k_2 \cos \theta_2) (\cos \theta_1 (1+\sin \theta_2)-\cos \theta_2 (1+\sin \theta_1))\\
	&{}+ \tfrac13(k_1(1+\sin\theta_1)+k_2(1+\sin\theta_2))2(\sin \theta_1 -  \sin \theta_2) \\
	\lambda_2 ={}& (k_1(1+\sin\theta_1)+k_2(1+\sin\theta_2))(\cos \theta_1(1-\sin \theta_2)-\cos \theta_2(1-\sin \theta_1))\\
	\lambda_3 ={}& (k_1 \cos \theta_1 + k_2 \cos \theta_2)(\cos \theta_1(1-\sin \theta_2)-\cos \theta_2(1-\sin \theta_1))\\
	\lambda_4 ={}& \tfrac13(k_1(1-\sin\theta_1)+k_2(1-\sin\theta_2))(\cos \theta_1(1-\sin \theta_2)-\cos \theta_2(1-\sin \theta_1))\\
	\lambda_5 ={}& \tfrac13(k_1(1+\sin\theta_1)+k_2(1+\sin\theta_2))(\cos \theta_1 (1+\sin \theta_2)-\cos \theta_2 (1+\sin \theta_1))\\
	\lambda_6 ={}& (k_1 \cos \theta_1 + k_2 \cos \theta_2) (\cos \theta_1 (1+\sin \theta_2)-\cos \theta_2 (1+\sin \theta_1))\\
	\lambda_7 ={}& (k_1(1-\sin\theta_1)+k_2(1-\sin\theta_2))(\cos \theta_1 (1+\sin \theta_2) - \cos \theta_2 (1+\sin \theta_1) )\\
	\lambda_8 ={}& \tfrac13(k_1 \cos \theta_1 + k_2 \cos \theta_2)(\cos \theta_1(1-\sin \theta_2)-\cos \theta_2(1-\sin \theta_1))\\
	&{}- \tfrac13(k_1(1-\sin\theta_1)+k_2(1-\sin\theta_2))(2\sin \theta_1 - 2 \sin \theta_2).
\end{align*}
Remark that the relation
\begin{equation}\label{E:lambdasymmetry}
	\lambda_{9-\ell}(\theta_1,\theta_2,k_1,k_2) = (-1)^{\ell-1}\lambda_{\ell}(\theta_1+\pi,\theta_2+\pi,k_1,k_2)
\end{equation}
holds for $\ell=1,2,3,4$.
When $\theta_1\neq\pi/2, 3\pi/2$ and $\theta_2 \neq \pi/2,3\pi/2$, we introduce new unknowns $v_1$ and $v_2$ by
\[
	v_j = \sqrt{1+\sin \theta_j} u_1 + \sign(\cos \theta_j) \sqrt{1-\sin \theta_j} u_2.
\]
Then, we claim that the system is rewritten as
\begin{equation}\tag{\ref{E:systh1}}
	\left\{
	\begin{aligned}
	(\square + 1) v_1 ={}& \tfrac23k_1(1-\cos (\theta_1 - \theta_2)) v_1^3,\\
	(\square + 1) v_2 ={}& -\tfrac23k_2(1-\cos (\theta_1 - \theta_2)) v_2^3.
	\end{aligned}
	\right.
\end{equation}
Define
\[
	\eta_{j,k}= \eta_{j,k}(\theta_1,\theta_2,k_1,k_2) = \sqrt{1+\sin \theta_j} \l_k + \sign(\cos \theta_j) \sqrt{1-\sin \theta_j} \l_{k+4}
\]
for $j=1,2$ and $k=1,2,3,4$. Then, we obtain
\[
	(\square+1)v_j = \sum_{k=0}^3 \eta_{j,k+1} u_1^{3-k}u_2^{k}.
\]
Remark that $\eta_{1,k}(\theta_1,\theta_2,k_1,k_2)=-\eta_{2,k}(\theta_2,\theta_1,k_2,k_1)$.
Hence, it is suffices to consider the case $j=1$.
Further, $\eta_{1,k}$ is linear with respect to $k_1$ and $k_2$.
One has
\begin{align*}
	\eta_{1,1}(\theta_1,\theta_2,1,0)={}&
	 \tfrac13\sqrt{1+\sin \theta_1}( \cos \theta_1  (\cos \theta_1 (1+\sin \theta_2) - \cos \theta_2 (1+\sin \theta_1) ) \\&
	\qquad +2(1+\sin\theta_1)(\sin \theta_1 - \sin \theta_2))\\
	&+\tfrac13 \sign(\cos \theta_1) \sqrt{1-\sin \theta_1}(1+\sin\theta_1)( \cos \theta_1 (1+\sin \theta_2) - \cos \theta_2 (1+\sin \theta_1))\\
	={}&\tfrac13 (1+\sin \theta_1)^{\frac32}( -\cos \theta_1 \cos \theta_2 +2(\sin \theta_1 - \sin \theta_2) -\cos \theta_1 \cos \theta_2)\\
	&+\tfrac23 \cos^2 \theta_1 (1+\sin \theta_1)^{\frac12}(1+\sin \theta_2)\\
	={}&\tfrac23 (1+\sin \theta_1)^{\frac32} (1- \cos(\theta_1- \theta_2))
\end{align*}
and
\begin{align*}
	\eta_{1,1}(\theta_1,\theta_2,0,1)={}&
	\tfrac13\sqrt{1+\sin \theta_1}( \cos \theta_2  (\cos \theta_1 (1+\sin \theta_2) - \cos \theta_2 (1+\sin \theta_1)) \\&
	\qquad +2(1+\sin\theta_2)(\sin \theta_1 - \sin \theta_2))\\
	&+\tfrac13 \sign(\cos \theta_1) \sqrt{1-\sin \theta_1}(1+\sin\theta_2)(\cos \theta_1 (1+\sin \theta_2) - \cos \theta_2 (1+\sin \theta_1) )\\
	={}&\tfrac13 (1+\sin \theta_1)^{\frac12}(1+\sin \theta_2)(\cos\theta_1 \cos \theta_2 -(1+\sin \theta_1)(1-\sin \theta_2) \\
	&+2(\sin \theta_1 -  \sin \theta_2) + (1-\sin \theta_1)(1+\sin \theta_2)-\cos \theta_1 \cos \theta_2 )\\
	={}&0.
\end{align*}
Hence,
\[
	\eta_{1,1}(\theta_1,\theta_2,k_1,k_2) = \tfrac23 k_1(1+\sin \theta_1)^{\frac32} (1- \cos(\theta_1- \theta_2)).
\]
Similarly,
\begin{align*}
	\eta_{1,2}(\theta_1,\theta_2,1,0)={}&
	\sqrt{1+\sin \theta_1}(1+\sin\theta_1)(\cos \theta_1(1-\sin \theta_2)-\cos \theta_2(1-\sin \theta_1))\\
	&+ \sign(\cos \theta_1) \sqrt{1-\sin \theta_1} \cos \theta_1 (\cos \theta_1 (1+\sin \theta_2) - \cos \theta_2 (1+\sin \theta_1) )\\
	={}& (1+\sin \theta_1)^{\frac12} \cos \theta_1 ((1+\sin\theta_1)(1-\sin\theta_2)-\cos \theta_1 \cos \theta_2\\
	&+(1-\sin\theta_1)(1+\sin\theta_2)-\cos \theta_1 \cos \theta_2)\\
	={}&2 (1+\sin \theta_1)\sign(\cos \theta_1)(1-\sin \theta_1)^\frac12 (1-\cos (\theta_1-\theta_2) )
\end{align*}
and
\begin{align*}
	\eta_{1,2}(\theta_1,\theta_2,0,1)={}&
	\sqrt{1+\sin \theta_1}(1+\sin\theta_2)(\cos \theta_1(1-\sin \theta_2)-\cos \theta_2(1-\sin \theta_1))\\
	&+ \sign(\cos \theta_1) \sqrt{1-\sin \theta_1} \cos \theta_2 (\cos \theta_1 (1+\sin \theta_2) - \cos \theta_2 (1+\sin \theta_1))\\
	={}& (1+\sin \theta_1)^{\frac12} \cos \theta_2 (\cos \theta_1 \cos \theta_2 -(1+\sin\theta_1)(1-\sin\theta_2)\\
	&+(1+\sin\theta_1)(1-\sin\theta_2)-\cos \theta_2 \cos \theta_1)\\
	={}&0.
\end{align*}
Hence,
\[
	\eta_{1,2} (\theta_1,\theta_2,k_1,k_2)=  2k_1 (1+\sin \theta_1)\sign(\cos \theta_1)(1-\sin \theta_1)^\frac12 (1-\cos (\theta_1-\theta_2) )
\]
One sees from \eqref{E:lambdasymmetry} that
\begin{align*}
	\eta_{1,3}(\theta_1,\theta_2,k_1,k_2) ={}& -\sign(\cos \theta_1) \eta_{1,2}(\theta_1+\pi,\theta_2+\pi,k_1,k_2)\\
	={}&2k_1 (1-\sin \theta_1)(1+\sin \theta_1)^\frac12 (1-\cos (\theta_1-\theta_2) )
\end{align*}
and
\begin{align*}
	\eta_{1,4}(\theta_1,\theta_2,k_1,k_2) ={}& \sign(\cos \theta_1) \eta_{1,1}(\theta_1+\pi,\theta_2+\pi,k_1,k_2)\\
	={}&\tfrac23 k_1 (1-\sin \theta_1)^{\frac32}\sign(\cos \theta_1) (1-\cos (\theta_1-\theta_2) ).
\end{align*}
Combining all above, we prove the claim.

When $\theta_1= \pi/2 \neq \theta_2$, we define $v_2$ as above and $v_1= \sqrt2 u_2$.
Similarly, when $\theta_2= \pi/2 \neq \theta_1$, we define $v_1$ as above and $v_2=\sqrt{2} u_1$.
When $\theta_1=3\pi/2$ and $\theta_2 \neq \pi/2,3\pi/2$ and when $\theta_2=3\pi/2$ and $\theta_1 \neq \pi/2,3\pi/2$,
we define $v_1$ and $v_2$ as above.
In all cases, one sees that the system for $(v_1,v_2)$ is the same as above.
\end{proof}

\begin{proof}[Proof of Theorem \ref{T:main1}]
We have seen in Proposition \ref{P:nuA} that $Z_{1,+}$ is invariant.
The  inclusion relation ``$\supset$'' for $Z_{1,+}/{\sim}$ follows from the invariant property
since it is easy to see that $A \in Z_{1,+}$ is true
when $A\in Z$ is the representative of either one of the equivalent
classes in the roster notation.

The  opposite inclusion relation ``$\subset$'' for $Z_{1,+}$
is an immediate consequence of Theorem \ref{T:main1a}.
Remark that a further application of the transform \eqref{E:linearchange} of the form \eqref{E:M1}
makes $\lambda_1,\lambda_8 \in \{0,\pm1\}$ in \eqref{E:systh1}.

Let us next prove $\# (Z_{1,+}/{\sim})=5$, i.e., the equivalent classes in the roster notation 
are mutually disjoint.  
To this end, it suffices to prove that
if $(\l_1,\l_8), (\tilde{\l}_1,\tilde{\l}_8) \in \{0,\pm1\}^2 \setminus \{(0,0)\}$ satisfy
\[
	A:= (\l_1,0,0,0,0,0,0,\l_8) \sim (\tilde{\l}_1,0,0,0,0,0,0,\tilde{\l}_8) =:A'
\]
then $(\l_1,\l_8) = (\tilde{\l}_1,\tilde{\l}_8)$ or $(\l_1,\l_8) = (\tilde{\l}_8,\tilde{\l}_1)$, where we have identified $Z$ and $\R^8$ via \eqref{E:defA}.
Remark that when $(\l_1,\l_8) = (\tilde{\l}_8,\tilde{\l}_1)$ a further application of \eqref{E:linearchange} with \eqref{E:M2}
gives us $(\l_1,\l_8) = (\tilde{\l}_1,\tilde{\l}_8)$.

Suppose that the relation between $A$ and $A'$ is given by
\[
	M= \begin{pmatrix} a & b \\ c & d \end{pmatrix} \in GL_2(\R).
\]
Then, Theorem \ref{T:classA} shows that $A'= (\det M)^{-1} D(M) A D(M)^{-1}$ holds with
\[
	D(M)^{-1} = \frac1{\det M}\begin{pmatrix} a^2 & 2ac & c^2 \\ ab & ad+ bc & cd \\ b^2 & 2bd & d^2 \end{pmatrix}.
\]
Remark that $\nu(A)= \nu (A') = \ltrans{( 0 , 1 , 0)}$.
Together with Proposition \ref{P:dnuA}, we see that the vector $\ltrans{( 0 , 1 , 0)}$ must be an eigenvector of
$\ltrans{D(M)^{-1}}$ associated with a nonzero eigenvalue. 
For this, we need $ab= cd = 0$. Since $ad-bc\neq0$, this implies
\[
	M = \begin{pmatrix} a & 0 \\ 0 & d \end{pmatrix} \text{ or } \begin{pmatrix} 0 & b \\ c & 0 \end{pmatrix}.
\]
In the first case, looking at the transformed system directly, we see that $(\tilde{\l}_1,\tilde{\l}_8) = (a^{-2} \l_1, d^{-2} \l_8)$.
Together with $\l_1, \l_8, \tilde{\l}_1, \tilde{\l}_8 \in \{0,\pm1\}$, we have $(\tilde{\l}_1,\tilde{\l}_8) = ( \l_1, \l_8)$.
Similarly, one has $(\tilde{\l}_1,\tilde{\l}_8) = (b^{-2} \l_8, c^{-2} \l_1)$ in the second case.
Using $\l_1, \l_8, \tilde{\l}_1, \tilde{\l}_8 \in \{0,\pm1\}$, we conclude that $(\tilde{\l}_1,\tilde{\l}_8) = ( \l_8, \l_1)$.
\end{proof}

\subsection{Proof of Theorems \ref{T:main2} and \ref{T:main2a}}

\begin{proof}[Proof of Theorem \ref{T:main2a}]
Suppose that $\rank A =1$ and $\nu(A) \in \mathcal{D} \cup (-\mathcal{D})$.
Without loss of generality, we may suppose that $\nu(A) \in \mathcal{D}$. Define
$\theta\in [0,2\pi)$ by
\[
	\nu(A) = 	\tilde{r}(\theta) \begin{pmatrix} 1-\sin \theta \\
	-2 \cos \theta \\
	 1+\sin \theta
	 \end{pmatrix},
\]
where $\tilde{r}(\theta)=1/\sqrt{5+\cos 2\theta}$ is a normalization factor.
Then, by Proposition \ref{P:dnuA}, 
there exists $(k,\ell) \in \R^2 \setminus\{0\}$ such that
\[
	A= \( k \begin{pmatrix} 1+\sin \theta \\ \cos \theta \\ 1-\sin \theta \end{pmatrix} +  \ell \begin{pmatrix} \cos \theta \\ -\sin \theta \\ -\cos \theta \end{pmatrix}\) \ltrans{\begin{pmatrix} 1-\sin \theta \\
	-2 \cos \theta \\
	 1+\sin \theta
	 \end{pmatrix}} .
\]
In this case, one sees that
\begin{align*}
	\l_1 ={}&\tfrac{1}3(k\cos \theta+\ell \sin \theta) (1+\sin \theta) +\tfrac23 \ell \cos^2 \theta, &
	\l_2 ={}&k(1-\sin^2\theta)+\ell \cos \theta (1-\sin \theta),\\
	\l_3 ={}&k\cos \theta (1-\sin\theta)-\ell\sin \theta (1-\sin \theta) ,&
	\l_4 ={}& \tfrac{k}3(1-\sin \theta)^2 - \tfrac{\ell}3 \cos \theta (1-\sin \theta), \\
	\l_5 ={}& -\tfrac{k}3(1+\sin \theta)^2 - \tfrac{\ell}3 \cos \theta (1+\sin \theta),&
	\l_6 ={}&- k\cos \theta (1+\sin\theta) + \ell\sin \theta (1+\sin \theta), \\
	\l_7 ={}& -k(1-\sin^2\theta)+\ell \cos \theta (1+\sin \theta),& 
	\l_8 ={}& -\tfrac{1}3(k\cos \theta+\ell \sin \theta) (1-\sin \theta) +\tfrac23 \ell \cos^2 \theta.
\end{align*}

We first consider the case $\theta \neq 3\pi/2$, in which case $1+\sin \theta>0$. Introduce
\[
	v_1 := \sqrt{1+\sin \theta} u_1 + \sign (\cos \theta) \sqrt{1-\sin \theta} u_2.
\]
Then,
\[
	(\square +1) v_1 = \sum_{k=0}^3 \eta_{k+1} u_1^{3-k} u_2^{k}, \qquad \eta_k = \sqrt{1+\sin \theta} \lambda_k + \sign (\cos \theta) \sqrt{1-\sin \theta} \lambda_{k+4}.
\]
By exploiting the identity
\[
	\cos \theta = \sign (\cos \theta) \sqrt{1+\sin \theta} \sqrt{1-\sin \theta},
\]
one sees that
\begin{align*}
	\eta_1 ={}& \tfrac{\ell}3 (1+\sin \theta)^{\frac32}, & \eta_2 = {}& \ell \sign(\cos\theta) (1+\sin \theta)(1-\sin \theta)^{\frac12},\\
	\eta_3 = {}& \ell (1+\sin \theta)^{\frac12}(1-\sin \theta), & \eta_4 ={}& \tfrac{\ell}3 \sign(\cos\theta) (1-\sin \theta)^{\frac32}, 
\end{align*}
Indeed,
\begin{align*}
	\eta_1 ={}& \sqrt{1+\sin \theta} [\tfrac{1}3(k\cos \theta+\ell\sin \theta) (1+\sin \theta) +\tfrac23l\cos^2 \theta] \\
	&+ \sign (\cos \theta) \sqrt{1-\sin \theta} [-\tfrac{k}3(1+\sin \theta)^2 - \tfrac{\ell}3 \cos \theta (1+\sin \theta)]\\
	={}&  \tfrac{1}3(k\cos \theta+l\sin \theta) (1+\sin \theta)^{\frac32} +\tfrac23 \ell \cos^2 \theta \sqrt{1+\sin \theta} \\
	&+ \cos \theta [-\tfrac{k}3(1+\sin \theta)^{\frac32} - \tfrac{\ell}3 \cos \theta \sqrt{1+\sin \theta}]\\
	={}&  \tfrac{\ell}3 \sin \theta  (1+\sin \theta)^{\frac32} +\tfrac{\ell}3(1-\sin^2 \theta) \sqrt{1+\sin \theta}\\
	={}&  \tfrac{\ell}3 (1+\sin \theta)^{\frac32} ,
\end{align*}
for instance.
Hence, the equation for $v_1$ now becomes
\begin{equation}\label{E:case2reducedeq1}
	(\square +1) v_1 = \tfrac{\ell}3 v_1^3.
\end{equation}
Next, we put the form
\[
	u_1 = \frac1{\sqrt{1+\sin\theta}}(v_1 - \sign (\cos \theta) \sqrt{1-\sin \theta} u_2)
\]
to the equation for $u_2$ and erase $u_1$. Recall that $1+\sin \theta>0$.
Then, one sees that
\begin{equation}\label{E:case2reducedeq2}
	(\square + 1) u_2 = \tilde{\lambda_5} v_1^3 + \ell v_1^2 u_2,
\end{equation}
where
\[
	\tilde{\lambda}_5 = -\tfrac13 (k \sqrt{ 1+\sin \theta} +\ell \sign(\cos \theta)\sqrt{1-\sin \theta}).
\]

If $\ell=0$ then the system \eqref{E:case2reducedeq1}-\eqref{E:case2reducedeq2} becomes \eqref{E:systh23}
by letting $v_2 = -3/(k\sqrt{1+\sin\theta})u_2$.
Let us consider the case $\ell\neq0$.
Set
\begin{align*}
	v_2 :={}& (k \sqrt{1+\sin \theta} +\ell \sign (\cos \theta) \sqrt{1-\sin \theta})u_1 \\
	&+ (k \sign (\cos \theta) \sqrt{1-\sin \theta} -\ell \sqrt{1+\sin \theta})u_2.
\end{align*}
If $\tilde{\lambda}_5=0$ then
 one sees from \eqref{E:case2reducedeq2} that
\[
	v_2 = (k \sign (\cos \theta) \sqrt{1-\sin \theta} -\ell \sqrt{1+\sin \theta})u_2
\]
solves
\begin{equation}\label{E:case2reducedeq3}
	(\square + 1) v_2 = \ell v_1^2 v_2.
\end{equation}
If $\tilde{\l}_5\neq0$, we have a formula
\[
	v_2 = 	
	(k+\ell\tfrac{\cos \theta}{1+\sin \theta})
		(v_1 + \tfrac{2\ell}{3\tilde{\lambda_5}} u_2),
\]
from which one easily verifies that
$v_2$ solves \eqref{E:case2reducedeq3}.
Thus, the system for $(v_1,v_2)$ is \eqref{E:systh21}.

To complete the proof, we consider the exceptional case $\theta=3\pi/2$.
In this case, we see that $\l_1=\l_2=\l_5=\l_6=\l_7=0$, and
\[
	\l_3 = 2\ell ,\quad \l_4 = \tfrac43 k, \quad \l_8 = \tfrac23 \ell.
\]
In particular, the systems is 
\[
	(\square + 1) u_1 = 2\ell u_1  u_2^2 + \tfrac{4}3 k  u_2^3\ , \qquad (\square + 1) u_2 = \tfrac23 \ell u_2^3.
\]
If $\ell=0$ then the system is \eqref{E:systh23}
by the change $v_1 = \sqrt2 u_2$ and $v_2 = (3/k\sqrt{2})u_1$.
If $\ell \neq0$ then 
the equation for $v_1 = \sqrt2 u_2$ and
$v_2 =  \sqrt2(\ell u_1 + k u_2) $
becomes \eqref{E:systh21}.
\end{proof}

\begin{proof}[Proof of Theorem \ref{T:main2}]
We have seen that $Z_{1,0}$ is invariant in Proposition \ref{P:nuA}.
The  inclusion relation ``$\supset$'' for $Z_{1,0}$ is immediate:
if $\tilde{\l}\neq0$ then
\[
	\rank (\tilde{\l},0,0,0,0,3\tilde{\l},0,0) = \rank (0,0,0,0,1,0,0,0) = 1
\]
and
\[
	\nu ((\tilde{\l},0,0,0,0,3\tilde{\l},0,0)) =\nu ((0,0,0,0,1,0,0,0))= \ltrans{(0 , 0, 1)} \in \mathcal{D},
\]
where we have identified $Z$ and $\R^8$ via \eqref{E:defA}.
Hence the representatives in the roster notation belong to $Z_{1,0}$.

The other inclusion relation ``$\subset$'' is a consequence of Theorem \ref{T:main2a}.
Remark that we may suppose $\ell=1$ in \eqref{E:systh21} 
by a further application of the transform \eqref{E:linearchange} of the form \eqref{E:M1}.

Let us prove the uniqueness property to show that $\# (Z_{1,0}/{\sim})=3$.
Suppose that  
\[
	A , A' \in \{ (\pm 1,0,0,0,0,\pm 3,0,0), \, (0,0,0,0,1,0,0,0)\} \subset Z_{1,0},
\]
where we have identified $Z$ and $\R^8$ via \eqref{E:defA}.
We show that if $A\sim A'$ then $A=A'$.
Suppose that the change between $A$ and $A'$ is given by 
\[
	M= \begin{pmatrix} a & b \\ c & d \end{pmatrix} \in GL_2(\R).
\]
Then, Theorem \ref{T:classA} shows that $A'= (\det M)^{-1} D(M) A D(M)^{-1}$ holds with
\[
	D(M)^{-1} = \frac1{\det M} \begin{pmatrix} a^2 & 2ac & c^2 \\ ab & ad+ bc & cd \\ b^2 & 2bd & d^2 \end{pmatrix}.
\]
Remark that $\nu (A)= \nu (A') = \ltrans{( 0 , 0 , 1)}$.
Together with Proposition \ref{P:dnuA}, we see that
 the vector $\ltrans{( 0 , 0 , 1)}$ is an eigenvector of
$\ltrans{D(M)^{-1}}$. 
For this, we need $b= 0$. Since $ad-bc\neq0$, we have $a\neq 0$.
Then, looking at the transformed system directly, we see that $\tilde{\l}_1 = a^2 \l_1$,
where $\l_1$ and $\tilde{\l}_1$ are first component of $A ,A'$ in the $\R^8$ representation, respectively. 
As $a\neq0$ and $\l_1, \tilde{\l}_1 \in \{0,\pm1\}$, we conclude that $\tilde{\l}_1 =  \l_1$, proving $A=A'$.
\end{proof}

\subsection{Proof of Theorems \ref{T:main5} and \ref{T:main5a}}

\begin{proof}[Proof of Theorem \ref{T:main5a}]
(1)
First remark that the matrix $\tilde{A}$ defined via \eqref{E:defA} from \eqref{E:systh5} is written as
\[
	\tilde{A} = - 3r\( \cos \theta \begin{pmatrix} 0 \\
	1\\
	0
	\end{pmatrix} + \sin \theta \begin{pmatrix} 1\\
	0\\
	-1
	\end{pmatrix}\)
	\begin{pmatrix} 1 &
	0 &
	1
	\end{pmatrix}.
\] 
Hence, being such a matrix is characterized as the validity of $\nu (\tilde{A}) = (2^{-1/2}) \ltrans{( 1 , 0 , 1)}$.

Suppose that $\rank A = 1$ and $\nu (A) = \ltrans{(a_0 , b_0 , c_0 )} \in S_2 \cup (-S_2)$.
We have $4a_0c_0>b_0^2 \ge 0$. In particular, $a_0\neq0$, $c_0 \neq 0$, and $\sign a_0 = \sign c_0$.
The goal is to find a change of unknowns of the form \eqref{E:linearchange} which makes the transformed matrix $\tilde{A}$
satisfy $\nu (\tilde{A}) = (2^{-1/2}) \ltrans{( 1 , 0 , 1)}$.
As seen in the proof of Theorem \ref{T:main1a},
if we apply the change of unknowns of the form \eqref{E:linearchange} with \eqref{E:M1}, \eqref{E:M2}, or \eqref{E:M3} to the system \eqref{E:sys} corresponding to $A$,
then the corresponding $\nu(A)$ turn into the normalization of $\ltrans{D^{-1}}\nu(A)$ with $D^{-1}$ given by
\eqref{E:D1}, \eqref{E:D2}, or \eqref{E:D3}, respectively.
When $b_0\neq0$, one has the identity
\[
	\ltrans{\begin{pmatrix} 1 & 0 & 0 \\ 0 & -\frac{\sqrt{4a_0c_0-b_0^2}}{b_0} & 0 \\ 0 & 0 & \frac{4a_0c_0-b_0^2}{b_0^2} \end{pmatrix}} \ltrans{\begin{pmatrix} 1 & 0 & 0 \\ 1 & 1 & 0 \\ 1 & 2 & 1 \end{pmatrix}} \ltrans{\begin{pmatrix} 1 & 0 & 0 \\ 0 & -\frac{b_0}{2c_0} & 0 \\ 0 & 0 & (-\frac{b_0}{2c_0})^2 \end{pmatrix}} \nu(A)
	= \frac{4a_0c_0-b_0^2}{4c_0} \begin{pmatrix} 1 \\ 0 \\ 1 \end{pmatrix}
\]
In view of \eqref{E:DM-1form} and the identity $\ltrans{D(M_1)^{-1}} \ltrans{D(M_2)^{-1}} = \ltrans{D(M_1M_2)^{-1}}$,
the above computation suggests to take
\[
	M = \begin{pmatrix} -1 & 0 \\ 0  & b_0^{-1} \sqrt{4a_0c_0-b_0^2} \end{pmatrix} \begin{pmatrix} 1 & 1 \\ 0  & 1\end{pmatrix}\begin{pmatrix} -2 c_0 & 0 \\ 0  & b_0\end{pmatrix}
	= \begin{pmatrix}  2c_0 & - b_0 \\ 0 & \sqrt{4a_0c_0-b_0^2} \end{pmatrix} \in GL_2(\R)
\]
since Corollary \ref{C:nuA} implies that the matrix has the desired property, i.e., this yields $\nu(\tilde{A}) = (2^{-1/2}) \ltrans{( 1 , 0 , 1)}$.
Similarly, when $b_0=0$, we have
\[
	\ltrans{\begin{pmatrix} 0 & 0 & -1 \\ 0 & -1 & 0 \\ -1 & 0 & 0 \end{pmatrix}}
	\ltrans{\begin{pmatrix} 1 & 0 & 0 \\ 0 & \pm \sqrt{a_0/c_0} & 0 \\ 0 & 0 & a_0/c_0 \end{pmatrix}}
	\ltrans{\begin{pmatrix} 0 & 0 & -1 \\ 0 & -1 & 0 \\ -1 & 0 & 0 \end{pmatrix}}
	 \nu(A)
	= c_0 \begin{pmatrix} 1 \\ 0 \\ 1 \end{pmatrix}.
\]
Hence, the change \eqref{E:linearchange} with the matrix
\[
	M = \begin{pmatrix} 0 & 1 \\  1 & 0 \end{pmatrix}\begin{pmatrix} 2c_0 & 0 \\ 0  & \sqrt{4a_0c_0} \end{pmatrix}
	\begin{pmatrix} 0 & 1 \\  1 & 0 \end{pmatrix}
	= \begin{pmatrix} 2c_0 & 0 \\ 0  & \sqrt{4a_0c_0} \end{pmatrix}
\]
works for both $a_0>0$ and $a_0<0$.

(2) Suppose that $(w_1,w_2)$ solves \eqref{E:systh5} for some $r>0$ and $\theta \in \R$. Then,
$w=w_1 + i w_2$ solves
\[
	(\square + 1)  w = r e^{i\theta} w^3.
\]
Multiplying the both sides by $r^{1/2} e^{i\theta/2}$, we have
\[
	(\square + 1)  (r^{1/2} e^{i\theta/2} w) = (r^{1/2} e^{i\theta/2} w)^3.
\]
This implies that
\[
	\begin{pmatrix}
	v_1 \\ v_2
	\end{pmatrix}
	:=
	\begin{pmatrix}
	\Re (r^{1/2} e^{i\theta/2} w) \\ \Im (r^{1/2} e^{i\theta/2} w)
	\end{pmatrix}
	=
	\begin{pmatrix} r^{\frac12} \cos \frac{\theta}2 & -r^{\frac12} \sin \frac{\theta}2 \\  r^{\frac12} \sin \frac{\theta}2 &  r^{\frac12} \cos \frac{\theta}2 \end{pmatrix} \begin{pmatrix} w_1 \\ w_2 \end{pmatrix}
\]
solves \eqref{E:sysnew2}.
\end{proof}

\begin{proof}[Proof of Theorem \ref{T:main5}]
The invariant property of $Z_{1,-}$ is shown in Proposition \ref{P:nuA}.
The inclusion relation ``$\supset$'' follows from the proposition.
We note that
\[
	\rank ((1 ,0,-3 ,0,0,3,0 ,-1)) = 1
\]
and
\[
	\nu ((1 ,0,-3 ,0,0,3,0 ,-1))
	= (2^{-1/2}) \ltrans{( 1 , 0 , 1)} \in S_2
\]
hold. Hence, the representative in the roster notation is in $Z_{1,-}$.

The opposite inclusion relation ``$\subset$'' follows from Theorem \ref{T:main5a}.
\end{proof}

\section{Proofs of classification theorems for $Z_2$}\label{S:main3}

\subsection{On the condition \eqref{E:main3cond}}
Let us briefly explain  
how the assumption \eqref{E:main3cond} comes into play.
As seen in Subsection \ref{subs:good},
when $\rank A =1$, the existence of two conserved quantities of the form
\begin{equation}\tag{\ref{E:conservedQ}}
	a |\alpha_1|^2 + 2 b \Re (\overline{\alpha_1}\alpha_2) + c |\alpha_2|^2
\end{equation}
to the limit ODE system is crucial.
However, the ODE has only one conserved quantity of this form when $\rank A =2$.
To make up for the lack, we suppose that we have another conserved quantity 
\begin{equation}\tag{\ref{E:conservedQ2}}
	2\Im (\overline{\alpha_1}\alpha_2 ).
\end{equation}
As discussed in Subsection \ref{subsec:another_conservation} below,
the conservation of $2\Im (\overline{\alpha_1}\alpha_2 )$ requires the condition
\begin{equation}\tag{\ref{E:main3cond}}
	\l_4=\l_5=0,\quad \l_6=\l_1,\quad \l_7=\l_2, \quad\l_8=\l_1.
\end{equation}

According to the discussion in Subsection \ref{subs:good},
a variable which has a conserved modulus will be a good variable to work with.
As for the conserved quantity of the form \eqref{E:conservedQ}, this is the case when $b^2=ac$.
In the case $\rank A=1$, we tried to obtain such a conserved quantity by considering linear combinations of two linearly independent solutions to \eqref{E:Aeq}.
The argument does not work here because we have only one solution to \eqref{E:Aeq}. 
We use the following:
If we have a conserved quantity of the form \eqref{E:conservedQ} with
$b^2<ac$ and if $\Im (\overline{\alpha_1}\alpha_2 )$ is conserved in addition then 
\begin{align*}
	&a |\alpha_1|^2 + 2 b \Re (\overline{\alpha_1}\alpha_2) + c |\alpha_2|^2 + 2\sqrt{ac-b^2}\Im (\overline{\alpha_1}\alpha_2 )\\
	&= (\sign a )|\sqrt{|a|}\alpha_1 + \gamma \alpha_2|^2
\end{align*}
becomes a conserved quantity,
where $\gamma = (\sign a)(b- i\sqrt{ac-b^2})/\sqrt{|a|}\in \C\setminus\R$.
This suggests that
\[
	 \sqrt{|a|} u_1 \pm \gamma u_2
\]
are good variables to work with.

As seen in the actual proof below, once we suppose \eqref{E:main3cond}, the change of variable is rather obvious without the above
discussion of the conserved quantities for the limit ODE system.

\subsection{Proof of Theorems \ref{T:main3} and \ref{T:main3a}}
\begin{proof}[Proof of Theorems \ref{T:main3} and \ref{T:main3a}]
Under the assumption \eqref{E:main3cond},
the system \eqref{E:sys} becomes
\begin{equation*}
\left\{
\begin{aligned}
&(\square + 1)u_1 
= (\l_1u_1^2 + \l_2u_1u_2 + \l_3u_2^2)u_1, \\
&(\square + 1)u_2 
= (\l_1u_1^2 + \l_2u_1u_2 + \l_3u_2^2)u_2.
\end{aligned}
\right.
\end{equation*}
If we regard $V:=\l_1u_1^2 + \l_2u_1u_2 + \l_3u_2^2$ as a linear potential
then the above system is a system of two linear equations
\begin{equation*}
\left\{
\begin{aligned}
&(\square + 1)u_1 
= V u_1, \\
&(\square + 1)u_2 
= V u_2.
\end{aligned}
\right.
\end{equation*}
This shows that, for any change of variable of the form \eqref{E:linearchange}, the transformed equation becomes
\begin{equation*}
\left\{
\begin{aligned}
&(\square + 1)v_1 
= V v_1, \\
&(\square + 1)v_2 
= V v_2
\end{aligned}
\right.
\end{equation*}
with the same $V$. Since $V$ is written as $V=\tilde{\l}_1v_1^2 + \tilde{\l}_2 v_1v_2 + \tilde{\l}_3v_2^2$ in terms of the new unknowns,
the validity of the condition \eqref{E:main3cond} is left invariant by any change of variable of the form \eqref{E:linearchange}.

Now, notice that $V=\l_1u_1^2 + \l_2u_1u_2 + \l_3u_2^2$ is a quadratic form with respect to $(u_1,u_2)$.
Hence, the theorems are immediate consequences of the theory of quadratic form.
The reduced system $(\pm1,0,\pm1,0,0,\pm1,0,\pm1)$ correspond to the reduced quadratic forms with sign $(+,+)$ and $(-,-)$, respectively.
Similarly, the systems $(\pm1,0,0,0,0,\pm1,0,0)$ correspond to the form with sign $(+,0)$ and $(-,0)$, respectively
and $(1,0,-1,0,0,1,0,-1)$ to the form with sign $(+,-)$.
\end{proof}

\section{Preliminary for the study of large time behavior}\label{S:pre}

In this section we briefly summarize preliminaries for the study of large time behavior of solutions
to \eqref{E:sys}.
We also consider the derivation of  the limit ODE system \eqref{E:limitODE}.

\subsection{Local existence}

We employ the following local existence result by Delort \cite[Proposition 1.4]{Del}.

\begin{proposition}[Local existence]\label{P:hyp}
Let $m\ge3$ be an integer. 
Suppose that $(u_{j,0},u_{j,1})\in H^{m}(\R)\times H^{m-1}(\R)$ is compactly supported.
Let $B$ be the number such that \eqref{E:support} holds. Fix $\tau_{0}>\max\{1,2B\}$. 
Then there exists $\eta_{0}>0$ 
with the following properties:
If $\eta:=\sum_{j=1}^2(\|u_{j,0}\|_{H^{m}}+\|u_{j,1}\|_{H^{m-1}})\le\eta_{0}$
then there exist $T\ge(\tau_{0}^{2}-3B^{3})/(2B)$ 
and unique solution $(u_1,u_2)$ to \eqref{E:sys} satisfying 
$u_j\in C([0,T];H^{m}(\R))\cap C^{1}([0,T];H^{m-1}(\R))$ ($j=1,2$). 
Especially, $(u_1,u_2)$ is defined on the curve $\{(t,x)\in\R^{2}||x|\le(\tau_{0}^{2}-B^{2})/(2B), 
(t+2B)^{2}-|x|^{2}=\tau_{0}^{2}\}$ and the solution obeys the following bound on the curve: 
\begin{eqnarray*}
\sum_{0\le k+\ell\le m} \sum_{j=1}^2
\int_{|y|\le\frac{\tau_{0}^{2}-B^{2}}{2B}}\left|\pt_{t}^{k}\pt_{x}^{\ell}u_j\bigl|_{(t,x)=(\sqrt{\tau_{0}^{2}+y^{2}}-2B,y)}\right|^{2}
dy\le C\eta^{2}.
\end{eqnarray*}
Furthermore, the solution has finite propagation speed:
\[
	\supp u_1(t) \cup \supp u_2 (t) \subset\{x\in \R\ |\ |x|\le t+B\}
\] for $0\le t\le T$. 
\end{proposition}

\begin{proof}[Proof of Proposition \ref{P:hyp}] 
By the standard local existence theorem for the nonlinear Klein-Gordon 
equation (see \cite{Hor} for instance), 
there exist $T=O(1/\eta)$ and unique solution $(u_1,u_2)$ to \eqref{E:sys} 
satisfying $u_j\in C([0,T];H^{m}(\R))\cap C^{1}([0,T];H^{m-1}(\R))$  
such that 
\begin{eqnarray*}
\sup_{0\le t\le T}(\|u_j(t)\|_{H^{m}}+\|\pt_{t}u_j(t)\|_{H^{m-1}})
\le C\eta \quad \text{for}\ j=1,2.
\end{eqnarray*}
Furthermore, by the property of finite speed of propagation   
for the solution to \eqref{E:sys}, we have 
$\text{supp}\ u_j(t)\subset\{x\in\R\ |\ |x|\le t+B\}$ for $0\le t\le T$. 
Choosing $\eta>0$ sufficiently small, we can take 
$T\ge(\tau_{0}^{2}-3B^{2})/(2B)$. Then we have the conclusion. 
\end{proof}

\subsection{Hyperbolic coordinate}

Let $(u_1,u_2)$ be a solution given by Proposition \ref{P:hyp}.
By Proposition \ref{P:hyp} and the property of finite speed of propagation 
for the solution to  \eqref{E:sys}, 
to obtain a global solution to \eqref{E:sys}
it suffices to consider it on the domain 
\begin{equation}\label{E:domain}
D:=\{(t,x) \in \R\times \R\ ;\ t \ge 0, (t+2B)^2-|x|^2 \ge \tau_0^2\}.
\end{equation}
We introduce the hyperbolic coordinates 
\[
\tau = \sqrt{(t+2B)^2 - x^2}, \quad z = \tanh^{-1}\left(\frac{x}{t+2B}\right),
\]
or equivalently,
\begin{equation}\label{E:hypcoord}
t + 2B = \tau \cosh z, \quad x = \tau \sinh z, \quad \tau \ge \tau_0,\ z \in \R.
\end{equation}
Then it holds that 
\begin{align*}
\partial_t &= (\cosh z)\partial_\tau - \frac{\sinh z}{\tau}\partial_z,\quad
\partial_x = -(\sinh z)\partial_\tau + \frac{\cosh z}{\tau}\partial_z,\\
\square &= \frac{\partial^2}{\partial \tau^2} 
+ \frac{1}{\tau}\frac{\partial}{\partial \tau} - \frac{1}{\tau^2}\frac{\partial^2}{\partial z^2}.
\end{align*}
For a solution $(u_1,u_2)$ to \eqref{E:sys}, 
we introduce new unknowns
$U_j = U_j(\tau,z)$ for $j = 1,2$ by 
\begin{equation}
u_j(t,x) = \frac{1}{\tau^{1/2}\cosh \kappa z}U_j(\tau,z),
\label{eq:2.27}
\end{equation}
where $\kappa$ is a positive constant to be fixed later. 
It will turn out that we would need $\kappa>7/2$.
It follows from \eqref{eq:2.27} that
\begin{equation}
\label{eq:2.29}
\begin{aligned}
\partial_\tau^2U_j + U_j 
= \frac{1}{\tau(\cosh \kappa z)^2}F_j(U_1,U_2) 
+ \frac{1}{\tau^2}PU_{j}
\end{aligned}
\end{equation}
($j=1,2$),
where 
 $P$ is a differential operator defined by 
\begin{equation}
\label{eq:2.28}
P = \partial_z^2-2\kappa(\tanh \kappa z)\partial_z 
-\frac{2\kappa^2}{(\cosh \kappa z)^2}-\frac{1}{4}+\kappa^2.
\end{equation}
and
\begin{equation}\notag
\begin{aligned}
F_1(U_1,U_2) &= \lambda_1U_1^3 + \lambda_2U_1^2U_2 + \l_3U_1U_2^2+\l_4U_2^3, \\
F_2(U_1,U_2) &= \lambda_5U_1^3 + \lambda_6U_1^2U_2+\l_7U_1U_2^2+\l_8U_2^3.
\end{aligned}
\end{equation}

\subsection{Reduction to the limit ODE system}\label{subsec:limitODEsys}
We would like to look at the behavior of $(U_1,U_2)$. 
To this end, we follow the argument in \cite{Del,Su3,S}.
Let us introduce a $\C^2$-valued function 
$\alpha = (\alpha_1,\alpha_2)$ by 
\[
	U_j(\tau,z) = 2 \Re (\alpha_j(\tau,z) e^{i\tau}), \quad
	\partial_\tau U_j(\tau,z) = -2 \Im (\alpha_j(\tau,z) e^{i\tau})
\]
or more explicitly by
\begin{equation}
\alpha_j(\tau,z) := \frac{1}{2}\left(
U_j(\tau,z) -i\partial_\tau U_j(\tau,z)\right)e^{-i\tau}
\label{eq:2.30}
\end{equation}
for $j = 1,2$. 

Let us find the system for $(\alpha_1,\alpha_2)$.
In one hand,
by definition of $\alpha_j$, we have the following identity for the linear part: 
\begin{equation}
\label{eq:2.31}
\begin{aligned}
\frac{\partial \alpha_j}{\partial \tau}(\tau,z)
&= -\frac{i}{2}e^{-i\tau}(\partial_\tau^2U_j(\tau,z) + U_j(\tau,z)).
\end{aligned}
\end{equation}
On the other hand, we are able to express 
the nonlinearities by means of $\alpha_1$ and $\alpha_2$ as follows:
\begin{equation}
\label{eq:2.32}
\begin{aligned}
F_j(U_1,U_2) &=(3\l_{4j-3}|\a_1|^2{\a_1}+\l_{4j-2}(2|\a_1|^2\a_2+\a_1^2\ol{\a_2})+
\l_{4j-1}(2\a_1|\a_2|^2+\ol{\a_1}\a_2^2)+3\l_{4j} |\a_2|^2 \a_2 )e^{i\tau}\\
&\quad + (3\l_{4j-3} |\a_1|^2 \ol{\a_1} +\l_{4j-2}\ol{\a_1}^2\a_2
+\l_{4j-1} \a_1\ol{\a_2}^2+3\l_{4j} |\a_2|^2 \ol{\a_2})e^{-i\tau}\\
&\quad + (\l_{4j-3} \a_1^3+\l_{4j-2} \a_1^2\a_2+\l_{4j-1} \a_1\a_2^2+\l_{4j} \a_2^3
)e^{3i\tau}\\
&\quad + (\l_{4j-3} \ol{\a_1}^3+\l_{4j-2} \ol{\a_1}^2\ol{\a_2}+\l_{4j-1} \ol{\a_1}\ol{\a_2}^2
+ \l_{4j} \ol{\a_2}^3)e^{-3i\tau}
\end{aligned}
\end{equation}
for $j=1,2$.
Combining \eqref{eq:2.29}, \eqref{eq:2.31} and \eqref{eq:2.32}, 
we have the system for $(\alpha_1,\alpha_2)$:
\begin{equation}
\label{eq:2.34}
\begin{aligned}
&2i\frac{\partial \alpha_j}{\partial \tau}(\tau,z)\\
&= 
\frac{1}{\tau(\cosh \kappa z)^2}
(3\l_{4j-3} |\a_1|^2 \a_1+\l_{4j-2} (2|\a_1|^2\a_2+\a_1^2\ol{\a_2})+
\l_{4j-1}(2\a_1|\a_2|^2+\ol{\a_1}\a_2^2)+3\l_{4j}|\a_2|^2\a_2)\\
&\quad +
\frac{1}{\tau}N_{j,\text{nr}}(\alpha_{1},\alpha_{2})
+e^{-i\tau} \frac{1}{\tau^2}PU_{j}
\end{aligned}
\end{equation}
for $j=1,2$, where $N_{j,\text{nr}}$ is a ``non-resonant'' term given by 
\begin{equation}
\label{eq:2.35}
\begin{aligned}
N_{j,\text{nr}}(\alpha_{1},\alpha_{2}) :=& \frac{1}{(\cosh \kappa z)^2}\bigg\{
(\l_{4j-3} \a_1^3+\l_{4j-2} \a_1^2\a_2+\l_{4j-1} \a_1\a_2^2+\l_{4j} \a_2^3)e^{2i\tau}
\\
&\qquad + (3\l_{4j-3} |\a_1|^2\ol{\a_1} +\l_{4j-2} \ol{\a_1}^2\a_2
+\l_{4j-1} \a_1\ol{\a_2}^2+3\l_{4j} |\a_2|^2 \ol{\a_2})e^{-2i\tau}\\
&\qquad+ (\l_{4j-3} \ol{\a_1}^3+\l_{4j-2} \ol{\a_1}^2\ol{\a_2}+\l_{4j-1} \ol{\a_1}\ol{\a_2}^2
+ \l_{4j} \ol{\a_2}^3)e^{-4i\tau}\bigg\}
\end{aligned}
\end{equation}
for $j=1,2$.

As seen below, the third term in the right hand side of \eqref{eq:2.34} can be regarded as a remainder term
because of the factor $1/\tau^2$.
We will see that the non-resonant term is also a remainder term.
To this end, 
we need to take the oscillation factors $e^{i\omega \tau}$ into account.
Thus, at least formally, the non-resonant term and the remainder term in \eqref{eq:2.34} are negligible and so
the asymptotic behavior of solutions to the system \eqref{eq:2.34} is described by the limit ODE system
\begin{equation}\tag{\ref{E:limitODE}}
\left\{
\begin{aligned}
	&2i\frac{\partial \alpha_1}{\partial s}=
	3\l_{1} |\a_1|^2\a_1+\l_{2} (2|\a_1|^2\a_2+\a_1^2\ol{\a_2})+
\l_{3}(2\a_1|\a_2|^2 +\ol{\a_1}\a_2^2)+3\l_{4}|\a_2|^2\a_2,\\
	&2i\frac{\partial \alpha_2}{\partial s}=
	3\l_{5} |\a_1|^2\a_1+\l_{6} (2|\a_1|^2\a_2+\a_1^2\ol{\a_2})+
\l_{7}(2\a_1|\a_2|^2 +\ol{\a_1}\a_2^2)+3\l_{8}|\a_2|^2\a_2,
\end{aligned}
\right.
\end{equation}
where $\alpha_j :\R \to \C$ ($j=1,2$) are complex-valued unknowns.
The system is obtained by dropping the non-resonant term and the remainder term
and by introducing $s=(\cosh \kappa z)^{-2}\log \tau $. 

\subsection{Conserved quantity for the limit ODE system}\label{subsec:conserved_quantity}
The matrix $A$ given in \eqref{E:defA}
is related with a conserved quantity of the limit ODE system \eqref{E:limitODE}.
We will see that the quantity plays an important role
in finding the explicit solutions to the limit ODE system.
Let us next look at the quantity.
In this and the next subsection, we suppose that $(\alpha_1,\alpha_2)$ is a smooth solution to \eqref{E:limitODE}.

It follows from \eqref{E:limitODE} that 
\begin{align*}
\pa_s|\a_1|^2 
&= \l_2|\a_1|^2\Im(\ol{\a_1}\a_2)+ \l_3\Im(\ol{\a_1}^2\a_2^2)+3\l_4|\a_2|^2\Im(\ol{\a_1}\a_2)
\end{align*}
and 
\begin{align*}
\pa_s |\a_2|^2
&= -3\l_5|\a_1|^2\Im(\ol{\a_1}\a_2)-\l_6\Im(\ol{\a_1}^2\a_2^2)
-\l_7\Im(\ol{\a_1}\a_2)|\a_2|^2.
\end{align*}
Similarly, we have 
\begin{align*}
\pa_s (2\Re(\ol{\a_1}\a_2))
&= (\l_6-3\l_1)|\a_1|^2\Im(\ol{\a_1}\a_2)
+(\l_7-\l_2)\Im(\ol{\a_1}^2\a_2^2)
+(3\l_8-\l_3)\Im(\ol{\a_1}\a_2)|\a_2|^2.
\end{align*}
Thus, we have the following identity: For constants $a,b,c \in \R$,
\begin{equation}
\label{eq:2.38}
\begin{aligned}
&\frac{\pa}{\pa s}
\left(
a|\alpha_1|^2+2b\Re(\overline{\alpha}_1\alpha_2)+c|\alpha_2|^2
\right)\\
&= |\alpha_1|^2\Im(\overline{\alpha}_1\alpha_2)
(a\lambda_2+b(\lambda_6-3\lambda_1)-3c\lambda_5)\\
&\quad + \Im(\ol{\alpha}_1^2\alpha_2^2)
(a\lambda_3+b(\lambda_7-\lambda_2)-c\lambda_6)\\
&\quad +|\alpha_2|^2\Im(\overline{\alpha}_1\alpha_2)
(3a\lambda_4+b(3\lambda_8-\lambda_3)-c\lambda_7).
\end{aligned}
\end{equation}
Hence, if the triplet $(a,b,c)$ satisfies
\begin{equation}
\left\{
\begin{aligned}
a\la_2+b(\la_6-3\la_1)-3c\la_5 &= 0, \\
a\la_3+b(\la_7-\la_2)-c\la_6 &= 0,\\
3a\la_4+b(3\la_8-\la_3)-\la_7 &= 0, 
\end{aligned}
\right.
\end{equation}
then 
\begin{equation}\label{E:conservedQ}
	a|\alpha_1|^2+2b\Re(\overline{\alpha}_1\alpha_2)+c|\alpha_2|^2
\end{equation}
is a conserved quantity for \eqref{E:limitODE}.
The condition reads as 
\begin{equation*}
\begin{pmatrix}
\la_2 & \la_6 - 3\la_1 & -3\la_5 \\
\la_3 & \la_7 - \la_2 & -\la_6\\
3\la_4 & 3\la_8-\la_3 & -\la_7
\end{pmatrix}
\begin{pmatrix}
a \\ b \\ c
\end{pmatrix}
= 0,
\end{equation*}
of which coefficient matrix is nothing but the matrix $A$ defined by \eqref{E:defA}. 
The equation reads as $\ltrans{(a,b,c)} \in \ker A$.

\subsection{Another conserved quantity for the limit ODE}\label{subsec:another_conservation}

We also consider the following quantity:
\begin{equation}\label{E:conservedQ2}
	2\Im (\overline{\alpha_1}\alpha_2 ).
\end{equation}
A computation shows that 
\begin{equation}
\label{eq:2.42}
\begin{aligned}
\pa_s (2\Im(\ol{\a_1}\a_2))
&= 3|\a_1|^2\Re(\ol{\a_1}\a_2)(\l_1-\l_6)
+2|\a_1|^2|\a_2|^2(\l_2-\l_7)\\
&\quad + \Re(\ol{\a_1}^2\a_2^2)(\l_2-\l_7)
+3|\a_2|^2\Re(\ol{\a_1}\a_2)(\l_3-\l_8)\\
&\quad +3\l_4|\a_2|^4-3\l_5|\a_1|^4.
\end{aligned}
\end{equation}
Hence,
it becomes a conserved quantity if and only if
\begin{equation*}
	\l_4=\l_5=0,\quad \l_6=\l_1,\quad \l_7=\l_2, \quad\l_8=\l_3.
\end{equation*}
This is the main part of the condition \eqref{E:main3cond}.
Notice that the conservation of \eqref{E:conservedQ} and of \eqref{E:conservedQ2} are independent in such a sense that
the conservation of a linear combination of them (with nonzero constants)
holds only if these quantities are both conserved.
This is because the terms in the right hand side of \eqref{eq:2.38} and \eqref{eq:2.42} are different.

\begin{remark}
The validity of \eqref{E:main3cond} implies $\rank A = 2$.
Hence, if $\rank A=3$ then neither \eqref{E:conservedQ} nor \eqref{E:conservedQ2} is not a conserved quantity.
This is one difficulty of the case $\rank A=3$.
\end{remark}

\section{Proof of Theorem \ref{T:main4}}\label{S:main4}

Now let us restrict ourselves to the case
\[
	\l_2=\l_3=\l_4=\l_5=\l_7=\l_8=0
\]
and either
\[
	\l_6 = 3\l_1 \quad \text{ or } \quad \l_6= \l_1.
\]
Then, the system \eqref{E:sys} becomes
\begin{equation}\tag{\ref{E:sysnew}}
\left\{
\begin{aligned}
&(\square + 1)u_1 
= \l_1 u_1^3, \\
&(\square + 1)u_2 
= \l_6 u_1^2u_2.
\end{aligned}
\right.
\end{equation}

\subsection{Global existence and global bounds}

We first consider the global existence of a solution.
Recall that the first equation of the system \eqref{E:sysnew} is nothing but the single equation \eqref{E:cNLKG}.
Hence, the small data global existence is well-known (see e.g. \cite{Del,HN08ZAMP,LS,Sti}).
Once we know the global existence of $u_1$, it is straightforward to see that $u_2$ also exists globally
since the equation for $u_2$ is linear with respect to $u_2$.

The main concern of the present subsection is global bounds on the solution.
In particular, we control growth of Sobolev norms of solutions.

\subsubsection{Estimate of the first component}

Let us begin with the estimate of the first component of \eqref{eq:2.29} corresponding to  \eqref{E:sysnew}.
In the present paper, we use the following version of the global bound result.
\begin{proposition}
\label{P:ap} Let $\kappa \in \R_+ \setminus \frac12 \Z$
and let\footnote{
$\lfloor x \rfloor = \max \{ z \in \Z \ |\ z \le x \}$ is the floor function.}
 $M=\lfloor2\kappa\rfloor+1$. 
 Let $L \ge 2$ be an integer.
For any $\delta \in (0,1/5)$ there exists $\eta_{1}>0$ such that 
if $\eta:=\|U_{1,0}\|_{H_z^{L+M+1}}+\|U_{1,1}\|_{H_z^{L+M}}\le\eta_{1}$, 
then there exists a unique global solution $U_1$ to (the first component of) \eqref{eq:2.29} 
with $U_{1}(\tau_0,z)=U_{1,0}(z), \pt_{\tau}U_{1}(\tau_0,z)=U_{1,1}(z)$ satisfying
$U_1\in C([\tau_0,\infty),H_z^{L+M+1}(\R))\cap C^1([\tau_0,\infty),H_z^{L+M}(\R))$ 
and estimates
\begin{equation}
\| \pt_{\tau}^{k_1} \pt_{z}^{k_2} U_1(\tau,\cdot)\|_{L^2_z(\R)}
\lesssim_\delta \eta \tau^{\tilde{\delta}(k_2)}  \label{ve}
\end{equation}
for any $k_1,k_2\ge0$ with
$0\le k_1+k_2\le L+M+1$ and $k_2 \le M+L$,
\[
\| \pt_{z}^{L+M+1} U_1(\tau,\cdot)\|_{L^2_z(\R)}
\lesssim_\delta \eta \tau^{\tilde{\delta}(M+L+1) + \frac12},
\]
 and
\begin{equation}\label{ve1.5}
	\sup_{\tau \ge \tau_0}\norm{ U_1 (\tau,\cdot) }_{L^\I_z(\R)} \lesssim_\kappa \delta^{\frac12},
\end{equation}
 where
\[
	\tilde{\delta}(k) := \delta + \frac12\max ( k-L ,0).
\]
\end{proposition}

\begin{remark}
As mentioned above, we will need $\kappa>7/2$.
Hence, we choose $\kappa=701/200$ in our theorem, which implies $M=8$.
For the choice of $\kappa$, if $\|u_{1,0}\|_{H^{L+9}}+
\|u_{1,1}\|_{H^{L+8}}$ is sufficiently small, 
then Proposition \ref{P:hyp} ensures that the assumption of Proposition \ref{P:ap} is fulfilled.
In this subsection, we consider general $\kappa$.
When $2\kappa \in \Z$, we have a similar result with a modification.
\end{remark}

\begin{proof}[Proof of Proposition \ref{P:ap}]
We briefly outline the proof.  Let $\tau_0=1$, for simplicity.
We employ the energy estimates used by Delort, Fang and Xue \cite{DFX}. 

Let 
\[
	E_m[U](\tau) := \frac12 \sum_{j=0}^m \( \norm{\pt_z^j \pt_\tau U}_{L^2_z(\R)}^2+ \frac1{\tau^2}\norm{\pt_z^{j+1} U}_{L^2_z(\R)}^2 +  \norm{\pt_z^j U}_{L^2_z(\R)}^2 \).
\]
By a standard energy argument, one has
\[
	\frac{\pt E_m[U_1]}{\pt \tau} (\tau)\le \( \frac{C_0}{\tau} \norm{U_1(\tau)}_{L^\I_z(\R)}^2 + \frac{C_1}{\tau^2}\) E_m[U_1](\tau) + 2\kappa \min \( \frac1{\tau} E_{m}[U_1](\tau), \frac1{\tau^2}E_{m+1}[U_1](\tau) \)
\]
for $0\le m \le L+M$ with positive constants $C_0(\tau, L)\ge1$ and $ C_1(\tau, L)$.

Let $T>1$.
Suppose that $U_1(\tau)$ can be extended upto $\tau=T$ and that
\begin{equation}\label{ve3}
	C_0 \norm{U_1(\tau)}_{L^\I_z(\R)}^2 \le 2\delta
\end{equation}
is true for $1 \le \tau \le T$. Applying the Gronwall inequality to
the above estimate with $m=L+M$, one obtains
\[
	E_{L+M} [U_1] (\tau) \le E_{L+M} [U_1](1) e^{C_1(1-\tau^{-1})}  \tau^{2\delta+ 2\kappa} \lesssim 
	\eta^2 \tau^{2\delta+ 2\kappa}
\]
for $1 \le \tau \le T$.
Then, by induction on $n$, we obtain
\[
	E_{L+M-n}[U_1] (\tau) \le C_3(\kappa,n) E_{L+M-n}[U_1] (1) e^{C_1(1-\tau^{-1})}  \tau^{2\delta + 2\kappa -n}
	\lesssim \eta^2 \tau^{2\delta + 2\kappa -n}
\]
for $1 \le \tau \le T$ and $0\le n \le\lfloor2\kappa\rfloor=M-1$.
Similarly,
\[
	E_{L}[U_1] (\tau) \le C_4(\kappa,n) E_{L}[U_1] (1) e^{C_1(1-\tau^{-1})}  \tau^{2\delta}\lesssim \eta^2 \tau^{2\delta}.
\]
By definition of $E_m$, this implies that
\begin{equation}
\|\pt_{\tau}^{k}\pt_{z}^{\ell}U_1(\tau,\cdot)\|_{L_z^{2}}
\lesssim_{M,L} \eta \tau^{\tilde{\delta}(\ell)} \label{ve2}
\end{equation}
for any $0\le k\le1$, $0\le \ell\le L+M$, and $1 \le \tau \le T$. Note that $M-1<2\kappa<M$.

Let $\alpha_1$ be as in \eqref{eq:2.30} and
$\tilde{\alpha}_1(\tau) := \exp( i\frac{3\lambda_1}{2 (\cosh \kappa z)^2} \int_1^\tau |\alpha_1|^2 ds )  \alpha_1(\tau)$.
The estimate \eqref{ve2} gives us
\[
	 \norm{ PU_1 }_{L^\I_z(\R)} \lesssim  \norm{U_1}_{H^2_z(\R)}^{\frac12}\norm{U_1}_{H^{3}_z(\R)}^{\frac12}
	\lesssim \eta \tau^{\delta+\frac14}.
\]
Remark that the right hand side can be $\eta \tau^\delta$ if $L\ge3$.
By the equation \eqref{eq:2.34}, we have
\[
	2i \pt_\tau \tilde{\alpha}_1 = \frac1\tau N_{1,\mathrm{nr}}(\alpha_1) e^{  i\frac{3\lambda_1}{2 (\cosh \kappa z)^2} \int_1^\tau |\alpha_1|^2 ds } +O(\eta\tau^{-\frac74+\delta})\IN L^\I_z(\R)
\]
for $1 \le \tau \le T$. By integration-by-parts and \eqref{ve2}, the first term of the right hand side is
estimated as
\[
	\norm{\int_1^\tau \frac1s N_{1,\mathrm{nr}}(\alpha_1) e^{  i\frac{3\lambda_1}{2 (\cosh \kappa z)^2} \int_1^s |\alpha_1|^2 d\sigma } ds}_{L^\I_z(\R)} 
	\lesssim \eta^3\tau^{-2+5\delta}.
\]
Combining these estimates, we have
\begin{equation}\label{ve4}
	C_0 \norm{U_1(\tau)}_{L^\I_z (\R)}^2 \le 4C_0 \norm{\tilde{\alpha}_1(\tau)}_{L^\I_z (\R)}^2 \le \delta
\end{equation}
for $1\le t \le T$ if $\eta_1$ is sufficiently small. 

In summary, we have proved in the previous two paragraphs that \eqref{ve3} implies \eqref{ve4} if $\eta_1 $ is sufficiently small.
By a standard argument, this implies that $U_1$ exists in $[1,\I)$ and
\eqref{ve2} and \eqref{ve4} are valid for any $\tau \ge 1$.
The estimate \eqref{ve} follows if we apply \eqref{ve2} to \eqref{eq:2.29}.
\end{proof}

\subsubsection{Estimate of the second component}

Let us turn to the estimate of the second component of \eqref{eq:2.29} corresponding to  \eqref{E:sysnew}.

\begin{proposition}
\label{P:ap2} Let $\kappa \in \R_+ \setminus \frac12 \Z$
and let $M=\lfloor2\kappa\rfloor+1$.
Let $L \ge 2$ be an integer.
Pick $\delta \in (0,1/5)$ and let $\eta_1>0$ be the number given in Proposition \ref{P:ap}.
Suppose $\eta:=\|U_{1,0}\|_{H_z^{L+M+1}}+\|U_{1,1}\|_{H_z^{L+M}}\le\eta_{1}$
and let $U_1(\tau)$ be the solution to the first equation of \eqref{eq:2.29}.
There exists a unique global solution $U_2$ to the second equation of \eqref{eq:2.29} 
with $U_{2}(0,z)=U_{2,0}(z), \pt_{\tau}U_{2}(0,z)=U_{2,1}(z)$ satisfying
$U_2\in C([\tau_0,\infty),H_z^{M+L+1}(\R))\cap C^1([\tau_0,\infty),H_z^{M+L}(\R))$ 
and estimates
\begin{equation}
\| \pt_{\tau}^{k_1} \pt_{z}^{k_2} U_2(\tau,\cdot)\|_{L^2_z(\R)}
\lesssim_\delta \tau^{2\delta+\tilde{\delta}(k_2)}  \label{ve5}
\end{equation}
for any $k_1,k_2\ge0$ with
$0\le k_1+k_2\le L+M+1$ and $k_2\le L+M$, and
\begin{equation*}
\| \pt_{z}^{M+L+1} U_2(\tau,\cdot)\|_{L^2_z(\R)}
\lesssim_\delta \tau^{2\delta+\tilde{\delta}(M+L+1)+\frac12} ,
\end{equation*}
 where
$\tilde{\delta}(k)$ is the same one as in Proposition \ref{P:ap}.
\end{proposition}

\begin{proof}
We let $\tau_0=1$ for simplicity.
Put
\[
	\mathcal{E}(\tau):= \tau^{-\delta} E_0[U_2](\tau) + 
	\sum_{m=1}^{M+L} \tau^{-4\delta -2\tilde{\delta}(m)} E_m[U_2](\tau),
\]
where $E_m$ is the same one as in the proof of Proposition \ref{P:ap}.
It suffices to show that $\mathcal{E}(\tau)$ is bounded.

By a standard energy argument, one has
\begin{equation}\label{E:Pap2pf1}
	\frac{\pt  E_0 [U_2]}{\pt \tau}(\tau)
	\le \frac{C_0}{\tau}  \norm{U_1(\tau)}_{L^\I_z(\R)}^2 E_0[U_2](\tau) 
	+ \frac{C_2}{\tau^2} E_0[U_2](\tau) + \frac{2\kappa }{\tau^{2}}  E_{1}[U_2](\tau)
\end{equation}
and
\begin{equation}\label{E:Pap2pf2}
\begin{aligned}
	\frac{\pt  E_m [U_2]}{\pt \tau}(\tau)
	\le{}& \frac{3C_0}{2\tau}  \norm{U_1(\tau)}_{L^\I_z(\R)}^2 E_m[U_2](\tau)
	+ \frac{C_1}{\tau} \norm{U_2(\tau)}_{L^\I_z(\R)}^2 E_m[U_1](\tau) \\
	& {}+ \frac{C_2}{\tau^2} E_m[U_2](\tau) + 2\kappa E_m[U_2](\tau)^\frac12 \min \( \frac1\tau  E_m[U_2](\tau)^\frac12, \frac1{\tau^{2}}  E_{m+1}[U_2](\tau)^\frac12\) 
\end{aligned}
\end{equation}
for $1 \le m \le M+L$, where $C_j=C_j(\kappa,L)$ are positive constants.
We can take $C_0\ge1$ as the same constant as in \eqref{ve4}. 
Here, we have used the following estimate for the term involving derivatives of the nonlinearity:
\begin{align*}
	\norm{\pt_z^j \pt_\tau U_2}_{L^2} \norm{\pt_z^j (U_1^2 U_2)}_{L^2}
	&{}\lesssim_j \norm{\pt_z^j \pt_\tau U_2}_{L^2} (\norm{\pt_z^j U_1}_{L^2} \norm{U_1}_{L^\I}\norm{U_2}_{L^\I}
		+ \norm{U_1}_{L^\I}^2\norm{\pt_z^j U_2}_{L^2} )\\
	&{} \le (\norm{U_1}_{L^\I}E_j[U_2]^{\frac12} )(\norm{U_2}_{L^\I} E_j[U_1]^{\frac12})
	+ \norm{U_1}_{L^\I}^2 E_j[U_2]\\
	&{} \le  \tfrac32 \norm{U_1}_{L^\I}^2 E_j[U_2] + \tfrac12 \norm{U_2}_{L^\I}^2 E_j[U_1]
\end{align*}
for $j \ge 1$.

Note that  $\tilde{\delta}(1)=\delta$.
Plugging \eqref{ve4} to \eqref{E:Pap2pf1}, we have
\begin{equation}\label{E:Pap2pf3}
	\pt_\tau ( e^{C_2(\tau^{-1}-1)} \tau^{-\delta} E_0[U_2](\tau)) \le 2\kappa \tau^{-2+5\delta} e^{C_2(\tau^{-1}-1)} (\tau^{-6\delta}E_{1}[U_2](\tau)) \le 2\kappa \tau^{-2+5\delta} e^{C_2(\tau^{-1}-1)}\mathcal{E}(\tau).
\end{equation}
Plugging \eqref{ve}, \eqref{ve4}, and the Sobolev embedding to \eqref{E:Pap2pf2}, one obtains
\begin{equation}\label{E:Pap2pf4}
\begin{aligned}
	\frac{\pt  E_m [U_2]}{\pt \tau}(\tau)	
	\le{}& \frac{3\delta}{2\tau}  E_m[U_2](\tau)
	  + \frac{C_2}{\tau^2}  E_m[U_2](\tau) + C_3 \eta_1  \tau^{2\tilde{\delta}(m) -1 } E_0[U_2](\tau)^{\frac12}E_1[U_2](\tau)^{\frac12}\\
	& {}+ 2\kappa  E_m[U_2](\tau)^\frac12 \min \( \frac1\tau   E_m[U_2](\tau)^\frac12, \frac1{\tau^{2}}   E_{m+1}[U_2](\tau)^\frac12 \)  .
\end{aligned}
\end{equation}
For $1 \le m \le M+L-1$, we deduce from $\tilde{\delta}(m)\ge \tilde{\delta}(m+1) - \frac12$
that
\begin{equation}\label{E:Pap2pf5}
\begin{aligned}
	&\partial_\tau (e^{C_2(\tau^{-1}-1)}\tau^{-4\delta-2\tilde{\delta}(m)} E_{m}[U_2](\tau) )\\
	&{} \le C_3 \eta_1 \tau^{-1-\frac\delta2}e^{C_2(\tau^{-1}-1)} (\tau^{-\delta}E_0[U_2](\tau))^\frac12 (\tau^{-6\delta}E_1[U_2](\tau))^\frac12 \\
	& \quad{} +2\kappa \tau^{-\frac32}e^{C_2(\tau^{-1}-1)}(\tau^{-4\delta -2\tilde{\delta}(m)} E_{m}[U_2](\tau))^\frac12 (\tau^{-4\delta -2\tilde{\delta}(m+1)} E_{m+1}[U_2](\tau))^\frac12 \\
	&{}\le 	(C_3 \eta_1 \tau^{-1-\frac{\delta}2} + 2\kappa \tau^{-\frac32})e^{C_2(\tau^{-1}-1)} \mathcal{E}(\tau).
\end{aligned}
\end{equation}
We next use \eqref{E:Pap2pf4} with the choice $m=M+L$ to obtain
\begin{equation}\label{E:Pap2pf6}
\begin{aligned}
	&{}\partial_\tau (e^{C_2(\tau^{-1}-1)}\tau^{-(2\delta+2\tilde{\delta}(M+L))} E_{M+L}[U_2](\tau) ) \\
	&{}\le C_3 \eta_1 \tau^{-1-\frac\delta2}e^{C_2(\tau^{-1}-1)} (\tau^{-\delta}E_0[U_2](\tau))^\frac12 (\tau^{-6\delta}E_1[U_2](\tau))^\frac12 \\
	&{}\le C_3 \eta_1 \tau^{-1-\frac\delta2}e^{C_2(\tau^{-1}-1)} \mathcal{E}(\tau).
\end{aligned}
\end{equation}

Combining \eqref{E:Pap2pf3}, \eqref{E:Pap2pf5}, and \eqref{E:Pap2pf6},
we have
\[
	\pt_\tau (e^{C_2(\tau^{-1}-1)} \mathcal{E}(\tau)) \lesssim_{\kappa,L}
	(\tau^{-2+5\delta}+\eta_1 \tau^{-1-\frac\delta2}+\tau^{-\frac32}) e^{C_2(\tau^{-1}-1)} \mathcal{E}(\tau).
\]
By Gronwall's inequality, we conclude that
	$\sup_{\tau\ge1}\mathcal{E}(\tau) \lesssim \mathcal{E}(1)$.
\end{proof}

\subsubsection{Bounds on $\alpha_j$}

We next give estimates on solutions to the system \eqref{eq:2.34} which are valid for general nonlinearity.
Remark that better estimates are possible for specific systems.
\begin{lemma}
Suppose that Assumption \ref{A:init} holds.
Let $(U_1,U_2)$ be a solution to \eqref{eq:2.29} given by Propositions \ref{P:ap} and \eqref{P:ap2} with the choice $\kappa=701/200$ and $L\ge2$.
Define $\alpha_j$ by \eqref{eq:2.30}.
Then, the following estimates are true for $\tau\ge \tau_0$ and $k,\ell \in \Z_{\ge0}$:
\begin{equation}\label{E:alest}
  \tnorm{\alpha_1(\tau)}_{C^{\ell}_z (\R)}+ \tnorm{\alpha_2(\tau)}_{C^{\ell}_z (\R)} 
\lesssim \tau^{3\delta +  \max (0, \frac14 + \frac12(l-L))}
\end{equation}
if $ \ell \le L+6$,
\begin{equation}\label{E:remainderest}
 \norm{\pt_\tau^{k} e^{-i\tau} \frac{1}{\tau^{2}}PU_1}_{C^{\ell}_z(\R)}
+ \norm{\pt_\tau^{k} e^{-i\tau} \frac{1}{\tau^{2}}PU_2}_{C^\ell_z(\R)}
\le C\tau^{-2+3\delta +  \max (0, \frac54 + \frac12(\ell-L)) }
\end{equation}
if $k+\ell \le L+5$,
and
\begin{equation}\label{E:altauest}
\tnorm{\pt_\tau^{1+k} \alpha_1(\tau)}_{C^{\ell}_z (\R)} 
+\tnorm{\pt_\tau^{1+k} \alpha_2(\tau)}_{C^{\ell}_z (\R)} 
\le C\tau^{-1+9\delta}
\end{equation}
if $k+\ell \le L+5$ and $\ell \le L-1$.
\end{lemma}
\begin{proof}
We have $M=\lfloor2\kappa\rfloor+1=8$.
The estimate \eqref{E:remainderest} is a consequence of \eqref{ve}, \eqref{ve5}, and the Sobolev embedding $\norm{f}_{C^\ell}
	\lesssim \norm{f}_{H^\ell}^\frac12  \norm{f}_{H^{\ell+1}}^\frac12 $. 
Remark that we have ``wasted'' one derivative in order to avoid the exceptional case $k_2=L+9$ of \eqref{ve}.

Similarly, for $k+\ell \le L+6$, 
\begin{equation}\label{E:alestprime}
  \tnorm{\pt_\tau^k \alpha_1(\tau)}_{C^{l}_z (\R)}+ \tnorm{\pt_\tau^k \alpha_2(\tau)}_{C^{l}_z (\R)} 
\lesssim \tau^{3\delta +  \max (0, \frac14 + \frac12(l-L))}
\end{equation}
follows from \eqref{eq:2.30}, \eqref{ve}, and \eqref{ve5}.
The estimate \eqref{E:alest} is a special case $k=0$ of \eqref{E:alestprime}.
Combining \eqref{E:alestprime} and \eqref{E:remainderest} and using the equation \eqref{eq:2.34}, we obtain \eqref{E:altauest}.
Remark that the leading order term is the nonlinear terms for small $\delta$.
\end{proof}

\subsection{Asymptotic behavior of $\alpha_1$ and $\alpha_2$}

We want to investigate large time behavior of solutions to \eqref{eq:2.34}.
In what follows, we always assume $\kappa=\frac{701}{200}$ and hence $M=8$.
In this subsection, we suppose $L\ge3$.

\begin{lemma}\label{L:alsingle}
There exist a function $\tilde{\alpha}_1\in C^{L-1}\cap W^{L-1,\I}(\R)$ and a positive constant $c$ such that 
\begin{eqnarray}
\alpha_1(\tau,z)=\tilde{\alpha}_1(z)
\exp\left(-i\frac{3\lambda_1}{2(\cosh \kappa z)^2}|
\tilde{\alpha}_1(z)|^2\log\tau\right)
+O\left(\tau^{-1+c\delta}\right)\label{E:thm4a1}
\end{eqnarray}
holds in $C^{L-3}_z(\R)$ as $\tau \to \infty$.
The same asymptotics holds in $C_z^{L-1}(\R)$ with the remainder $O(\tau^{-\frac14+c\delta})$ instead.
In particular, it holds that
\begin{equation}\label{E:alesta1}
	\norm{\pt_z^\ell \alpha_1(\tau)}_{L^\I_z (\R)} \lesssim (\log \tau)^\ell
\end{equation}
for $\tau\ge \tau_0$ and $0\le \ell \le L-1$. 
\end{lemma}

\begin{remark}
 \eqref{E:alesta1} is a refinement of
\eqref{E:alest} with respect to the growth order in $\tau$.
\end{remark}

\begin{proof}
The first equation of the system \eqref{eq:2.34} takes the form 
\begin{equation}
2i\pt_{\tau}\alpha_1(\tau,z)
=
\frac{3\lambda_1}{\tau(\cosh \kappa z)^2}
|\alpha_1(\tau,z)|^2\alpha_1(\tau,z)+\frac1\tau N_{1,\text{nr}}(\alpha_1)+O(\tau^{-\frac54+3\delta}),\label{a11}
\end{equation}
in $C_z^{L-1}(\R)$ with
\[
N_{1,\text{nr}}(\alpha_1)
=\frac{\lambda_1}{(\cosh \kappa z)^2}
\left\{\alpha_1(\tau,z)^3e^{2i\tau}+3|\alpha_1(\tau,z)|^2\overline{\alpha_1(\tau,z)}e^{-2i\tau}
+\overline{\alpha_1(\tau,z)}^3e^{-4i\tau}\right\},
\]
where we have used \eqref{E:remainderest} with $(k,l)=(0,L-1)$.
The equation \eqref{a11} reads as
\[
	2i\pt_{\tau} \( e^{i\frac{3\l_1}{2(\cosh \kappa z)^2} \int_{\tau_0}^\tau |\alpha_1(s,z)|^2 \frac{ds}s}  \alpha_1(\tau,z)\)
	= \frac1\tau N_{1,\text{nr}}(\alpha_1) e^{i\frac{3\l_1}{2(\cosh \kappa z)^2} \int_{\tau_0}^\tau |\alpha_1(s,z)|^2 \frac{ds}s} +O(\tau^{-\frac54+c\delta})
\]
in $C^{L-1}_z(\R)$ with $c=c(L)\ge1$.
By an integration by parts and using the estimates \eqref{E:alest} and \eqref{E:altauest}, we see
that 
\[
	\norm{ \int_\tau^\I \frac1s N_{1,\text{nr}}(\alpha_1(s)) e^{i\frac{3\l_1}{2(\cosh \kappa z)^2} \int_{\tau_0}^s |\alpha_1(\sigma)|^2 \frac{d\sigma}\sigma} ds}_{C^{L-1}_z(\R)} \lesssim \tau^{-1+c\delta}
\]
for some constant $c>0$.
Hence, if $\delta$ is chosen small enough then 
$e^{i\frac{3\l_1}{2} \int_{\tau_0}^\tau |\alpha_1(s,z)|^2 \frac{ds}s}  \alpha_1(\tau,z)$ is a Cauchy sequence in $C_z^{L-1}(\R)$
and so there exists a function $A \in C_z^{L-1} \cap W^{L-1,\I}_z (\R)$ such that 
\[
	e^{i\frac{3\l_1}{2(\cosh \kappa z)^2} \int_{\tau_0}^\tau |\alpha_1(s)|^2 \frac{ds}s}  \alpha_1(\tau) = A + O(\tau^{-\frac14+c\delta})
\]
in $C^{L-1}_z(\R)$ as $\tau\to\I$. 
This implies
\[
	|\alpha_1(\tau)|^2 = |A|^2 + O(\tau^{-\frac14+c\delta})
\]
in $C_z^{L-1}(\R)$
and hence
there exists a real-valued function $S(z) \in C^{L-1}_z \cap W^{L-1,\I}_z (\R)$ such that
\[
	\int_{\tau_0}^\tau |\alpha_1(s,z)|^2 \frac{ds}s = |A|^2 \log \tau + S(z) + O(\tau^{-\frac14+c\delta})
\]
in $C^{L-1}_z(\R)$.
Thus, \eqref{E:alesta1} holds by letting $\tilde{\alpha}_1(z) := e^{-i\frac{3\l_1}{2(\cosh \kappa z)^2}S(z)}A(z) \in C^{L-1}_z \cap W^{L-1,\I}_z (\R)$.
Remark that $|A|=|\tilde{\alpha}_1|$ holds.

Repeating the above argument with the estimate \eqref{E:remainderest} with a different choice of $l$, we have
\[
	e^{i\frac{3\l_1}{2(\cosh \kappa z)^2} \int_{\tau_0}^\tau |\alpha_1(s)|^2 \frac{ds}s}  \alpha_1(\tau) = A + O(\tau^{-1+c\delta})
\]
and
\[
	\int_{\tau_0}^\tau |\alpha_1(s,z)|^2 \frac{ds}s = |A|^2 \log \tau + S(z) + O(\tau^{-1+c\delta})
\]
in $C^{L-3}_z(\R)$ as $\tau\to\I$, from which
\eqref{E:thm4a1} follows. 
Further, \eqref{E:thm4a1} implies \eqref{E:alesta1}.
\end{proof}

Next, we consider the behavior of the second component.

\begin{lemma}
There exist a function $\tilde{\alpha}_2\in C^{L-1}_z \cap W^{L-1,\I}_z (\R)$ and a constant $c'>0$ such that 
\begin{equation}\label{E:thm4a2a}
\begin{aligned}
\alpha_2(\tau,z)
=&\left\{-i\frac{3\lambda_1}{2(\cosh \kappa z)^2}
(|\tilde\alpha_1(z)|^2\tilde{\alpha}_2(z)
+\tilde{\alpha}_1(z)^2\overline{\tilde{\alpha}}_2(z))\log\tau+\tilde{\alpha}_2(z)\right\}\\
&\quad \times
\exp\left(-i\frac{3\lambda_1}{2(\cosh \kappa z)^2}|
\tilde{\alpha}_1(z)|^2\log\tau\right)
+O\left(\tau^{-1+c'\delta}\right).
\end{aligned}
\end{equation}
holds if $\l_6 = 3\l_1$ and
\begin{equation}\label{E:thm4a2b}
\begin{aligned}
\alpha_2(\tau,z)
=&\left\{i\frac{\lambda_1}{2(\cosh \kappa z)^2}
(|\tilde\alpha_1(z)|^2\tilde{\alpha}_2(z)
-\tilde{\alpha}_1(z)^2\overline{\tilde{\alpha}}_2(z))\log\tau+\tilde{\alpha}_2(z)\right\}\\
 &\qquad\times
\exp\left(-i\frac{3\lambda_1}{2(\cosh \kappa z)^2}|
\tilde{\alpha}_1(z)|^2\log\tau\right)
+O\left(\tau^{-1+c'\delta}\right).
\end{aligned}
\end{equation}
holds if $\l_6= \l_1$
in $C^{L-3}_z(\R)$ for $\tau \ge \tau_0$.
The same asymptotics holds in $C_z^{L-1}(\R)$ with the remainder $O(\tau^{-\frac14+c'\delta})$ instead.
In particular,
\[
	\norm{\pt_z^\ell \alpha_2(\tau)}_{L^\I_z (\R)} \lesssim (\log \tau)^{\ell+1}
\]
for $\tau\ge \tau_0$ and $0\le \ell \le L-1$.
\end{lemma}
\begin{remark}\label{R:Thm4conservation}
When $\l_6=3\l_1$, $\ltrans{(0,1,0)}$ is a solution to \eqref{E:Aeq}. 
Hence, $\Re (\ol{\alpha_1} \alpha_2)$ is a conserved quantity for the corresponding limit ODE system.
Similarly, the condition \eqref{E:main3cond} is satisfied when $\l_6=\l_1$.
This implies that $\Im (\ol{\alpha_1} \alpha_2)$ is a conserved quantity for the ODE system in this case.
It will turn out that working with these two quantities is actually crucial in the proof.
\end{remark}

\begin{proof}
First consider the case $\l_6=3\l_1$.
In the proof, we denote various constants depend on $L$ by $c>0$.
The first equation of \eqref{eq:2.34} becomes \eqref{a11} and the second equation becomes
\begin{equation}\label{a21}
\begin{aligned}
2i\pt_{\tau}\alpha_2(\tau,z)
=&
\frac{3\lambda_1}{\tau(\cosh \kappa z)^2}
\left\{
\alpha_1(\tau,z)^2\overline{\alpha_2(\tau,z)}
+2|\alpha_1(\tau,z)|^2\alpha_2(\tau,z)\right\}\\
 &+\frac1\tau N_{2,\text{nr}}(\alpha_1,\alpha_2)+O(\tau^{-\frac54+3\delta})
\end{aligned}
\end{equation}
in $C^{L-1}_z(\R)$
with
\begin{align*}
N_{2,\text{nr}}(\alpha_1,\alpha_2)
=
\frac{3\lambda_1}{(\cosh \kappa z)^2}
(& \alpha_1(\tau,z)^2\alpha_2(\tau,z)e^{2i\tau}
+\overline{\alpha_1(\tau,z)}^2\alpha_2(\tau,z)e^{-2i\tau} \\
&+\overline{\alpha_1(\tau,z)}^2\overline{\alpha_2(\tau,z)}e^{-4i\tau}),
\end{align*}
respectively, where we have used \eqref{E:remainderest}.
We have
\[
\pt_{\tau}(2 \Re(\overline{\alpha}_1\alpha_2))
=
\frac1\tau \Im(N_{1,\text{nr}}(\alpha_1)\overline{\alpha}_2)
+\frac1\tau \Im(N_{2,\text{nr}}(\alpha_1,\alpha_2)\overline{\alpha}_1) + O(\tau^{-\frac54+6\delta})
\]
in $C^{L-1}_z(\R)$.
By integrating by parts and using \eqref{E:alesta1} and \eqref{E:altauest}, one has
\[
	\norm{ \int_\tau^\I \frac1s N_{1,\text{nr}}(\alpha_1(s))\ol{\alpha_2(s)}  ds}_{C^{L-1}_z(\R)}
	+ \norm{ \int_\tau^\I \frac1s N_{2,\text{nr}}(\alpha_1(s),\alpha_2(s))\ol{\alpha_1(s)}  ds}_{C^{L-1}_z(\R)}\lesssim \tau^{-1+c\delta}.
\]
Hence, there exists a real valued function 
$\tilde{\beta}\in C_z^{L-1}\cap W^{L-1,\I}_z(\R)$ such that 
\begin{equation}
2 \Re(\overline{\alpha}_1(\tau,z)\alpha_2(\tau,z))
=\tilde{\beta}(z)+O\left(\tau^{-\frac14+c\delta}\right)\label{a32}
\end{equation}
in $C^{L-1}_z(\R)$ as $\tau\to\I$.
Substituting (\ref{E:thm4a1}) and (\ref{a32}) into (\ref{a21}), we obtain 
\begin{align*}
&2i\pt_{\tau}\alpha_2(\tau,z)
-\frac{3\lambda_1}{\tau(\cosh \kappa z)^2}
|\tilde{\alpha}_1(z)|^2\alpha_2(\tau,z)\\
&{}=\frac{3\lambda_1}{\tau(\cosh \kappa z)^2}
\tilde{\alpha}_1(z)\tilde{\beta}(z)
\exp\left(-i\frac{3\lambda_1}{2(\cosh \kappa z)^2}|
\tilde{\alpha}_1(z)|^2\log\tau\right)\\
&\quad +\frac1\tau N_{2,\text{nr}}(\alpha_1,\alpha_2)+O\left(\tau^{-\frac54+c\delta}\right),
\end{align*}
in $C_z^{L-1}(\R)$ as $\tau\to\I$ with a suitable constant $c>0$, which reads as
\begin{align*}
&2i\pt_{\tau}
\left\{\alpha_2(\tau,z)
\exp\left(i\frac{3\lambda_1}{2(\cosh \kappa z)^2}|
\tilde{\alpha}_1(z)|^2\log\tau\right)+ i\frac{3\lambda_1}{2(\cosh \kappa z)^2}
\tilde{\alpha}_1(z)\tilde{\beta}(z)\log\tau \right\}\\
&=
\frac1\tau N_{2,\text{nr}}(\alpha_1,\alpha_2)
\exp\left(i\frac{3\lambda_1}{2(\cosh \kappa z)^2}|
\tilde{\alpha}_1(z)|^2\log\tau\right)
 +O\left(\tau^{-\frac54+c\delta}\right).
\end{align*}
By integration by parts, one has
\[
	\norm{ \int_\tau^\infty \frac1s N_{2,\text{nr}}(\alpha_1(s),\alpha_2(s))
\exp\left(i\frac{3\lambda_1}{2(\cosh \kappa z)^2}|
\tilde{\alpha}_1(z)|^2\log s \right)ds }_{C^{L-1}_z(\R)}
\lesssim \tau^{-1+c\delta}.
\]
Hence, there exists a function 
$\tilde{\alpha}_2\in C^{L-1}_z\cap W^{L-1,\I}_z(\R)$ such that 
\begin{eqnarray*}
\lefteqn{\alpha_2(\tau,z)\exp\left(i\frac{3\lambda_1}{2(\cosh \kappa z)^2}|
\tilde{\alpha}_1(z)|^2\log\tau\right)}\\
&=&\tilde{\alpha}_2(z)-i\frac{3\lambda_1}{2(\cosh \kappa z)^2}
\tilde{\alpha}_1(z)\tilde{\beta}(z)\log\tau
+O\left(\tau^{-\frac14+c\delta}\right)
\end{eqnarray*}
in $C^{L-1}_z (\R)$.
We have
\begin{align*}
\alpha_2(\tau,z)
={}&
\left\{-i\frac{3\lambda_1}{2(\cosh \kappa z)^2}
\tilde{\alpha}_1(z)\tilde{\beta}(z)\log\tau+\tilde{\alpha}_2(z)\right\}
\exp\left(-i\frac{3\lambda_1}{2(\cosh \kappa z)^2}|
\tilde{\alpha}_1(z)|^2\log\tau\right)\\
&+O\left(\tau^{-\frac14+c\delta}\right)
\end{align*}
in $C^{L-1}_z(\R)$ as $\tau\to\infty$.
Hence, the conclusion follows if we show the identity
\[
	\tilde{\beta}(z) = 2 \Re (\ol{\tilde{\alpha}_1(z)}\tilde{\alpha}_2(z)).
\]
One sees from the above asymptotics of $\alpha_2(\tau,z)$ and \eqref{E:thm4a1} that
\[
	\ol{\alpha_1(\tau,z)}\alpha_2(\tau,z)
	= -i\frac{3\lambda_1}{2(\cosh \kappa z)^2}
|\tilde{\alpha}_1(z)|^2 \tilde{\beta}(z)\log\tau+\ol{\tilde{\alpha}_1(z)}\tilde{\alpha}_2(z)
+O\left(\tau^{-\frac14+c\delta}\right).
\]
Comparing the real part with \eqref{a32}, one obtains the desired identity.
If we consider the asymptotics in $C_z^{L-3}(\R)$ then we gain $\tau^{-3/4}$ in the decay rate of the remainder term by the same argument.

The other case $\l_6=\l_1$ is handled in a similar way.
We have
\[
\pt_{\tau}(2 \Im(\overline{\alpha}_1\alpha_2))
=
\frac1\tau \Re(N_{1,\text{nr}}(\alpha_1)\overline{\alpha}_2)
-\frac1\tau \Re(N_{2,\text{nr}}(\alpha_1,\alpha_2)\overline{\alpha}_1) + O(\tau^{-\frac54+3\delta})
\]
and hence there exists a real-valued function $\tilde{\beta} \in C^{L-1}\cap W^{L-1,\I}(\R)$ such that
\[
2 \Im(\overline{\alpha}_1(\tau,z)\alpha_2(\tau,z))
=\tilde{\beta}(z)+O\left(\tau^{-\frac14+c\delta}\right).
\]
We use it instead of \eqref{a32}.
The remainder is order $O(\tau^{-1+C\delta})$ in $C^{L-3}_z(\R)$.
\end{proof}

\subsection{Asymptotic behavior of $u_1$ and $u_2$}

So far, we have obtained asymptotic behavior for $\alpha_1$ and $\alpha_2$ as $\tau\to\I$.
The final step is translate the asymptotics into those for $u_1$ and $u_2$.

Since we want to use the (real) variable $\sqrt{1-|\frac{x}t|^2}$ to express the asymptotic behavior of $(u_1,u_2)$,
we first consider the outside of the light cone, i.e., $\{(t,x) \in \R^2 \ |\  |x| > t \ge0 \}$, in which the variable is not well-defined.
In this region (and also at the boundary), the solution itself is small.

\begin{lemma}[Smallness outside of the light cone]\label{L:thm4outsidedecay}
It holds that
\[
	\sum_{j=1}^2 \norm{u_j(t,\cdot)}_{L^\I_x (\{|x|\ge t\})} = o( t^{-2} )
\]
as $t\to \I$. 
\end{lemma}
\begin{proof}
Recall that the support of $u_j(t)$ is contained in $\{ |x| \le t + B \}$ for any $t\ge 0$.
One verifies that there exists $t_1 = t_1(B)$ such that if $t\ge t_1$ then the coordinate $(\tau(t,x),z(t,x))$ defined by
\eqref{E:hypcoord} satisfies
\[
	\tau \sim \cosh z \sim \sqrt{t}
\]
for any $t \le  |x| \le t+B$,
where the implicit constants depends only on $B$.
Hence, if $t\ge t_1$ then
\[
	\sum_{j=1}^2 \norm{u_j(t,\cdot)}_{L^\I (|x|\ge t)}
	\le \sup_{\tau \sim \sqrt{t},\, \cosh z \sim \sqrt{t} } \frac1{\tau^{1/2} \cosh (\kappa z)} \sum_{j=1}^2|\alpha_j(\tau,z)|
	\lesssim t^{-\frac{1+2\kappa}{4}}\log t
\]
follows from \eqref{E:thm4a1}, \eqref{E:thm4a2a}, and \eqref{E:thm4a2b}.
This is $o(t^{-2})$ since $\kappa>7/2$.
\end{proof}

Now, let us consider the behavior of inside of the light cone $\{|x|< t\}$.
The following is useful:
\begin{lemma}[\cite{Del}]
For $t> 0$ and $|x|< t$, 
\begin{align}
\left|\frac{1}{(t+2B)^{1/2}}-\frac1{t^{1/2}}\right|
&\le Ct^{-\frac32},\label{ab31}\\
\left|\sqrt{1-\left|\tfrac{x}{t+2B}\right|^{2}}-\sqrt{1-\left|\tfrac{x}{t}\right|^{2}}
-\frac{2B\left(\frac{x}{t}\right)^{2}}{t\sqrt{1-\left|\frac{x}{t}\right|^{2}}}\right|
&\le Ct^{-2}\left(1-\left|\tfrac{x}{t}\right|^{2}\right)^{-\frac32}.\label{ab41}
\end{align}
\end{lemma}

For $j=1,2$, define
\[
	A_j(y) := \(\(\frac{(\cosh z)^{1/2}}{\cosh (\kappa z)} \tilde{\alpha}_j(z)\) \circ \tanh^{-1} \) (y)
\]
for $-1<y<1$ and $A_j(y)=0$ for $|y|\ge 1$. Remark that $ A_j \in C^{L-1} \cap W^{L-1,\I}(\R)$ and
\begin{equation}\label{E:Adecay}
	\pt_y^k A_j(y) = O((1-|y|)^{\frac{\kappa}2-\frac14})
\end{equation}
as $|y| \to 1 $ for $0 \le k \le L-1$.

\begin{proof}[Proof of Theorem \ref{T:main4}]
Remark that $L=6$.
We see from \eqref{E:thm4a1} that
\begin{equation}
	u_1(t,x) = \frac{2}{t^{1/2}} \Re ( \tilde{\Phi}_1(y) \exp (it\sqrt{1-|y|^2}+i\Psi(y) \log t) )|_{y=\frac{x}{t+2B}}
	+ O(t^{-\frac32+c\delta})
\end{equation}
in $L^\I_x(\R)$ as $t\to\I$, where we have put
\[
	\tilde{\Phi}_1(y) := A_1(y) \exp \(2iB\sqrt{1-|y|^2} +i\Psi(y) \log \sqrt{1-|y|^2}\)
	\in C^2_y\cap W^{2,\I}_y(\R)
\]
and
\[
	\Psi(y) := -\frac{3\l_1}2\sqrt{1-|y|^2} |A_1(y)|^2 \in C^3_y\cap W^{3,\I}_y(\R)
\]
and used \eqref{ab31} and $\log \frac{t+2B}t = O(t^{-1})$.
Remark that the regularity  of $\tilde{\Phi}_1$ and $\Psi$ are determined by the choice of $\kappa$ (and $L$)
via \eqref{E:Adecay}.
One deduces from 
\begin{align}\label{E:Psi1decay}
	\tilde{\Phi}_1(y) ={}& O((1-|y|)^{\frac{\kappa}2-\frac14}),&
	\pt_y \tilde{\Phi}_1(y) ={}& O((1-|y|)^{\frac{\kappa}2-\frac34}), 
\end{align}
as $|y|\uparrow1$ that
\[
	\tilde{\Phi}_1(y) \exp\( i \Psi(y) \log t  \)|_{y=\frac{x}{t+2B}} 
	= \tilde{\Phi}_1(\tfrac{x}t) \exp ( i\Psi(\tfrac{x}t) \log t  ) + O(t^{-1}\log t)
\]
in $L^\I(|x|\le t)$ as $t\to\I$ provided $\kappa>5/2$ (see e.g. \cite{Del}). 
Further, \eqref{ab41} and \eqref{E:Psi1decay} yield
\begin{multline*}
	\tilde{\Phi}_1(\tfrac{x}t) \exp \( it\sqrt{1-|\tfrac{x}{t+2B}|^2} + i\Psi(\tfrac{x}t) \log t  \)\\
	= e^{i\frac{2B(\frac{x}t)^2}{(1-(\frac{x}t)^2)^{1/2}}}\tilde{\Phi}_1(\tfrac{x}t) \exp \( it\sqrt{1-|\tfrac{x}{t}|^2} + i\Psi(\tfrac{x}t) \log t  \)
	+ O(t^{-1})
\end{multline*}
as long as $\kappa>7/2$.
Setting
\[
	\Phi_1(y) := \tilde{\Phi}_1(y) \exp\( i\tfrac{2By^2}{(1-y^2)^{1/2}} \) \in (C^1_y \cap W^{1,\I}_y)(\R) \cap C_y^{5} (\R \setminus \{\pm1\})
\]
and combining the above estimates, we obtain the asymptotics for the first component $u_1$. Remark that
$|\Phi_1|^2 = |A_1|^2 \in C^5 \cap W^{5,\I}(\R)$.

Next, we consider the behavior of the second component.
When $\l_6=3\l_1$,
\eqref{E:thm4a1} and \eqref{E:thm4a2a} give us
\begin{equation*}
	\begin{aligned}
	&u_2(t,x) \\
	&{}= \frac{2}{t^{1/2}}\left. \Re \( \( \tilde{\Psi}(y) ( \log t +  \log \sqrt{1-|y|^2}) +\tilde{\Phi}_2 (y) \)	 \exp\(it\sqrt{1-|y|^2} +i\Psi(y) \log t \)  \)\right|_{y=\frac{x}{t+2B}}\\
	& \quad + O(t^{-\frac32+c\delta})
	\end{aligned}
\end{equation*}
in $L^\I_x(\R)$ as $t\to\I$, where we have put
\[
	\tilde{\Phi}_2(y) := A_2(y) \exp \(2iB\sqrt{1-|y|^2} +i\Psi(y) \log \sqrt{1-|y|^2}\)
	\in C^2_y\cap W^{2,\I}_y(\R)
\]
and
\begin{align*}
	\tilde{\Psi} (y) :={}& - \frac{3 \l_1}2 i \sqrt{1-|y|^2}
	(|\tilde{\Phi}_1(y)|^2\tilde{\Phi}_2(y) + \tilde{\Phi}_1(y)^2 \ol{\tilde{\Phi}_2(y)})\\
	={}& - \frac{3 \l_1}2 i \sqrt{1-|y|^2} (|A_1(y)|^2 A_2(y) + A_1(y)^2 \ol{A_2(y)})\\
	& \times  \exp \(2iB\sqrt{1-|y|^2} +i\Psi(y) \log \sqrt{1-|y|^2}\) 
	\in C^5_y\cap W^{5,\I}_y(\R)
\end{align*}
and used \eqref{ab31} and $\log \frac{t+2B}t = O(t^{-1})$.
Arguing as above, one obtains the desired estimate.

Similarly, when $\l_6=\l_1$,
\eqref{E:thm4a1} and \eqref{E:thm4a2b} give us the same estimate with
\[
	\tilde{\Psi} (y) := \frac{ \l_1}2 i \sqrt{1-|y|^2}
	(|\tilde{\Phi}_1(y)|^2\tilde{\Phi}_2(y) - \tilde{\Phi}_1(y)^2 \ol{\tilde{\Phi}_2(y)})
	\in C^2_y\cap W^{2,\I}_y(\R).\qedhere
\]
\end{proof}

\begin{remark}\label{R:lac}
The behavior of a solution is similar to that for \eqref{E:sysnewa} or \eqref{E:sysnewb}. 
Especially, the appearance of a logarithmic amplitude correction is common.
However, the mechanism of arising the amplitude correction is slightly different.
As for the present case, 
the limit ODE system \eqref{E:limitODE} becomes
\[
	\left\{
	\begin{aligned}
	&2i \pt_s \alpha_1 = 0, \\
	&2i \pt_s \alpha_2 = \l_5 |\alpha_1|^2 \alpha_1.
	\end{aligned}
	\right.
\]
Recall the correspondence $s=\log \sqrt{t^2-|x|^2} (= \log t + \log \sqrt{1-|x/t|^2})$.
Hence, the linear growth in $s$ of $\alpha_2$ corresponds to the appearance of logarithmic amplitude correction for $\alpha_2$.
Note that $|\alpha_1|^2 \alpha_1$ is constant in $s$ by the first equation.
Hence, the presence of a constant external term in the equation for $\alpha_2$
is the source of logarithmic amplitude correction.
On the other hand, as for the system \eqref{E:sysnewa} or \eqref{E:sysnewb}, the corresponding system \eqref{E:limitODE} takes the form
\[
	\left\{
	\begin{aligned}
	&2i \pt_s \alpha_1 = \lambda |\alpha_1|^2 \alpha_1 , \\
	&2i \pt_s \alpha_2 = \lambda |\alpha_1|^2 \alpha_2 + \l \beta \alpha_1  .
	\end{aligned}
	\right.
\]
Here, $\beta \in \R \cup i \R$ is a constant given by the corresponding conserved quantity of the system.
The both equations have the same factor $\lambda |\alpha_1|^2$, which produces a logarithmic phase modification.
The agreement of the phase modification is the heart of matter. Because of this, the term $\lambda \beta \alpha_1(s) $, which has the same oscillation as $\alpha_1(s)$, produces the linear growth in $s$.
The growth is nothing but the appearance of a logarithmic amplitude correction.
It can be phrased that the appearance in this case is due to a resonance phenomenon in the ODE theory.
\end{remark}

\section{Proof of Theorem \ref{T:main6}}\label{S:main6}

This section is devoted to the proof of Theorem \ref{T:main6}.

\begin{proof}
We first consider \eqref{E:newODE1}.
Denote two conserved quantities as
\[
	\beta_1:=|\alpha_1(s)|^2-|\alpha_2(s)|^2
\qquad \text{and}
\qquad
	\beta_2:=2 \Re(\overline{\alpha}_1(s)\alpha_2(s)).
\]
Plugging this to the ODE, one obtains
\[
\left\{
\begin{aligned}
	\frac{\pt \alpha_1}{\pt s} &{}= -\frac{i}{2} (3\beta_1 \alpha_1 -3 \beta_2 \alpha_2), \\
	\frac{\pt \alpha_2}{\pt s} &{}= -\frac{i}{2} (3\beta_1 \alpha_2 +3 \beta_2 \alpha_1).
\end{aligned}
\right.
\]
We define
\[
	\tilde{\alpha}_j (s) = \alpha_j(s) \exp \( \frac{3i}{2 } \beta_1 s \)
\]
for $j=1,2$. Then,
\[
	\frac{\pt }{\pt s}\tilde{\alpha}_1 = \frac{3i\beta_2}{2} \tilde{\alpha}_2,\quad
	\frac{\pt }{\pt s}\tilde{\alpha}_2 = -\frac{3i\beta_2}{2} \tilde{\alpha}_1.
\]
This reads as
\[
	\frac{\pt }{\pt s}\(\frac{\tilde{\alpha}_1\pm i \tilde{\alpha}_2}2\) = \pm \frac{3\beta_2}2  \(\frac{\tilde{\alpha}_1\pm i \tilde{\alpha}_2}2\).
\]
Hence,
\[
	\alpha_1 (s) = c_+ \exp \( \frac{-3i \beta_1 +3\beta_2}{2 }  s \)
	+ c_- \exp \( \frac{-3i \beta_1 -3\beta_2}{2 }  s \)
\]
and
\[
	\alpha_2 (s) = -i c_+ \exp \( \frac{-3i \beta_1 +3\beta_2}{2 }  s \)
	+i c_- \exp \( \frac{-3i \beta_1 -3\beta_2}{2 }  s \).
\]
One sees from this formula that
\[
	\beta_1 = |\alpha_1(s)|^2-|\alpha_2(s)|^2 = 4 \Re (c_+\ol{c_-})
\]
and
\[
	\beta_2 = 2 \Re (\ol{\alpha_1(s)}\alpha_2(s)) = 4 \Im (c_+\ol{c_-}).
\]
This shows$\beta_1 + i \beta_2 = 4 c_+ \ol{c_-}$ and $\beta_1 - i \beta_2 = 4\ol{c_+} c_-$.
Plugging this formulas to the above solution formula, we obtain the desired result.

Let us turn to \eqref{E:newODE2}. Denote two conserved quantities as
\[
	\beta_1:=|\alpha_1(s)|^2-|\alpha_2(s)|^2
\qquad \text{and}
\qquad
	\beta_2:=2 \Im(\overline{\alpha}_1(s)\alpha_2(s)).
\]
Plugging this to the ODE, one obtains
\[
\left\{
\begin{aligned}
	\frac{\pt \alpha_1}{\pt s} &{}= -\frac{i}{2} (3\beta_1 \alpha_1 -i \beta_2 \alpha_2), \\
	\frac{\pt \alpha_2}{\pt s} &{}= -\frac{i}{2} (3\beta_1 \alpha_2 -i \beta_2 \alpha_1).
\end{aligned}
\right.
\]
We define
\[
	\tilde{\alpha}_j (s) = \alpha_j(s) \exp \( \frac{3i}{2 } \beta_1 s \)
\]
for $j=1,2$. Then,
\[
	\frac{\pt }{\pt s}(\tilde{\alpha}_1\pm \tilde{\alpha}_2) = \mp \frac{\beta_2}{2} (\tilde{\alpha}_1\pm \tilde{\alpha}_2).
\]
This means
\[
	\frac{\pt \alpha_{\pm}}{\pt s}
	= 0,
\]
where
\[
	\alpha_{\pm} (s) = \frac12(\tilde{\alpha}_1\pm \tilde{\alpha}_2)(s)
	\exp \(\pm \frac{\beta_2}{2} s \).
\]
Hence, 
\[
	\alpha_{\pm} (s) = \alpha_{\pm} (0)=c_\pm'.
\]
Therefore, we conclude that
\begin{align*}
	\alpha_1(s) ={}& c_+' \exp \(- \frac{3i\beta_1+\beta_2}{2}   s \)
	 + c_-' \exp \( -\frac{3i\beta_1-\beta_2}{2} s  \)
\end{align*}
and
\begin{align*}
	\alpha_2(s) ={}& c_+' \exp \(- \frac{3i\beta_1(z)+\beta_2(z)}{2}  s  \)
	 - c_-' \exp \( -\frac{3i\beta_1(z)-\beta_2(z)}{2} s \)
.
\end{align*}
Notice that these two imply
\[
	 \beta_1 = |\alpha_1(s)|^2 - |\alpha_2(s)|^2 = 4 \Re (c_+' \ol{ c_-'}) 
\]
and
\[
	\beta_2 = 2 \Im (\ol{\alpha_1(s)} \alpha_2(s)) = 4\Im (c_+' \ol{c_-'}) .
\]
Hence,
\[
	3\beta_1 - i \beta_2 =  4(c_+' \ol{c_-'} + 2 \ol{c_+'} c_-) 
	\quad \text{and} \quad
	3\beta_1 + i \beta_2 = 4(\ol{c_+'} c_- + 2c_+ \ol{c_-}),
\]
which complete the proof.
\end{proof}

\subsection*{Acknowledgements} 
The authors express their gratitude to Professors Kenji Nakanishi and Hideaki Sunagawa for valuable
comments on a preliminary version of the manuscript.
S.M. was supported by JSPS KAKENHI Grant Numbers JP17K14219, JP17H02854, JP17H02851, and JP18KK0386. J.S was supported by JSPS KAKENHI Grant Numbers 
JP17H02851, JP19H05597, and JP20H00118.  K.U was supported by JSPS KAKENHI Grant Number 
JP19K14578.

\bibliographystyle{abbrv}

\bibliography{NLKG}

\end{document}